\documentclass[preprint]{imsart}

\RequirePackage{amsthm,amsmath,amsfonts,amssymb,mathtools}
\usepackage{dsfont}
\usepackage{footmisc}
\usepackage[authoryear,round]{natbib}
\usepackage[colorlinks=true,linkcolor=blue,citecolor=blue]{hyperref}
\usepackage{hypernat}
\usepackage{verbatim}
\usepackage{textcomp}
\usepackage{enumerate}
\usepackage{graphicx}
\usepackage{bbm}
\usepackage{array}
\usepackage{multirow}
\usepackage{tikz}
\usetikzlibrary{calc}

\newcolumntype{H}{>{\setbox0=\hbox\bgroup}c<{\egroup}@{}}

\newcommand{\R}{{\mathbb R}}

\newcommand{\E}{{\mathbb E}}
\renewcommand{\P}{\mathbb{P}}
\newcommand{\N}{{\mathbb N}}

\newcommand{\eps}{\varepsilon}

\newcommand{\X}{\mathcal X}

\usepackage{xr}

\newcommand\norm[1]{\left\lVert#1\right\rVert}

\newtheorem{theorem}{Theorem}[section]
\newtheorem{proposition}[theorem]{Proposition}
\newtheorem{lemma}[theorem]{Lemma}

\newtheorem{remark}[theorem]{Remark}
\newtheorem*{remark*}{Remark}

\newtheorem*{condition*}{Condition}
\newtheorem*{definition*}{Definition}
\newtheorem{definition}{Definition}
\newtheorem{assumption}{Assumption}

\numberwithin{equation}{section}

\newcounter{rcnt}[section]

\newcommand{\rem}[1]{}

\newcounter{desccount}

\newcommand{\descref}[1]{\hyperref[#1]{#1}}

\begin{document}

\sloppy

\begin{frontmatter}
\title{Towards multi-purpose locally differentially private synthetic data release via spline wavelet plug-in estimation}
\runtitle{Locally differentially private synthetic data release}

\begin{aug}
\author[A]{\fnms{Thibault}~\snm{Randrianarisoa}
\ead[label=e1]{t.randrianarisoa@utoronto.ca}\thanksref{t1}} ,
\author[B]{\fnms{Lukas}~\snm{Steinberger}\ead[label=e2]{lukas.steinberger@univie.ac.at}\thanksref{t2}\orcid{0000-0002-2376-114X}}
\\
\and
\author[C]{\fnms{Botond}~\snm{Szabo}\ead[label=e3]{botond.szabo@unibocconi.it}\thanksref{t1} \orcid{0000-0002-5526-8747}}
\address[A]{Department of Statistical Sciences,
University of Toronto 
\printead[presep={ ,\ }]{e1}}

\address[B]{Department of Statistics and OR,
University of Vienna 
\printead[presep={,\ }]{e2}}
\address[C]{Department of Decision Sciences and BIDSA,
Bocconi University
\printead[presep={,\ }]{e3}}

\thankstext{t1}{Co-funded by the European Union (ERC, BigBayesUQ, project number:
101041064). Views and opinions expressed are however those of the author(s) only and do not
necessarily reflect those of the European Union or the European Research Council. Neither the
European Union nor the granting authority can be held responsible for them.}
\thankstext{t2}{Supported by the Austrian Science Fund (FWF) Project: I 5484-N,
as part of the Research Unit 5381 of the German Research Foundation.}
\end{aug}

\begin{abstract}
We develop plug-in estimators for locally differentially private semi-parametric estimation via spline wavelets. The approach leads to optimal rates of convergence for a large class of estimation problems that are characterized by (differentiable) functionals $\Lambda(f)$ of the true data generating density $f$. The crucial feature of the locally private data $Z_1,\dots, Z_n$ we generate is that it does not depend on the particular functional $\Lambda$ (or the unknown density $f$) the analyst wants to estimate. Hence, the synthetic data can be generated and stored a priori and can subsequently be used by any number of analysts to estimate many vastly different functionals of interest at the provably optimal rate. In principle, this removes a long standing practical limitation in statistics of differential privacy, namely, that optimal privacy mechanisms need to be tailored towards the specific estimation problem at hand.
\end{abstract}



\end{frontmatter}

\section{Introduction}

Differential privacy \citep[][]{Dwork06b, Dwork06, Evfim03} was conceived 20 years ago as a once-and-for-all defense against any kind of possible data privacy breach. Already in these early references the idea of synthetic data release by private histograms was envisioned. It means that a new, and in some sense artificial, data set is (randomly) generated from sensitive original data such that these new data can be made public without any concern for privacy violation but contain similar statistical information. In the last decade, however, mathematical statisticians have demonstrated in a multitude of papers \citep[cf.][to mention only a few]{BerrettButucea19,Acharya21,Cai21, ButuceaPrivacy, Duchi17, Lalanne23, Rohde20, Yu23, Steinberger24} that statistically optimal privacy mechanisms have to be tailor made for the specific inference problem at hand. In other words, no single private data release mechanism can be simultaneously optimal for many different  estimation problems.

To give a very simple example, in the local paradigm of DP, where there is no trusted curator available, when original sensitive iid data are $X_1,\dots, X_n$ and the goal is moment estimation (that is, $\theta = \E[X^k]\in\R$) on a sample space $\mathcal X\subseteq[-M,M]$, we can achieve the optimal dependence of the risk on sample size $n$ and the privacy budget $\alpha$ by simple local Laplace addition
$$
Z_{i,k} = X_i^k + Lap\left(\frac{\alpha}{2M^k}\right)
$$
and then averaging $\hat{\theta}_n^{(k)} = \frac1n\sum_{i=1}^n Z_{i,k}$.
If an analyst has retrieved, say, $Z_{i,1}$, $i=1,\dots, n$, for mean estimation, then the privacy budget $\alpha$ for each of those $n$ data owners is already used up. If, at a later point, she decides to also estimate the second moment, then there are only two possible alternatives: either she queries $n$ new data owners to obtain a new sanitized sample $Z_{i,2}$, $i=n+1,\dots, 2n$, which may be expensive, or she can try to use the $Z_{i,1}$, $i=1,\dots,n$, also for the new estimation task. The latter option is, indeed, feasible by a deconvolution approach, but it is clearly statistically much less efficient than using the $Z_{i,2}$. Notice that if the analyst had first retrieved $Z_{i,2}$, $i=1,\dots, n$, and then later decided to estimate also the mean, then, depending on the distribution of $X_i$, this task may be impossible altogether without querying new data owners. For instance, if $\P(X_i=1)=(\theta+1)/2=1-\P(X_i=-1)$ then $\E[X_i] = \theta\in[-1,1]$ but the distribution of $X_i^2 = 1$ does not depend on $\theta$.

In light of the above, it would be highly desirable to query the $n$ data owners only once to obtain a sanitized synthetic data set $Z_1, \dots, Z_n$ which can subsequently be used in a statistically optimal way by analysts for arbitrary inference tasks. In this paper we provide a proof of concept that this can, indeed, be done for a large class of possible inference problems.

A natural approach is that of (semiparametric) plug-in estimation. If the sensitive data $X_1, \dots, X_n$ are iid from some smooth density $f\in\mathcal F$ on $[0,1]$, then we let $\Lambda:\mathcal F \to\R$ denote some functional of interest. We construct a local privacy mechanism that turns $X_i$ into a private view $Z_i$ on some appropriate space $\mathcal Z$, such that these sanitized data can be used to compute a density estimator $\hat{f}_n$ for which the plug-in rule $\Lambda(\hat{f}_n)$ achieves the (private) minimax rate for estimating $\Lambda(f)$ (among all locally differentially private estimators of $\Lambda(f)$ and not only the plug-in rules). This is shown to hold for a large class of differentiable functionals $\Lambda$. In other words, no matter what the estimation task formalized by the functional $\Lambda$ is, the analyst can use the synthetic data $Z_1,\dots, Z_n$ to construct a minimax optimal estimator. Since $Z_1, \dots, Z_n$ are generated only once and protect $\alpha$ local differential privacy of data owners, there is no limit to the number of analysts or to the number of estimation tasks each analyst can solve. 

\subsection{Related results}

Needless to say, the locally private density estimation problem has been treated before by several authors \citep{Duchi17, Kroll19, ButuceaPrivacy, Gyorfi23, Kroll21, Kroll22, Schluttenhofer22}. However, not every private density estimator $\hat{f}_n$ is suitable for multi-purpose synthetic data release in the sense that it is optimal for a large class of estimation problems. First of all, if we want to allow also for functionals $\Lambda$ that involve derivatives of the true density $f$ (such as Fisher-Information; see Section~\ref{sec:Examples} for a list of examples), then the estimated density $\hat{f}_n:[0,1]\to \R$ should also be smooth, ruling out the classical histogram approach. Smoothness can be obtained, for instance, via kernel density estimation $\tilde{f}_n(t) = \frac{1}{n}\sum_{i=1}^n K_h(X_i-t)$, where $K_h(x) = \frac{1}{h}K(\frac{x}{h})$. However, it is not a trivial task to combine the kernel density approach with differential privacy. The problem arises from the fact that we can not simply add noise to the original data $X_i$ without running into the deconvolution problem mentioned above. An optimal randomization procedure should rather be applied to $K_h(X_i-t)$. This is straightforward by Laplace noise addition if the point $t$ is fixed. However, for estimation of the entire density, which is essential for multi-purpose data release, we should publish the whole randomized trajectory $t\mapsto Z_i(t) = K_h(X_i-t) + \Xi(t)$, where $\Xi$ is an appropriate stochastic process. \citet{Kroll19} explored this approach, but achieved only \emph{approximate} local differential privacy. 

Another practical difficulty with the kernel density idea is data storage and numerical differentiation. As we can never hope to generate and store the entire function $Z_i:[0,1]\to\R$ on a server (recall that $X_i$ has to remain hidden), we would have to a priori discretize and therefore limit the precision any analyst can achieve by specifying a grid of values $Z_i(t_1), \dots, Z_i(t_N)$. This situation is markedly different from, e.g., using a projection type approach where only sanitized (wavelet) coefficients need to be stored and differentiation could even be done analytically by the analyst.

Optimal rates for nonparametric density estimation under privacy constraints with a wavelet approach were obtained in \citet{ButuceaPrivacy}. However, they use a specific wavelet basis and it is unclear whether it is simultaneously optimal for estimation of derivatives and point evaluations of the density, required for our analysis. \citet{GoldsteinMesser92} developed a framework to extend nonparametric density estimators to semiparametric plug-in estimators and obtained optimal rates, but they only considered kernel estimators, which are not suitable for our purpose.

\subsection{Our contribution}

For the reasons mentioned above, we here follow a projection type approach with a B-spline basis as a practical and natural choice for estimation of a smooth density and its derivatives (cf. Section~\ref{sec:SplineEst}). More precisely, if $B_{k,j_n}$ are the B-spline functions with equispaced knot sequence of step $2^{-j_n}$ and degree $d$, we estimate $f$ by
\begin{equation}\label{eq:B-splineEst}
\hat{f}_n^{(j_n)}(t) = \sum_{k=1}^{2^{j_n}+d} \bar{Z}_k^{(j_n)} B_{k,j_n}(t),
\end{equation}
where $\bar{Z}_k^{(j_n)} = \frac1n\sum_{i=1}^n Z_{i,k,j_n}$ and $Z_{i,k,j_n} = e_{k,j_n}(X_i) + W_{i,k,j_n}$ for appropriate Laplace random variables $W_{i,k,j_n}$ and functions $e_{k,j_n}$. Even though this locally private procedure achieves optimal rates on atomic functionals (e.g., point evaluations of the density), it does not perform well enough for estimation of smooth functionals (e.g., functionals of integral form such as moments or integrals of functions of $f$, etc.) due to excessive noise accumulation \citep[the latter was also observed by][Section~5.2.2., who proposed a more refined noise addition scheme to fix the problem at least for density estimation in $L^2$]{Duchi17}. It turns out, however, that with the structure of a multiresolution analysis we can efficiently distribute noise across resolution levels, leading to a locally private density estimator that achieves optimal rates on both smooth as well as atomic functionals. See Remark \ref{rk: failure spline smooth} in Section \ref{sec:WaveletEst} for more details on this issue. Thus, in order to widen the class of estimation problems which can be solved by our synthetic private data, we extend the B-spline approach to spline wavelets \citep{jia2006stable} (cf. Section~\ref{sec:WaveletEst} for details) and consider an estimator of the form
\begin{equation}\label{eq:waveletEst}
    \hat{f}_n^{(j_n)}(t) = \sum_{j=j_0-1}^{j_n}\sum_{k\in\mathcal M_j} \bar{Z}_{jk} \tilde{\psi}_{jk}(t),
\end{equation}
where $Z_{ijk} = \psi_{jk}(X_i) + W_{ijk}$ and $\bar{Z}_{jk} = \frac1n\sum_{i=1}^n Z_{ijk}$.

Extending results by \citet{GoldsteinMesser92} we prove optimal rates of convergence and adaptation for locally private plug-in estimators based on spline wavelets.

\subsection{Practical advantages of the multi resolution analysis and adaptation}

Notice that there is another fundamental problem with the private estimation approach in \eqref{eq:B-splineEst} that is resolved by the wavelet estimator \eqref{eq:waveletEst}. In order to achieve the optimal rate of convergence the tuning parameter $j_n$ has to be chosen in dependence of the smoothness of the true data generating density (and in general also in dependence on the functional $\Lambda$ we are trying to estimate). In the notation of \eqref{eq:B-splineEst}, our multi-purpose synthetic data should therefore consist of $\mathbf{Z}^{(j_n)} = (\bar{Z}_1^{(j_n)}, \bar{Z}_2^{(j_n)}, \dots, \bar{Z}_{2^{j_n}+d}^{(j_n)})$, for some appropriate choice of $j_n$. For any given $j_n$, these data can be released such that $\alpha$-local differential privacy is protected. However, we cannot use the exact same data in order to adaptively choose $j_n$. To do that we would actually need to try different values of $j_n$ (e.g., model selection or Lepski's method) and fit models based on $\mathbf{Z}^{(j)}$, for $j\in\mathcal J$, say, resulting in the need for further locally private data generation. This comes at the price of a privacy loss. If $\mathbf{Z}^{(j)}$ can be generated with a privacy parameter $\alpha_j$ and $\alpha$ is the overall privacy budget, then we have to make sure that $\sum_{j\in\mathcal J}\alpha_j\le \alpha$. This is achieved, for instance, by setting $\alpha_j=\alpha/|\mathcal J|$, but of course estimation based on $\mathbf{Z}^{(j)}$ with tuning parameter $j$ and privacy budget $\alpha$ will lead to a faster rate than estimation based on the privacy budget $\alpha/|\mathcal J|$. Typically the loss is polynomial in $|\mathcal J|$, which is itself of logarithmic order $|\mathcal J| \asymp \log(n)$. Hence, the price of adaptation of this private estimation scheme is at least another poly-logarithmic term and we would also have to a priori specify the grid of values $\mathcal J$ which would subsequently be fixed for all data analysts. 

This additional poly-logarithmic cost of adaptation arising from the need of repeated private data generations may not seem to be a big issue theoretically, given that in certain cases there is the price of a logarithmic factor anyways, even when original data $X_1,\dots, X_n$ are observed. Furthermore, it is an open question whether local differential privacy necessarily causes a logarithmic price for adaptation \citep[cf.][Section~6]{ButuceaPrivacy}. However, with a wavelet basis \eqref{eq:waveletEst} the situation is a bit more convenient, at least from a practical point of view. In fact, we can release all the randomized data $\bar{Z}_{jk}$, $j\in\{j_0-1, j_0,\dots, j_n\}$, $k\in\mathcal M_j$, up to arbitrary $j_n$, at a constant privacy budget of $\alpha$ (cf. Proposition~\ref{prop:PrivacyWavelets}). Hence, we can pick a privacy level $\alpha$ and a very large number $j_{max}$, say, and release $\bar{Z}_{jk}$, $j\in\{j_0-1, j_0, \dots, j_{max}\}$, $k\in\mathcal M_j$, to protect $\alpha$ local differential privacy. Now any data analyst can subsequently use these data to compute \eqref{eq:waveletEst} as long as $j_n\le j_{max}$. In particular, they can also compute \eqref{eq:waveletEst} and derivatives thereof for different values of tuning parameters $j\in\mathcal J$, say, as long as $\max \mathcal J\le j_{max}$. See Section~\ref{sec:Adaptation} for further details, where we pick $j_{max}=(\log n)/3$.

\subsection{Organization of the paper}

In Section~\ref{sec:Model} we give a detailed description of the statistical model, the set of estimation problems characterized by a functional $\Lambda$ and the privacy constraint that needs to be satisfied. We also collect our basic definitions and notations. Section~\ref{sec:LowerBounds} contains private minimax lower bounds for our class of functionals $\Lambda$ and for the class of all locally private estimators, which we aim to achieve with our plug-in approach. In Section~\ref{sec:Background} we provide some background on B-splines and associated wavelet bases and summarize the approximation theoretic facts needed to prove our main results. Sections~\ref{sec:SplineEst} and \ref{sec:WaveletEst} contain the main results on optimality of our synthetic data release and plug-in estimators. Section~\ref{sec:SplineEst} deals with a simpler B-spline estimator that is, however, only optimal on atomic functionals. Section~\ref{sec:WaveletEst} then introduces the spline wavelet approach, which is subsequently shown to be optimal on atomic and smooth functionals. The optimal choice of resolution level $j_n$ for estimation of atomic functionals, however, depends on the unknown smoothness of the data generating density $f$. Thus, we introduce a fully data-driven choice of $j_n$ in Section~\ref{sec:Adaptation}. Finally, in Section~\ref{sec:Examples}, we list a few examples of common statistical functionals $\Lambda$ and show how they fit into our general private estimation framework. The proofs are deferred to the Appendix.

\section{Notations and model assumptions}
\label{sec:Model}

The original, possibly confidential, data $X_1,\dots,X_n$ are assumed to be independent and identically distributed random variables with Lebesgue density $f$ on $[0,1]$ and we want to estimate $\Lambda(f)$ for some functional $\Lambda$ (not necessarily linear) while ensuring local differential privacy. We denote by $P_f$ the associated probability measure on the sample space $([0,1],\mathcal B([0,1]))$ and by $\P_f = P_f^n$ the corresponding product measure. We assume in the following that, for $p$ a nonnegative integer, $f\in C^p[0,1]$ the set of continuous functions with $p$ or more continuous derivatives and denote $f^{(j)}$ the $j$--th derivative of $f$. Also, for $1\leq q<\infty$ and $\lambda=(\lambda_0,\dots,\lambda_p)$ a vector of finite measures on $[0,1]$, we note
\[ \norm{f}_{(q,p,\lambda)} = \sum_{j=0}^p \left(\int_0^1 |f^{(j)}|^q d\lambda_j\right)^{1/q},\]
and define $\norm{f}_{(\infty,p,\lambda)}$ similarly when $q=\infty$, replacing integrals with essential supremums. When $\lambda$ is the vector of Lebesgue measures, we drop the dependence on this vector in the previous display and write $\norm{f}_{(q,p)}$ for simplicity. We also introduce \[\mathcal{B}_{(q,p)}(M)\coloneqq \left\{f\in C^p[0,1]\colon\ \norm{f}_{(q,p)}<M\right\}\] the ball of radius $M$ in this metric. For some measure $G$ on $[0,1]$, we denote $\norm{\cdot}_{L^p(G)}$ the usual $L^p$-norm, and drop the dependence on $G$ when it is the Lebesgue measure. In the following, we restrict our attention to densities satisfying
\begin{equation}\label{eq: func space} f\in \mathcal{W}_p\coloneqq \left\{f\in C^p[0,1]:\ f\geq0,\ \int_0^1f=1\right\}\ \cap \ \mathcal{B}_{(\infty,p)}(M),\end{equation}
for some unknown $M>0$.
Also, we consider functionals $\Lambda$ that admit the following expansion for $0\leq m<p$ and $\lambda$ a vector of finite measures depending on $\Lambda$ but not $f$, 
\begin{equation}\label{eq: form functional}
    \Lambda(f+h) = \Lambda(f) + T_f(h) + O(\norm{h}^2_{(2,m,\lambda)}),
\end{equation}
for $f\in \mathcal{W}_p$, $h\in C^p[0,1]$ with $\norm{h}_{(\infty,m)}$ small enough and $T_f$ (the functional derivative of $\Lambda$, depending on $f$) a bounded linear functional on $C^p[0,1]$. As argued in \cite{GoldsteinMesser92}, the Hahn-Banach theorem and the Riesz representation theorem imply that
\begin{equation}\label{eq: Hahn-Banach}
    T_f(h)=\sum_{j=0}^p \int_0^1 h^{(j)} d\mu_j,
\end{equation} where $\mu_j$ is a finite signed Borel measure on $[0,1]$ (possibly depending on $f$). From Lebesgue's decomposition theorem, the finite measures $\nu$ we consider here can all be decomposed as
\[\nu = \nu_{\text{cont}}+\nu_{\text{sing}}+\nu_{\text{pp}},\]
where $\nu_{\text{cont}}$ is the absolutely continuous part, $\nu_{\text{sing}}$ is the singular continuous part and $\nu_{\text{pp}}$ is the pure point part (a discrete measure with a countable set of atoms). We restrict ourselves to situations where the pathological singular continuous part is null, $\nu_{\text{sing}}=0$, which will be the case in our practical examples. We also assume that all absolutely continuous measures have bounded densities from now on.

\begin{assumption}\label{assum: no singular continuous part}
    We assume the finite signed measures $\mu_j$ in the derivative \eqref{eq: Hahn-Banach} and $\lambda_j$ in the remainder term of \eqref{eq: form functional} are composed of only absolutely continuous parts, with Lebesgue densities bounded by quantities depending only on $M$ as in \eqref{eq: func space}, and discrete parts. Also, we assume that the total mass they assign to $[0,1]$ is upper bounded by a quantity depending only on $M$.
\end{assumption}

In the following, we consider differentiable functionals $\Lambda$ (of order $m$), whose decomposition \eqref{eq: form functional} is such that for any $j\leq p$, the nonnegative measure $|\mu_j|= \mu_j^++\mu_j^-$, where we use the Jordan decomposition $\mu_j= \mu_j^+-\mu_j^-$ (recall that these measures may depend on $f$), satisfies
\begin{equation}\label{eq:bound total var}\underset{f\in \mathcal{W}_p,\, j\leq p}{\sup}\ |\mu_j|([0,1])<\infty.\end{equation} 
This implies that the decomposition is uniform over $\mathcal{W}_p$ in the sense that
\begin{equation}\label{eq: bound deriv func}\underset{f\in \mathcal{W}_p}{\sup} \norm{T_f}_p\coloneqq \underset{f\in \mathcal{W}_p}{\sup}\quad  \underset{\norm{h}_{(\infty, p)}\leq 1}{\sup} \left|T_fh\right|<\infty.\end{equation}
Among these, we can identify the smooth functionals of order $m$ for which $T_f(h)=\int_0^1 h\ \omega_f$, for all $f\in\mathcal{W}_p$, and where the bounded measurable function $\omega_f$ depends on $f$ only and satisfies
\begin{equation}\label{eq: bound deriv omega}\underset{f\in\mathcal{W}_p}{\sup} \norm{\omega_f}_{L^\infty} < \infty.\end{equation}
We also study the atomic functionals of index $s$, $0\leq s\leq p$, which are such that for all $h\in C^p[0,1]$, $T_f$ can be represented as
$T_f(h)=\sum_{j=0}^{s_f} \int_0^1 h^{(j)} d\mu_{j,f}$
with $\mu_{s_f,f}$ having a discrete component $\delta_{s_f,f}$, and $s = \sup_{f\in\mathcal W_p} s_f$ being the largest integer for which such representation is valid. In the following, $C_p, C_{f}, C_{p,f}$, etc. denote constants depending on the subscript parameters only, and whose value may change from line to line. Also, for $a,b\in \mathbb{R}$, $a\vee b$ is the largest value of $a,b$ and $a\wedge b$ is the smallest one. If $a\leq b$ are integers in $\mathbb{Z}$, the integer interval $\left\{a,a+1,\dots,b\right\}$ is denoted $[\![ a,b]\!]$. For sequences $(a_n)_n,(b_n)_n$, we note $a_n\lesssim b_n$ if $a_n\leq Cb_n$ for some constant $C$ and $a_n\asymp b_n$ if $a_n\lesssim b_n$ and $b_n\lesssim a_n$.

Finally, we focus on the local paradigm of differential privacy, where data owners can generate sanitized views of their sensitive information ``on their local machine'' without sharing the original information with anybody else. For a privacy budget $\alpha>0$, we write $\mathcal Q_\alpha$ for the set of all Markov kernels $Q$ from $([0,1],\mathcal B([0,1]))$ to some measurable space $(\mathcal Z, \mathcal G)$ such that 
\begin{equation}\label{eq:LDP}
Q(A|x)\le e^\alpha Q(A|x'), \quad \text{for all } A\in\mathcal G \text{ and for all } x,x'\in[0,1].
\end{equation} 
Using such a privacy mechanism $Q\in\mathcal Q_\alpha$, data owners can independently generate sanitized versions $Z_i$ of their sensitive information $X_i$ by $Z_i\thicksim Q(\cdot|x)$ with $x=X_i$. In other words, $Q(\cdot|x)$ denotes the conditional distribution of $Z_i$ given $X_i=x$ and the marginal distribution of $Z_i$ is given by $QP_f := \int_\X Q(\cdot|x) P_f(dx)$. Also notice that $Z_1,\dots, Z_n$ are iid according to $QP_f$ and the conditional distribution of $(Z_1,\dots, Z_n)$ given $X_1=x_1, \dots, X_n=x_n$ is given by $Q(\cdot|x_1)\times\dots\times Q(\cdot|x_n)$, which satisfies the conventional definition of (central) differential privacy as in \citet{Dwork06}. Notice that $Q(\cdot|x)\ll Q(\cdot|x')$ for all $x,x'\in[0,1]$ and hence \eqref{eq:LDP} can be verified by checking whether
$$
\frac{dQ(\cdot|x)}{dQ(\cdot|x')}(z) \le e^\alpha.
$$
We write $\E$ and $\mathbb{P}$ to denote expectation and probability with respect to the joint distribution of all the private and public data $X_1,\dots, X_n$, $Z_1,\dots, Z_n$. To stress dependence on $Q$ and $f$, we sometimes also write $\E_{QP_f}$ or simply $\E_f$, if the privacy mechanism $Q$ is clear from the context.

\section{Minimax lower bounds}
\label{sec:LowerBounds}
In this section we prove private minimax lower bounds of estimation for large classes of functionals $\Lambda$ over $\mathcal W_p$. In subsequent sections we then show that these lower bounds can be achieved by our plug-in estimators. Notice that the collection of estimators considered in our minimax lower bounds is the collection of all measurable functions of the sanitized data $Z_1,\dots, Z_n$ rather than just the plug-in rules. 

We begin with a general private minimax lower bound of parametric order that holds for estimation of any differentiable and non-constant functional $\Lambda$ on a convex subset $\mathcal W\subseteq\mathcal W_p$. Notice that $\mathcal W_p$ itself is convex. The proof of the following result is deferred to Appendix~\ref{app:LB}.

\begin{theorem} \label{thm:LBparam}
Fix $n,p\in\N$, $\alpha>0$ and a convex set $\mathcal W\subseteq \mathcal W_p$. Let $\Lambda:\mathcal W \to \mathbb R$ be differentiable as in \eqref{eq: form functional}. If $\Lambda$ is not constant on $\mathcal W$, then there exist constants $c,C>0$ depending only on $\Lambda$ and $\mathcal W$, such that the minimax-risk of locally private estimation is lower bounded by
\begin{align*}
\inf_{Q\in\mathcal Q_\alpha} \inf_{\hat{\Lambda}_n} \sup_{f\in\mathcal W} \E_{QP_f}\left[|\hat{\Lambda}_n - \Lambda(f)|\right] \ge c \left(n(e^\alpha-1)^2\right)^{-1/2},
\end{align*}
provided that $n(e^\alpha-1)^2\ge C$. Here, the second infimum is over all estimators $\hat{\Lambda}_n$ that take the sanitized data $Z_1,\dots, Z_n \stackrel{iid}{\thicksim} QP_f$ as input.
\end{theorem}

For large $\alpha$ the bound in Theorem~\ref{thm:LBparam} can become arbitrarily small. Notice, however, that the classical lower bounds without a privacy constraint derived by \cite{GoldsteinMesser92} also apply here. Simply condition on the original data $X=(X_1,\dots, X_n)$ and use Jensen's inequality to get
\begin{align*}
\E_{QP_f}\left[|\hat{\Lambda}_n - \Lambda(f)|\right]
&=
\E\left[\E\left[|\hat{\Lambda}_n(Z) - \Lambda(f)|\Big|X\right]\right]
\ge
\E\left[\left|\E\left[\hat{\Lambda}_n(Z)\Big|X\right] - \Lambda(f)\right|\right]\\
&=
\E_{P_f}\left[\left| \tilde{\Lambda}_n - \Lambda(f)\right|\right],
\end{align*}
where $\tilde{\Lambda}_n$ is an estimator of $\Lambda(f)$ that takes the original data $X_1,\dots, X_n \stackrel{iid}{\thicksim} f$ as inputs.

In order to prove a tighter lower bound of nonparametric order, we need to make sure that our statistical model $\mathcal W\subseteq\mathcal W_p$ does not degenerate to a parametric class. The following notion is an adaptation of Definition~3.4 of \cite{GoldsteinMesser92}.

\begin{definition}\label{def:ND}
    Let $\mathcal W\subseteq \mathcal W_p$ and $\Lambda:\mathcal W \to \mathbb R$ be differentiable as in \eqref{eq: form functional} and atomic of index $s$. We say that $(\mathcal W, \Lambda)$ is \emph{non-degenerate} if for all $x_0\in[0,1]$ there exists a compactly supported function $\psi\in C^p(\R)$ with $\psi^{(s)}(0)\neq 0$ and a positive constant $\delta_0>0$, such that the following holds true: For all $f_0\in\mathcal W$ there exists $\eta_0>0$, such that for all $\eta\in(0,\eta_0)$ and all $\delta\in(0,\delta_0)$ the function
    $$
    f_1(x) := f_0(x) + \eta\delta^p\psi((x-x_0)/\delta), \quad x\in[0,1],
    $$
    belongs to $\mathcal W$. If this is not the case we call $(\mathcal W,\Lambda)$ \emph{degenerate}.
\end{definition}

For instance, any set of polynomials of bounded degrees must be degenerate because no non-trivial polynomial can have compact support. Lemma~\ref{lemma:W0} in Appendix~\ref{app:LB} shows that for a differentiable functional $\Lambda$ of arbitrary index $s$ and $\mathcal W^{\circ} := \{f\in\mathcal W_p : f>0\}$ the model $(\mathcal W^\circ,\Lambda)$ is non-degenerate. Hence, the following theorem which is proved in Appendix~\ref{app:LB}, in particular, provides a lower bound on $\inf_{Q\in\mathcal Q_\alpha} \inf_{\hat{\Lambda}_n} \sup_{f\in\mathcal W_p} \E_{QP_f}[|\hat{\Lambda}_n - \Lambda(f)|]$, since $\mathcal W^\circ\subseteq\mathcal W_p$.

\begin{theorem} \label{thm:LBatomic}
Fix $n,p\in\N$, $\alpha>0$ and a differentiable functional $\Lambda:\mathcal W\to\R$ that is atomic of index $s$, $0\le s\le p$. If $(\mathcal W, \Lambda)$ is non-degenerate as in Definition~\ref{def:ND}, then there exist constants $c,C>0$ depending only on $\Lambda$ and $\mathcal W$, such that the minimax-risk of locally private estimation is lower bounded by
\begin{align*}
\inf_{Q\in\mathcal Q_\alpha} \inf_{\hat{\Lambda}_n} \sup_{f\in\mathcal W} \E_{QP_f}\left[|\hat{\Lambda}_n - \Lambda(f)|\right] \ge c \left(n(e^\alpha-1)^2\right)^{-\frac{p-s}{2p+2}},
\end{align*}
provided that $n(e^\alpha-1)^2\ge C$. Here, the second infimum is over all estimators $\hat{\Lambda}_n$ that take the sanitized data $Z_1,\dots, Z_n \stackrel{iid}{\thicksim} QP_f$ as input.
\end{theorem}

\section{Splines and spline wavelets}
\label{sec:Background}

In this section we collect some background on B-splines and spline wavelets that is necessary to understand the construction of our estimators and the proofs of our main results. The reader who is familiar with that theory may skip this section.

\subsection{Splines}
\label{sec:splines}

For a knot sequence \[\mathbf{\xi}\coloneqq {\left\{\xi_1\leq\dots\leq\xi_l\right\}},\quad l\in\mathbb{N},\]
in $\mathbb{R}$, we define recursively the $B$--splines $B_{k,d,\mathbf{\xi}}$ of degree $d\leq l-2$, $1\leq k\leq l-d-1$, by
\[B_{k,0,\mathbf{\xi}}=\mathds{1}_{[\xi_k, \xi_{k+1})},\]
and
\[B_{k,d,\mathbf{\xi}}(x)=\frac{x-\xi_k}{\xi_{k+d}-\xi_k}B_{k,d-1,\mathbf{\xi}}(x)+\frac{\xi_{k+d+1}-x}{\xi_{k+d+1}-\xi_{k+1}}B_{k+1,d-1,\mathbf{\xi}}(x),\]
or identically equal to $0$ if $\xi_{k+d+1}=\xi_k$ (convention: fractions with denominator $0$ have value $0$).
We also recall that this definition guarantees that (see Section 1, Theorem 3 of \cite{LycheSpline})
\begin{equation}\label{eq: derivatives} \frac{dB_{k,d,\mathbf{\xi}}}{dx}(x) = d \left[\frac{B_{k,d-1,\mathbf{\xi}}(x)}{\xi_{k+d}-\xi_k}-\frac{B_{k+1,d-1,\mathbf{\xi}}(x)}{\xi_{k+d+1}-\xi_{k+1}}\right],\quad d> 1.\end{equation}
From this, we deduce that $B_{k,d,\mathbf{\xi}}$ is $(d-1)$--times continuously differentiable. In addition, Theorem 5 in Section 1 of \cite{LycheSpline} implies
\begin{equation}\label{eq: integral b-splines}\int_{0}^1 B_{k,d,\xi}(x)dx = \int_{\xi_k}^{\xi_{k+d+1}} B_{k,d,\xi}(x)dx = \frac{\xi_{k+d+1}-\xi_k}{d+1},\end{equation} as the support of $B_{k,d,\xi}$ is $[\xi_k,\xi_{k+d+1})$. We then deduce that
\begin{equation}\label{eq: norm splines}\frac{\xi_{k+d+1}-\xi_k}{(d+1)^2}\leq\int_{\xi_k}^{\xi_{k+d+1}} B_{k,d,\xi}(x)^2dx\leq \frac{\xi_{k+d+1}-\xi_k}{d+1},\end{equation}
the upper bound coming from $0\leq B_{k,d,\xi}\leq 1$ ((1.7) and (1.21) of Section 1 in \cite{LycheSpline}) and the lower bound coming from Jensen's inequality. (1.21) in \cite{LycheSpline} states that the B-splines define a local partition of unity
\begin{equation}\label{eq: partition of unity}\sum_{k=1}^{l-d-1} B_{k,d,\xi}(x) = 1,\qquad x\in[0,1].\end{equation}

\noindent In the following, we note $\mathbf{B}_{k,d,\xi}\coloneqq B_{k,d,\xi}/ \norm{B_{k,d,\xi}}_{L^2}$. The space
\[\mathbb{S}_{d,\xi} \coloneqq \left\{s\colon [\xi_{1},\xi_{l})\to \mathbb{R},\ s=\sum_{k=1}^{l-d-1}c_k B_{k,d,\xi},\ c_k\in \mathbb{R} \right\}\]
is the space of splines, spanned by the B-splines. Assuming the knot sequence satisfies $\xi_{d+1}<\xi_{d+2}<\dots<\xi_{l-d-1}<\xi_{l-d}$, this space coincides with the space of piecewise polynomials of degree $d$ over the intervals $[\xi_{k},\xi_{k+1})$ with $d-1$ continuous derivatives. This space has dimension $l-d-1$.

\begin{theorem}\label{th: from L2 to l2}
    For any $d\geq 0$, there exist positive constants $K_d,\ \Tilde{K}_d$ depending only on $d$ such that for any vector $\mathbf{b}=(b_1,\dots,b_{l-d-1})$, we have
    \[K_d^{-1}\norm{\mathbf{b}}_{\ell^2}\leq \norm{\sum_{k=1}^{l-d-1} b_j \mathbf{B}_{k,d,\xi}}_{L^2}\leq K_d \norm{\mathbf{b}}_{\ell^2},\]
    and
\[\Tilde{K}_d^{-1}\norm{\mathbf{b}}_{\ell^\infty}\leq \norm{\sum_{k=1}^{l-d-1} b_k B_{k,d,\xi}}_{L^\infty}\leq \Tilde{K}_d \norm{\mathbf{b}}_{\ell^\infty}.\]
\end{theorem}
\begin{proof}
    It follows from Theorem 11 of Section 1 in \cite{LycheSpline} with $q=2$, noting the integral and the squared norm of $B_{k,d,\xi}$ from \eqref{eq: integral b-splines} and \eqref{eq: norm splines} are equal up to a constant depending on $d$ only, and $q=\infty$.
\end{proof}

Given the knot sequence $\{\xi_1\leq\dots\leq\xi_{l}\}$ such that $\xi_j<\xi_{j+d+1}$ for any $j$ and whose end knots have multiplicity $d+1$ (i.e., $0=\xi_1=\xi_{d+1}<\xi_{d+2}$ and $\xi_{l-d-1}<\xi_{l-d}=\xi_{l}=1$), an approximation of the density $f\in\mathcal{W}_p$ (or any function in $\mathcal{B}_{(\infty,p)}(M)$) is defined as
\[
\mathcal{S}_{d,\mathbf{\xi}}f(x)\coloneqq \sum_{k=1}^{l-d-1} \mathcal{L}_{k,d,\mathbf{\xi}}(f)B_{k,d,\mathbf{\xi}}(x) = \sum_{k=1}^{l-d-1} \mathcal{L}_{k,d,\mathbf{\xi}}(f)\ \norm{B_{k,d,\mathbf{\xi}}}_{L^2}\ \mathbf{B}_{k,d,\mathbf{\xi}}(x)\in\mathbb{S}_{d,\xi},
\]
see (1.122) in \cite{LycheSpline}. We define $h_{k,d,\xi}=\xi_{m_{k,d}+1}-\xi_{m_{k,d}}$ with $m_{k,d}$ so that this gap is maximal among indices $k,\dots,k+d$ and 
\[
\mathcal{L}_{k,d,\mathbf{\xi}}(f)\coloneqq \frac{1}{h_{k,d,\mathbf{\xi}}}\int_{\xi_{m_{k,d}}}^{\xi_{m_{k,d}+1}} \left(\sum_{i=0}^da_{k,i}\Big(\frac{x-\xi_{m_{k,d}}}{h_{k,d,\mathbf{\xi}}}\Big)^i\right)f(x)\ dx,
\]
for some coefficients $a_{k,i}$ satisfying the relation 
\[
H_{d+1}\mathbf{a}_k = \mathbf{c}_k,
\]
where $\mathbf{a}_k=(a_{k,0},\dots,a_{k,d})^T$, $H_{d+1}$ is a $(d+1)\times(d+1)$ matrix with elements $(H_{d+1})_{\alpha,\beta}=\frac{1}{\alpha+\beta-1}$ (see (1.124) and the proof of Lemma~2 in Section 1 of \cite{LycheSpline}) and $\mathbf{c}_k$ is a vector with its $i$--th entry bounded by $(d+1)^{i-1}$ (see the proof of Lemma~3 in Section~1 of \cite{LycheSpline}). Therefore, the entries of $\mathbf{a}_k$ are bounded by a quantity that depends on $d$ only. We note that $h_{k,d,\xi}$ above is constant and equal to $2^{-j}$ when $\xi=\xi^{(j)}$.

From Proposition~7 in Section~1 of \cite{LycheSpline}, if $\xi_1=0$, $\xi_{l-d}=1$, and $f$ is such that $\norm{f}_{(\infty,r+1)}<\infty$, with $0\leq r \leq d$, and $0\leq q\leq r$,
\begin{equation}
    \norm{\left(f-\mathcal{S}_{d,\mathbf{\xi}}f\right)^{(q)}}_{L^\infty} \leq K_{d,r,q,\xi}\ (\max_k\ h_{k,d,\mathbf{\xi}})^{r+1-q}\ \norm{f^{(r+1)}}_{L^\infty},
\end{equation} 
where, for $C$ a constant depending on $d$ only,
\[ K_{d,r,q,\xi} = (2d+1)^{r+2-q}\left[1+C\max_{d+1\leq m\leq 2^j+d} \prod_{k=d-q+1}^d \frac{\xi_{m+d+1}-\xi_{m-d}}{\min_{m-k+1 \leq i\leq m}\ \xi_{i+k}-\xi_i}\right].\]
In the following, we will consider equispaced knot sequences of the form $\xi=\xi^{(j)}$ 
of step $2^{-j}$, that is, 
\begin{equation}\label{eq:knot seq multires}
0=\xi^{(j)}_1=\dots=\xi^{(j)}_{d+1}<\dots<\xi^{(j)}_k\coloneqq(k-d-1)/2^{j}<\dots<\xi^{(j)}_{2^j+d+1}=\dots=\xi^{(j)}_{2^j+2d+1}=1.
\end{equation}
This entails that for any possible $k$, $h_{k,d,\xi^{(j)}}= 2^{-j}$. Also, for any $d+1\leq m\leq 2^j+d$, this implies
\[\prod_{k=d-q+1}^d \frac{\xi^{(j)}_{m+d+1}-\xi^{(j)}_{m-d}}{\min_{m-k+1 \leq i\leq m}\ \xi^{(j)}_{i+k}-\xi^{(j)}_i} \leq \prod_{k=d-q+1}^d \frac{(2d+1)2^{-j}}{2^{-j}}\leq (2d+1)^q .\] 
Therefore, $K_{d,r,q,\xi^{(j)}}$ is upper bounded by a quantity that depends on $d$ only and we can rewrite the bound as
\begin{equation}\label{eq:bias}
    \norm{\left(f-\mathcal{S}_{d,\mathbf{\xi}^{(j)}}f\right)^{(q)}}_{L^\infty} \leq K_d 2^{-j(r+1-q)} \norm{f^{(r+1)}}_{L^\infty},
\end{equation} 
where $K_d$ is a constant depending  on $d$ only. We write 
\begin{equation}\label{eq: dual basis splines}
e_{k,j}(x)=2^j\norm{B_{k,d,\mathbf{\xi}^{(j)}}}_{L^2}\, \mathds{1}_{\left[\xi_{m_{k,d}}^{(j)}, \xi_{m_{k,d}+1}^{(j)}\right]}(x)\, \sum_{i=1}^d a_{k,i}\left(\frac{x-\xi_{m_{k,d}}^{(j)}}{2^{-j}}\right)^i.
\end{equation} 
If $X$ follows a distribution with density $f$, \[\mathbb{E}e_{k,j}(X)= \mathcal{L}_{k,d,\mathbf{\xi}^{(j)}}(f)\norm{B_{k,d,\mathbf{\xi}^{(j)}}}_{L^2}.\] Given the upper bound on the $a_{k,i}$ and the fact that $[(x-\xi^{(j)}_{m_{k,d}})/h_{k,d,\xi^{(j)}}]^i$ is bounded by one on $[\xi_{m_{k,d}}^{(j)}, \xi_{m_{k,d}+1}^{(j)}]$, then there exists a quantity $C_d$ depending on $d$ only such that $|e_{k,j}(x)|\leq C_d2^j \norm{B_{k,d,\mathbf{\xi}^{(j)}}}_{L^2} \leq C_d2^{j/2}$ for any $x\in[0,1]$, recalling \eqref{eq: norm splines}.

\subsection{Spline wavelets}
\label{sec: splines wav}

\subsubsection{Frames and Riesz bases}

For $j\geq0$, let's consider the space 
\begin{equation}\label{eq: equi spline space}\mathcal{V}_j\coloneqq\mathbb{S}_{d,j}\coloneqq\mathbb{S}_{d,\xi^{(j)}}\end{equation}
for $\xi^{(j)}$ the equispaced knot sequence of step $2^{-j}$ as in \eqref{eq:knot seq multires}.
It is possible to verify that $\mathcal{V}_{0}\subset \mathcal{V}_{1}\subset\dots$ and $\cup_{j\geq 0}\mathcal{V}_j$ is dense in the $L^2$--space. Denote by $\mathcal{U}_j$ the orthogonal complement of $\mathcal{V}_j$ in $\mathcal{V}_{j+1}$, which is of dimension $2^j$ since $\text{dim}(\mathcal{V}_j)=2^j+d$. Therefore, $L^2=\mathcal{V}_0 \bigoplus \mathcal{U}_0 \bigoplus \mathcal{U}_1 \bigoplus \dots=\mathcal{V}_j \bigoplus \mathcal{U}_j \bigoplus \mathcal{U}_{j+1} \bigoplus \dots$ for any $j\geq 0$.

We now recall some definitions and properties from Riesz bases and frames which can be found in \cite{FramesRieszbases}.

\begin{definition}\label{def: riesz}
A collection of vectors $\left\{v_k\right\}$ in a Hilbert space $\mathcal{H}$ is a \textit{Riesz basis} for $\mathcal{H}$ if it is the image of an orthonormal basis for $\mathcal{H}$ under an invertible linear transformation.

A sequence $\left\{v_k\right\}$ in a Hilbert space $\mathcal{H}$, with associated norm $\norm{\cdot}_{\mathcal{H}}$, is a \textit{frame} if there exist numbers $A, B >0$ such that for all $v\in \mathcal{H}$ we have
\begin{equation}\label{eq: frame bounds def}
A\norm{v}_{\mathcal{H}}^2 \leq \sum_k |\langle v,v_k \rangle|^2\leq B\norm{v}_{\mathcal{H}}^2.
\end{equation}
The numbers $A,B$ are called the \textit{frame bounds}. The frame is
\textit{exact} if it ceases to be a frame whenever any single element is deleted from the sequence.
\end{definition}
The following theorem describes the link between frames and Riesz bases.
\begin{theorem}\label{th: riesz is frame}
    A sequence $\left\{v_k\right\}$ in a Hilbert space $\mathcal{H}$ is an exact frame for $\mathcal{H}$ if and only if it is a Riesz basis for $\mathcal{H}$.
\end{theorem}

Theorem \ref{th: from L2 to l2} and Theorem 3.6.6 (ii) of \cite{FramesRieszbases} imply that the normalized B-splines $\mathbf{B}_{k,d,\xi^{(j)}}$, $1\leq k\leq 2^j+d$, in $V_{j}$ form a frame with frame bounds $1$ and $K_d^2$. The following theorem makes it clear why frames are useful in situations where orthonormal bases are more difficult to explicit, as frames have similar properties. For any two bounded self-adjoint operators $U,V$ from a Hilbert space $\mathcal{H}$ to itself, we write $U\preceq V$ if $\langle Uh,h\rangle\leq \langle Vh,h\rangle$ for every $h\in\mathcal{H}$.

\begin{theorem}\label{th: frame operator}
    Given a sequence $\left\{v_k\right\}$ in a Hilbert space $\mathcal{H}$, the following two statements are equivalent:
    \begin{itemize}
        \item $\left\{v_k\right\}$ is a frame with bounds $A,B$.
        \item $Sv = \sum_k  \langle v,v_k \rangle v_k$ defines a bounded linear operator with $AI \preceq S \preceq BI$, for $I$ the identity operator, called the frame operator for $\left\{v_k\right\}$.
    \end{itemize}
\end{theorem}

The frame operator is invertible, with $B^{-1}I \preceq S^{-1} \preceq A^{-1}I$ and every $u$ in the Hilbert space $\mathcal{H}$ can be written $u= \sum_k  \langle u,v_k \rangle S^{-1}v_k= \sum_k  \langle u,S^{-1}v_k \rangle v_k$ (Theorem 5.1.6 \cite{FramesRieszbases}). In addition, $\left\{S^{-1}v_k\right\}$ is a frame, called the \textit{dual frame}, with frame bounds $B^{-1}, A^{-1}$. Also, by Proposition 5.4.4 of \cite{FramesRieszbases}, we have the following bounds on the operator norms

\begin{align}
    \norm{S}&\leq B,\\
    \norm{S^{-1}}&\leq A^{-1}.\label{eq: bound operator norm frame inv}
\end{align}
From Lemma 5.1.5 in \cite{FramesRieszbases}, $S$ and $S^{-1}$ are both self-adjoint.

\subsubsection{Riesz bases of splines}\label{sec: Riesz bases of splines}

Let $j_0\geq1$ be the smallest integer satisfying $2^{j_0}\geq 2d+1$. For $j\geq j_0$, \cite{jia2006stable} constructs $2^j$ wavelets constituting a Riesz basis for the space $\mathcal{U}_j$ introduced above with the frame bounds being independent of $j$. There are $2^j-2d$ inner wavelets and $d$ boundary-corrected wavelets at each end of the interval. These take the form
\begin{equation}\label{eq: spline wavelets}
\psi_{j,i}(\cdot)=
\begin{cases}
    \psi_{j,i}^{\text{left}}(\cdot) =2^{j/2}\psi_i(2^j \cdot ), \quad &i=-d,\dots,-1\\
    \psi_{j,i}(\cdot) = 2^{j/2}\psi(2^j \cdot - i), \quad &i=0,\dots,2^j-2d-1, \\
    \psi_{j,i}^{\text{right}}(\cdot) =\psi_{j,2^j-2d-1-i}^{\text{left}}(1-\cdot), \quad & i=2^j-2d,\cdots,2^j-d-1,\\
\end{cases}
\end{equation}
for some wavelet functions $\psi, \psi_i$ supported inside $[0,2d+1]$. This implies that $\psi_{j,i}$ is itself supported on an interval of the form $[c_i/2^j,d_i/2^j]$, for $|d_i-c_i|\leq 2d+1$. By Theorem \ref{th: riesz is frame}, this Riesz basis of $\mathcal{U}_j$ is a frame and we note $S_j$ its frame operator, with frame bounds $A_d,B_d$ (depending on $d$ only, not $j$).

We recall that $L^2=\mathcal{V}_{j_0} \bigoplus \mathcal{U}_{j_0} \bigoplus \mathcal{U}_{j_0+1} \bigoplus \dots$ and $\mathcal{V}_{j_0}$ admits the normalized B-splines $\mathbf{B}_{k,d,\xi^{(j_0)}}$ as a Riesz basis (see Theorem \ref{th: riesz is frame} and the discussion below it). Up to taking the minimum and maximum of the frame bounds $A_d,B_d$ of the spline wavelets and those of the B-splines basis of $\mathcal{V}_{j_0}$, we consider that $A_d,B_d$ are common frame bounds for each level $j\geq j_0$ of this multiresolution analysis of the $L^2$--space. Also, we can directly check that it is possible to normalize the above wavelets by $\norm{\psi}_{L^2}$ and $\norm{\psi_i}_{L^2}$ respectively to ensure that the wavelets have $L^2$--norms equal to $1$, while keeping the aforementioned properties of the frame. Indeed, recalling \eqref{eq: frame bounds def} from Definition \ref{def: riesz}, we see that the frame bounds for these new normalized wavelets can be replaced with $A_d \max(\norm{\psi}_{L^2}, (\norm{\psi_i}_{L^2})_i)^{-1}$ and $B_d \min(\norm{\psi}_{L^2}, (\norm{\psi_i}_{L^2})_i)^{-1}$, independent of $j$ once again.

To facilitate the reading, we use below the sets of indices $\mathcal{M}_{j_0-1}=[\![ 1,2^{j_0}+d]\!]$ and $\mathcal{M}_{j}=[\![ -p,2^{j}-d-1]\!]$ for $j\geq j_0$. We note \begin{equation}\label{eq: notation spline wav}\psi_{j_0-1,k}\coloneqq \mathbf{B}_{k,d,\xi^{(j_0)}}\quad\text{ for $k\in \mathcal{M}_{j_0-1}$,}\end{equation}$S_{j_0-1}$ the frame operator of this basis and \begin{equation}\label{eq: notation dual spline}\Tilde{\psi}_{j,i}= S_j^{-1} \psi_{j,i}\end{equation} the elements of the dual of the above wavelet basis for $j\geq j_0-1$ and $i\in \mathcal{M}_{j}$. We note that \eqref{eq: norm splines} and \eqref{eq: partition of unity} ensure $\norm{\psi_{j_0-1,k}}_{L^\infty}\leq C_d 2^{-j_0}$ for any $k\in \mathcal{M}_{j_0-1}$.

As the orthogonal spaces $\mathcal{V}_{j_0},\mathcal{U}_{j_0},\dots,\mathcal{U}_{j_n}$ and their spline wavelet bases have frame bounds depending on $d$ only, the union of these bases is a Riesz basis/frame for $\mathcal{V}_{j_n+1}$ with frame bounds depending on $d$ only. Since $\mathcal{V}_{j_0}$ contains all the polynomials of degree $d$ on $[0,1]$, the above wavelets in $\mathcal{U}_j$ with $j\geq j_0$ are orthogonal to all of these such that, for any $f \in \mathcal{B}_{(\infty,l)}(M)$, $1\leq l\leq d+1$,
\begin{align}\label{eq: decay spline wavelet coeff}
    \left|\langle f, \psi_{j,i}\rangle\right| &= \left|\int_0^1 \left(f(x)-\sum_{r=0}^{l-1}\frac{f^{(r)}(0)}{r!}x^r\right) \psi_{j,i}(x) dx\right|\nonumber\\
    &= \left|\int_0^1 \frac{1}{l!}f^{(l)}(\tilde{x})x^{l} \psi_{j,i}(x) dx\right|
    \leq \frac{M}{l!} \int_{c_i/2^j}^{d_i/2^j} x^{l} \left|\psi_{j,i}(x)\right| dx\nonumber\\
    &\leq \frac{M}{l!}2^{-(l+1/2)j} \left(\max_{i=-d,\dots,-1} \norm{\psi_i}_{L^\infty} \vee \norm{\psi}_{L^\infty}\right) \int_0^{2d+1} u^l du  \leq C_d M 2^{-(l+1/2)j},
\end{align}
where we used that $\int_0^1 x^r \psi_{j,i}(x) dx=0$ for $r\leq d$, the Taylor expansion of $f$ around $0$ with exact remainder and some $\tilde{x}\in[0,x]$, and the change of variable $u=2^jx$. We can check that this bound is also valid for $l=0$.

%


\section{Spline plug-in estimator}
\label{sec:SplineEst}

In this section we present a simple private spline projection plug-in estimator of which we can show that it achieves the nonparametric lower bounds of Section~\ref{sec:LowerBounds} on atomic functionals. However, as we will see subsequently (cf. Remark~\ref{rk: failure spline smooth}), it is generally suboptimal for smooth functionals.

For a privacy budget $\alpha>0$, a spline degree $d$ and a resolution level $j$, we let the $i$-th data owner  generate sanitized data $Z_i$ by
\begin{equation}\label{eq: private data}
    Z_i^{(j)} = \begin{bmatrix}
           e_{1,j}(X_i) \\
           \vdots \\
           e_{2^j+d,j}(X_i)
         \end{bmatrix} + \sigma_{\alpha,j} \begin{bmatrix}
             Y_{1,i}\\
             \vdots\\
             Y_{2^j+d,i}
         \end{bmatrix},
\end{equation}
where the functions $e_{k,j}$ defined in \eqref{eq: dual basis splines} satisfy $\|e_{k,j}\|_{L^\infty} \leq C_d2^{j/2}$ for some large enough $C_d$ depending only on $d$ and for independent $Y_{k,i}$ distributed as Laplace random variables with parameter $1$ and $\sigma_{\alpha,j}>0$. Then, letting $j=j_n$ depend on the sample size $n$, we can define estimators of $f$ and $\Lambda(f)$ as
\begin{equation}\label{eq: estimator}
\hat{f}_n(x)= \sum_{k=1}^{2^{j_n}+d} \bar{Z}^{(j_n)}_k\mathbf{B}_{k,d,\xi^{(j_n)}}(x) \quad\text{ and }\quad \Lambda(\hat{f}_n),\end{equation}
where $\bar{Z}^{(j_n)}=\frac1n \sum_{i=1}^n Z_i^{(j_n)}$. 

The objective of the following result is to find an appropriate value of $\sigma_{\alpha,j}$ to ensure local differential privacy of the above mechanism.
\begin{proposition}\label{th: privacy}
    The privacy mechanism \eqref{eq: private data} with $\sigma_{\alpha,j}=C_d\alpha^{-1}2^{j/2}$, for $C_d$ a large enough constant depending on $d$ only, is $\alpha$-locally differentially private, that is, it satisfies \eqref{eq:LDP}.
\end{proposition}
\begin{proof}
The Lebesgue density of the privacy mechanism is equal to
\begin{equation*}
(z_1,\dots, z_{2^{j}+d})\mapsto\prod_{k=1}^{2^{j}+d} \frac{1}{2\sigma_{\alpha,j}} \exp\left( -\frac{\left|e_{k,j}(x)-z_k\right|}{\sigma_{\alpha,j}}\right).
\end{equation*}
Then, reverse and ordinary triangle inequalities yield
\begin{align*}
\frac{dQ(\cdot | x)}{dQ(\cdot |x')}(z_1,\dots, z_{2^{j}+d})&\leq \prod_{k=1}^{2^j+d} \exp\left( \frac{\left|e_{k,j}(x')-z_k\right|-\left|e_{k,j}(x)-z_k\right|}{\sigma_{\alpha,j}}\right) \\
&\leq \exp\left( \sum_{k=1}^{2^j+d} \frac{\left|e_{k,j}(x')\right|+\left|e_{k,j}(x)\right|}{\sigma_{\alpha,j}}\right).
\end{align*}
By definition of $e_{k,j}$, for a fixed $x$, there are at most $d+1$ indices $k$ such that $e_{k,j}(x)$ is not identically null.  Since $|e_{k,j}(x)|\leq C_d 2^{j/2}$, it is sufficient to take $\sigma_{\alpha, j}=2(d+1)C_d\alpha^{-1}2^{j/2}$.
\end{proof}

We consider the plug-in estimator \eqref{eq: estimator} with $\sigma_{\alpha,j_n}$ as in Proposition~\ref{th: privacy} and $j_n$ suitably defined below. See Appendix~\ref{app:Main} for the proof of the following result.

\begin{theorem}\label{th: rate atomic}
    Let $f\in \mathcal{W}_{p}$, $p\in\mathbb{N}^*$, and suppose $\Lambda$ is an atomic functional of index $s$ satisfying \eqref{eq: form functional}. Under Assumption \ref{assum: no singular continuous part} and $p\geq \max(s+1, m+2, 2m-s)$, the plug-in estimator obtained from 
    \eqref{eq: estimator} with $2^{j_n}\leq(n\alpha^2)^{\frac{1}{2p+2}}\wedge n^{\frac{1}{2p+1}} < 2^{j_n+1}$ and $p+1\leq d$ converges at rate 
    \[\underset{f\in\mathcal{W}_p}{\sup}\mathbb{E}_f|\Lambda(f)-\Lambda(\hat{f}_n)|\leq C_{d,M}\, \left[(n\alpha^2)^{-\frac{(p-s)}{2p+2}}\vee n^{-\frac{(p-s)}{2p+1}}\right],\]
    where $C_{d,M}$ is a constant depending on $d$ and $M$ only.
\end{theorem}

Notice that in view of Theorem~\ref{thm:LBatomic} and for bounded $\alpha=\alpha_n\le C$, the rate of Theorem~\ref{th: rate atomic} is, indeed, the fastest possible that any estimator of $\Lambda(f)$ can achieve for an atomic functional $\Lambda$ and a non-degenerate estimation problem $(\mathcal W,\Lambda)$, since $e^\alpha-1\le \alpha\frac{e^C-1}{C}$. However, this is not the case for smooth functionals $\Lambda$, which is why we turn to wavelet estimators in the next section.


\section{Spline wavelet plug-in estimator}
\label{sec:WaveletEst}

In this section we introduce an alternative estimator, along with a new privacy mechanism, that will allow us to obtain both, the optimal rate of Theorem~\ref{th: rate atomic} on atomic functionals and the parametric rate for smooth functionals. In the following, we consider wavelet estimators defined using a spline wavelet basis on the compact space $[0,1]$. Given the compact support property of $\psi, \psi_i$ in \eqref{eq: spline wavelets}, it follows that for any $j\geq j_0$ and $x\in[0,1]$, there are at most $3d+1$ non-zero wavelets at level $j$ and point $x$. Indeed, there are $d$ boundary wavelets ($\psi_{j,i}^{\text{left}}$ near $0$ and $\psi_{j,i}^{\text{left}}$ near $1$, and these have disjoint support since $j\geq j_0$ and $2^{j_0}\geq 2d+1$) and, among the inner wavelets, at most $2d+1$ are non-zero at any point in the unit interval. 
For some level $j_n$ to be defined later and $a>1$, we now define the private data for $\alpha>0$, $1\leq i\leq n$ and $j_0-1\leq j\leq j_n$ as

\begin{equation}\label{eq: private spline data}
    Z_{ijk} = \begin{cases}
        \mathbf{B}_{k,d,\xi^{(j_0)}}(X_i) + \sigma_{\alpha,j_0-1}Y_{i(j_0-1)k}, &\quad\text{if }j=j_0-1,k\in\mathcal{M}_{j_0-1},\\
        \psi_{j,k}(X_i)+\sigma_{\alpha,j} Y_{ijk}, &\quad\text{if }j\geq j_0, k\in\mathcal{M}_j,
    \end{cases}
\end{equation}
where $Y_{ijk}$ are independent Laplace distributed random variables with parameter 1, and
\begin{equation}\label{eq: noise spline private}\sigma_{\alpha,j_0-1}= \frac{4(d+1)^2}{\alpha}2^{j_0/2},\quad \sigma_{\alpha,j}=\frac{4(3d+1)\max(\norm{\psi}_{L^\infty},(\norm{\psi_i}_{L^\infty})_{i=-1,\dots,-d})}{\alpha} \frac{a}{a-1}j^a 2^{j/2},\end{equation}
for $j\geq j_0$. We will see that the spline degree $d$ only affects the constants but not the rates. Thus, for multi-purpose synthetic data release we will choose a large value of $d$ such that the conditions below are satisfied. We show in the following result that this defines $\alpha$-differentially private data.

\begin{proposition}\label{prop:PrivacyWavelets}
    The privacy mechanism defined by \eqref{eq: private spline data} and \eqref{eq: noise spline private} is $\alpha$-locally differentially private in the sense of \eqref{eq:LDP} irrespective of the value of $j_n$.
\end{proposition}

\begin{proof}
By definition, for fixed $i$, the conditional Lebesgue density of $Z_i=(Z_{ijk})_{j_0-1\leq j\leq j_n, k\in\mathcal{M}_j}$ given $X_i = x_i$ can be written as
 \begin{align*}
&\prod_{k\in\mathcal{M}_{j_0-1}}  \frac{1}{2\sigma_{\alpha,j_0-1}} \exp\left( -\frac{\left|\mathbf{B}_{k,d,\xi^{(j_0)}}(x_i)-z_{i(j_0-1)k}\right|}{\sigma_{\alpha,j_0-1}}\right)\\ &\quad\cdot \prod_{j=j_0}^{j_n}\prod_{k\in\mathcal{M}_{j}} \frac{1}{2\sigma_{\alpha, j}} \exp\left( -\frac{\left|\psi_{j,k}(x_i)-z_{ijk}\right|}{\sigma_{\alpha,j}}\right).
\end{align*}
An application of reverse and ordinary triangle inequality leads to
\begin{align*}
\frac{dQ(\cdot |x_i)}{dQ(\cdot |x_i')}(z)&\leq \prod_{k\in\mathcal{M}_{j_0-1}} \exp\left( \frac{\left|\mathbf{B}_{k,d,\xi^{(j_0)}}(x_i')-z_{i(j_0-1)k}\right|-\left|\mathbf{B}_{k,d,\xi^{(j_0)}}(x_i)-z_{i(j_0-1)k}\right|}{\sigma_{\alpha,j_0-1}}\right) \\
&\quad \cdot \prod_{j=j_0}^{j_n}\prod_{k\in\mathcal{M}_{j}} \exp\left( \frac{\left|\psi_{j,k}(x_i')-z_{ijk}\right|-\left|\psi_{j,k}(x_i)-z_{ijk}\right|}{\sigma_{\alpha,j}}\right) \\
\leq \exp&\left( \sum_{k\in\mathcal{M}_{j_0-1}}\frac{\left|\mathbf{B}_{k,d,\xi^{(j_0)}}(x_i')\right|+\left|\mathbf{B}_{k,d,\xi^{(j_0)}}(x_i)\right|}{\sigma_{\alpha,j_0-1}}\right)\\ &\quad\cdot\exp\left( \sum_{j=j_0}^{j_n}\sum_{k\in\mathcal{M}_{j}} \frac{\left|\psi_{j,k}(x_i')\right|+\left|\psi_{j,k}(x_i)\right|}{\sigma_{\alpha,j}}\right).
\end{align*}
We recall that for any $x$ in the unit interval there are at most $d+1$ non-zero splines $\mathbf{B}_{k,d,\xi^{(j_0)}}$ (it follows from (1.6) in Section 1 of \cite{LycheSpline}) and $3d+1$ non-zero wavelets, and $\norm{\mathbf{B}_{k,d,\xi^{(j_0)}}}_{L^\infty} \leq (d+1)2^{j_0/2}$ (unnormalized B-splines are bounded by $1$ and their squared $L^2$ norm are lower bounded by $(d+1)^{-2}2^{-j_0}$). We then have
\begin{align*}
\frac{dQ(\cdot |x_i)}{dQ(\cdot |x_i')}(z)
&\leq \exp\left(  \frac{2(d+1)^2 2^{j_0/2}}{\sigma_{\alpha,j_0-1}}\right)\exp\left( 2(3d+1)\, \max(\norm{\psi}_{L^\infty},(\norm{\psi_i}_{L^\infty})_{i=-1,\dots,-d})\, \sum_{j=j_0}^{j_n}\frac{2^{j/2}}{\sigma_{\alpha,j}}\right)\\
&\leq \exp\left(\frac{\alpha}{2} + \frac{\alpha}{2}\frac{a-1}{a} \sum_{j=j_0}^{j_n} j^{-a}\right)\leq \exp(\alpha).
\end{align*}
In the last line, we used that $\sum_{j=j_0}^{j_n} j^{-a} \leq \sum_{j=1}^{\infty} j^{-a}$ and $ \sum_{j=2}^{\infty} j^{-a}\leq \int_1^{+\infty} t^{-a} dt\leq (a-1)^{-1}$, so that $\sum_{j=j_0}^{j_n} j^{-a}\leq 1 + (a-1)^{-1} = a/(a-1)$.
\end{proof}

With these private data, we can use notation introduced in Section \ref{sec: splines wav} and define private empirical wavelet coefficients $\bar{Z}_{jk}=\frac1n \sum_{i=1}^n Z_{ijk}$ and plug-in estimator
\begin{equation}\label{eq: plug_in wavelet estimator}
\hat{f}_n(x)= \sum_{k\in\mathcal{M}_{j_0-1}} \bar{Z}_{(j_0-1)k}\Tilde{\psi}_{j_0-1,k}(x) + \sum_{j=j_0}^{j_n}\sum_{k\in\mathcal{M}_{j}} \bar{Z}_{jk} \Tilde{\psi}_{j,k}(x) \quad\text{ and }\quad \Lambda(\hat{f}_n),
\end{equation}
where we used the dual spline wavelets from \eqref{eq: notation dual spline}. This estimator can now be shown to be optimal on both atomic and smooth functionals. See Section~\ref{app:Main} for the proofs of the following results.

\begin{theorem}\label{th: atomic wav}
    Let $f\in \mathcal{W}_{p}$, $p\in\mathbb{N}^*$ and suppose $\Lambda$ is an atomic functional of index $s$ satisfying \eqref{eq: form functional}. Under Assumption \ref{assum: no singular continuous part} and $p\geq \max(s+1, m+2, 2m-s)$, the plug-in estimator obtained from 
    \eqref{eq: plug_in wavelet estimator} with $2^{j_n}\leq (n\alpha^2\log^{-2a} n)^{\frac{1}{2p+2}}\wedge n^{\frac{1}{2p+1}} < 2^{j_n+1}$ and $p+1\leq d$ satisfies
    \[\underset{f\in\mathcal{W}_p}{\sup}\mathbb{E}_f|\Lambda(f)-\Lambda(\hat{f}_n)|\leq C_{d,M}\,\left[ (n\alpha^2 \log^{-2a} n)^{-\frac{(p-s)}{2p+2}}\vee n^{-\frac{(p-s)}{2p+1}}\right],\]
    where $C_{d,M}$ is a constant depending on $d$ and $M$ only.
\end{theorem}

The spline wavelet plug-in approach does, indeed, achieve the lower bound of Theorem~\ref{thm:LBatomic} and is thus rate optimal. However, the optimal choice of the resolution level $j_n$ still depends on the unknown smoothness $p$ of the density $f$. Thus, we develop an adaptive procedure in Section~\ref{sec:Adaptation} below.

On smooth functionals we obtain the optimal parametric rate.

\begin{theorem}\label{th: rate smooth}
    Let $f\in \mathcal{W}_{p}$, $p\in\mathbb{N}^*$, and suppose $\Lambda$ is a smooth functional satisfying \eqref{eq: form functional} with $m\geq0$, such that for some $M^\prime>0$, $\omega_f\in \mathcal{B}_{(\infty,1)}(M^\prime)$ satisfies \eqref{eq: bound deriv omega}. Under Assumption \ref{assum: no singular continuous part} and $2m+2< p\leq d-1$ the plug-in estimator obtained from \eqref{eq: plug_in wavelet estimator} with $a'>1/4$ and 
\begin{equation} \label{eq:SmoothJn}
\begin{split}
\left(n\wedge(n\alpha^2)\right)^{1/2p}&\leq 2^{j_n} \leq \left[\log^{-a/(m+1)}(n\alpha^2)(n\alpha^2)^{1/(4m+4)}\right]\wedge n^{1/(4m+3)}\text{ if }m>0, \\
\left(n\wedge(n\alpha^2)\right)^{1/2p}&\leq 2^{j_n} \leq \left[\log^{-a}(n\alpha^2)(n\alpha^2)^{1/4}\right]\wedge \left[\log^{-a'}(n)n^{1/4}\right]\qquad\text{ if }m=0,
\end{split}
\end{equation}
satisfies 
\[\underset{f\in\mathcal{W}_p}{\sup}\mathbb{E}_f|\Lambda(f)-\Lambda(\hat{f}_n)|\leq C_{a,d,M}\, \left[n^{-1/2}\vee(n\alpha^2)^{-1/2}\right],\]
up to a constant depending on $a$, $d$, and $M$. Above, we can also take $a'=0$ whenever $m\neq 0$.
\end{theorem}

It is important to notice that the upper bound in \eqref{eq:SmoothJn} of Theorem~\ref{th: rate smooth} does not depend on the unknown smoothness $p$ of the density $f$ and choosing $j_n$ to be of that order leads to an optimal procedure. Hence, Theorem~\ref{th: rate smooth} offers already a fully data driven estimation procedure and there is no need for adaptive results in the smooth case.

\begin{remark}\normalfont
\label{rk: failure spline smooth}
    So far, we obtained results for smooth and atomic functionals with plug-in estimators \eqref{eq: plug_in wavelet estimator} based on spline wavelets. However, for plug-in estimators \eqref{eq: estimator} based on splines we only proved convergence rates for atomic functionals. The reason is that the spline estimator does not offer a good enough control of the variance coming from the Laplace noise in the privacy mechanism \eqref{eq: private data}. An inspection of the proof of Theorem \ref{th: rate smooth} shows that parametric rates are achieved in the smooth case because the term $V_{2,j_n}$ is upper-bounded by a constant free of $n$. This is true when spline wavelets are used because they leverage the smoothness of $\omega_f$ in that $\left|\langle \omega_f, \Tilde{\psi}_{j,k}\rangle\right|$ is decreasing exponentially fast in $j$ (see \eqref{eq: decay spline wavelet coeff}). An equivalent expression for $V_{2,j_n}$ with spline-based estimators would be
    \[V_{2,j_n}=2\sigma_{\alpha,j}^2\sum_{k=1}^{2^{j_n}+d} T_f(\mathbf{B}_{k,d,\xi^{(j_n)}})^2=2\sigma_{\alpha,j}^2\sum_{k=1}^{2^{j_n}+d} \left|\langle \omega_f, \mathbf{B}_{k,d,\xi^{(j_n)}}\rangle\right|^2.\]
    Since the normalized B-splines $\mathbf{B}_{k,d,\xi^{(j_n)}}$ form a frame with bounds independent of $n$, the above display is lower bounded by $C\sigma_{\alpha,j}^2\norm{\omega_f}_{L^2}^2=
    \Tilde{C}\alpha^{-2}2^{j_n}\norm{\omega_f}_{L^2}^2$. Therefore, this term cannot be bounded, and B-splines cannot leverage the assumed regularity of $\omega_f$ to counterbalance the variance of the Laplace noise. Finally, we note that it is independent of the scaling chosen for the B-splines, since multiplying $\mathbf{B}_{k,d,\xi^{(j_n)}}$ by a constant $c_n$ results in $e_{k,j_n}$ in \eqref{eq: dual basis splines} being replaced with $e_{k,j_n}/c_n$ and $\sigma_{\alpha,j_n}$ being replaced with $\sigma_{\alpha,j_n}/c_n$ in Proposition~\ref{th: privacy}. Therefore, the above upper bound would remain unchanged.
\end{remark}

\section{Adaptation with spline wavelets}
\label{sec:Adaptation}

As seen in Theorem \ref{th: atomic wav}, the optimal threshold is $2^{j_n}\asymp (n\alpha^2\log^{-2a}n)^{\frac{1}{2p+2}}\wedge n^{\frac{1}{2p+1}}$ for atomic functionals (not depending on the smoothness $s$ of the functional). This, however, depends on the regularity of the unknown truth $f_0\in \mathcal{W}_{p}$, which is typically not known in practice. Therefore, one has to provide a data driven choice of this tuning parameter which results in optimal convergence rates for all smoothness classes, simultaneously. For this we adapt Lepski's method for this privacy constrained setting.

We start by deriving some bounds following from the non-adaptive analysis. First, as a direct consequence of \eqref{eq: bound l2 sum adj wav} (see Appendix \ref{app:Main}),
\begin{align}
 \norm{\tilde\psi_{jk}^{(s+1)}}_{L^\infty}^2\leq \Tilde{C}_{d,s} 2^{j(2s+3)}\label{UB:basis:deriv}
\end{align}
for some universal constant $\Tilde{C}_{d,s}$. Furthermore, in view of \eqref{eq: supnorm bias proj} in Appendix \ref{app:Main}, we have for all $s\in\{0,1,..,\lfloor p\rfloor\}$ that, for $P_{\mathcal{V}_j} f$ the orthogonal projection of $f$ on $\mathcal{V}_j$,
\begin{align}
\|(f-P_{\mathcal{V}_j} f)^{(s)}\|_{L^\infty}\leq C_{p,d} 2^{-(p-s) j}\|f^{(p)}\|_{L^{\infty}}=:B(j,f,s,p).\label{eq:def:B}
\end{align}
Also note that
\begin{align*}
\|( f-P_{\mathcal{V}_j} f)^{(s)}\|_{L^2}\leq \| (f-P_{\mathcal{V}_j} f)^{(s)}\|_{L^\infty}\leq B(j,f,s,p).
\end{align*}
In view of \eqref{variance stim sp} and \eqref{variance stim sp derivatives} in Appendix \ref{app:Main}, note that for all $s\in\{0,1,..,\lfloor p\rfloor\}$
\begin{align}
\E\|(\hat{f}_{n}^{j})^{(s)}-\E (\hat{f}_n^{j})^{(s)}\|_{L^2}^2\lesssim n^{-1}2^{2js}(2^{j}+2^{2j}j^{2a}\alpha^{-2}).\label{eq:UB:var}
\end{align}
Then the optimal, oracle choice of the thresholding parameter is
\begin{align}
j_n^*=\min\{j\in\mathcal{J}: B(j,f,s,p)^2\leq n^{-1}2^{2js}j(2^{j}+2^{2j}j^{2a}\alpha^{-2}), \forall s\in\{0,1,..,\lfloor p\rfloor\}\},\label{def:j*}
\end{align}
where $\mathcal{J}=\{1,...,j_{\max}\}$ with $j_{\max}=(\log n)/3$.\footnote{We note that this upper bound follows from the assumption $p\geq 1$. If this assumption is relaxed one has to choose a larger $j_{\max}$.} Furthermore, note that we have an extra $j$ multiplier on the right hand side, which will result in a logarithmic factor loss in the proof. Finally, note that the above $j_n^*$ satisfies that 
\begin{equation}
2^{j^*_n}\asymp (n\alpha^2\log^{-(1+2a)} n)^{1/(2+2p)}\wedge (n/\log n)^{1/(1+2p)},\label{eq:asymp:oracle}
\end{equation}
 and as a consequence
\begin{align}
 n^{-1}2^{2j_n^*s}\big(2^{j_n^*}+2^{2j_n^*}(j_n^*)^{2a}\alpha^{-2}\big)\lesssim (n/\log n)^{-\frac{p-s}{2p+1}}\vee  (n\alpha^2\log^{-(1+2a)} n)^{-\frac{p-s}{2p+2}}.\label{ub:var:jn*}
\end{align}

This oracle choice $j_n^{*}$, however, depends on $B(j,f,s,p)$ which in turn depends on the unknown regularity $p$. Hence we consider a data driven choice for the optimal threshold $j^*_n$. We do not consider an empirical version of the supremum norm as it would provide an additional logarithmic factor, while we are interested in pointwise and $L^2$-concentration rates in the first place. Furthermore, controlling the variance in supremum norm of our estimator is technically highly challenging. In our proof one of the key step is to achieve the optimal concentration rate for arbitrary, unspecified point $x\in[0,1]$. Since the estimator can not depend on the unknown $x$, we introduce a fine enough grid and define the estimator on it. This turns out to be sufficiently accurate approach to provide optimal estimation rate at arbitrary point $x$.

Let us introduce the grid $x_t=t/M_n$, $t=0,...,M_n$, for $M_n\gtrsim n^{4/3}$ and define the estimator as
\begin{align}
\hat{j}_n=\min \Big\{ j\in\mathcal{J}:\,& \| (\hat{f}_n^{j}-\hat{f}_n^l)^{(s)}\|_{L^2}^2\vee \max_{x_t,\, t=0,...,M_n}|(\hat{f}_n^{j}- \hat{f}_n^l )^{(s)}(x_t)|^2\nonumber\\
&\leq \tau  n^{-1}2^{2ls}l (2^{l}+2^{2l}l^{2a}\alpha^{-2}),
 \qquad\forall l>j,\, l\in\mathcal{J},\forall s\in\{0,1,..,\lfloor p\rfloor\} 
 \Big\},\label{def:hat:jn}
\end{align}
with $\tau= C_0c_{a,p,s,L,\psi}$ and $C_0$ as in Lemma \ref{lem:overshoot}, for a large enough universal constant $c_{a,p,s,L,\psi}$ to be specified in the proof.
Then we define our adaptive estimator as
\begin{align}
\hat{f}_n(x)=\hat{f_n}^{\hat{j}_n}(x).\label{def:Lepski}
\end{align}
The theorem below shows that this estimator achieves the minimax concentration rate both pointwise and for the $L^2$--norm.

Before we state the theorem, we provide a lemma that controls the probability of overfitting by providing exponential upper bound for selecting larger thresholds $\hat{j}_n$ than $j_n^*$. In the proof we apply Bernstein's inequality, similarly to  Proposition B.1 of \cite{ButuceaPrivacy}. In our setting, however, we work with a Riesz (not orthogonal) basis, we provide simultaneously $L^2$ and pointwise and not supremum norm control on the error (and hence do not lose an additional $\log n$ factor), consider the estimation of the whole function, not a specific functional and derive results for the (up to $s$th order) derivative of the functions as well, making our proofs substantially different and more involved. See Section~\ref{sec:app:adaptation} for the proofs of the results in this section.

\begin{lemma}\label{lem:overshoot}
For $f\in\mathcal{W}_{p}$, $p\in\mathbb{N}$, and $j>j_n^*$ we have for $M_n\gtrsim n^{4/3}$  that
\begin{align*}
\mathbb{P}(\hat{j}_n=j)\leq  2^6j_{\max} (M_n+1)( p+1) e^{-(C_0/2)j},
\end{align*}
for  arbitrarily large parameter $C_0>2\log 2$ (given in the definition of $\tau$ in \eqref{def:hat:jn}). 
\end{lemma}

\begin{theorem}\label{thm:lepski}
The estimator $\hat{f}_n$ given in \eqref{def:Lepski} satisfies that for all $s+1\leq p$, $p\in\mathbb{N}$,
\begin{align*}
\sup_{f\in\mathcal{W}_{p}}\, &\left[\E_f\norm{(\hat{f}_n- f)^{(s)}}_{L^2}^2\vee \sup_{x\in[0,1]} \E_f\left|(\hat{f}_n-f)^{(s)}(x)\right|^2\right]\\
& \lesssim (n\alpha^2\log^{-(1+2a)} n)^{-\frac{2(p-s)}{2p+2}}\vee (n/\log n)^{-\frac{2(p-s)}{2p+1}},
\end{align*}
up to a constant that depends on $a$, $d$ and $M$ only.
\end{theorem}

This is now easily turned into a tight upper bound on the risk of the plug-in estimator in the case of an atomic functional $\Lambda$.

\begin{theorem}\label{th: adaptive rate semiparam splwav}
Let $f_0\in \mathcal{W}_{p}$, $p\in\mathbb{N}^*$, be such that $\|f_0\|_{L^\infty}\leq L$ and suppose that the operator $\Lambda$ satisfies Assumption~\ref{assum: no singular continuous part} for $m,s\geq 0$ and $p\geq \max(s+1,m+2,2m-s)$, where $T_f(h)=\sum_{j=1}^s \int h^{(j)}d\mu_j$, $\mu_s$ with discrete component. Then the plug in estimator $\Lambda(\hat{f}_n)$ with $\hat{f}_n$ given in \eqref{def:Lepski} satisfies that
\begin{align*}
\E_f|\Lambda(\hat{f}_n)- \Lambda(f_0)|\lesssim (n\alpha^2\log^{-(1+2a)} n)^{-\frac{p-s}{2p+2}}\vee (n/\log n)^{-\frac{p-s}{2p+1}},
\end{align*}
up to a constant that depends on $a$, $d$ and $M$ only.
\end{theorem}

\begin{proof}
From \eqref{eq: Taylor expansion Expectation} and Lemma \ref{lem: exponential decay outer taylor exp}, for any $c>0$
\[\E_f|\Lambda(\hat{f}_n)- \Lambda(f_0)|\leq C\left(\mathbb{E}\left[|T_f(f-\hat{f})|\right]+ \mathbb{E}\left[\norm{\hat{f}_n-f}_{(2,m,\lambda)}^2\right]\right)+o(n^{-c}).\]
Then, as in the proof of Theorem \ref{th: rate atomic},
\[\mathbb{E}\left[|T_f(f-\hat{f})|\right]\leq C_{a,d,m} \sum_{j\leq s} \left[\E_f\norm{(\hat{f}_n- f)^{(s)}}_{L^2}\vee \sup_{x\in[0,1]} \E_f\left|(\hat{f}_n-f)^{(s)}(x)\right|\right],\]
which is bounded by $(n\alpha^2\log^{-(1+2a)} n)^{-\frac{p-s}{2p+2}}\vee (n/\log n)^{-\frac{p-s}{2p+1}}$ according to Theorem \ref{thm:lepski}, since $p\geq s+1$.
\begin{align*}
\E \Big| \int_0^1  (\hat{f}_n-f_0)^{(j)}d\mu_j\Big|^2\leq |\mu_j|([0,1])\Big(\int_0^1 \E | (\hat{f}_n-f_0)^{(j)}|^2d\mu_{j,\text{cont}}+ \int_0^1 \E | (\hat{f}_n-f_0)^{(j)}|^2d\mu_{j,\text{pp}}\Big).
\end{align*}
Similarly, 
\[\mathbb{E}\left[\norm{\hat{f}_n-f}_{(2,m,\lambda)}^2\right]\leq C_{a,d,M}\left[(n\alpha^2\log^{-(1+2a)} n)^{-\frac{2(p-m)}{2p+2}}\vee (n/\log n)^{-\frac{2(p-m)}{2p+1}}\right],\]
which is smaller than $(n\alpha^2\log^{-(1+2a)} n)^{-\frac{p-s}{2p+2}}\vee (n/\log n)^{-\frac{p-s}{2p+1}}$ since $p\geq 2m-s$. This concludes the proof.
\end{proof}

\section{Examples}
\label{sec:Examples}

\subsection{$\Lambda(f)=f^{(r)}(x_0)$}

The point evaluation functional is a differentiable functional of order $m=0$ and index $s=r$ since
\[\left(f+h\right)^{(r)}(x_0) = f^{(r)}(x_0) + \int h^{(r)}\delta_{x_0}.\]
Whenever $f\in\mathcal{W}_{p}$ for $p\leq d+1$, the plug-in estimators from Theorem \ref{th: rate atomic} and Theorem \ref{th: atomic wav}, with $2^{j_n} = \lfloor(n\alpha^2)^{\frac{1}{2p+2}}\wedge n^{\frac{1}{2p+1}}\rfloor$ and $2^{j_n} = \lfloor(n\alpha^2\log^{-2a} n)^{\frac{1}{2p+2}}\wedge n^{\frac{1}{2p+1}}\rfloor$, respectively, converge at the rates $(n\alpha^2)^{-\frac{(p-r)}{2p+2}}\vee n^{-\frac{(p-r)}{2p+1}}$ and $(n\alpha^2 \log^{-2a} n)^{-\frac{(p-r)}{2p+2}}\vee n^{-\frac{(p-r)}{2p+1}}$ if $\max(2,r+1)\leq p\leq d+1$. We note that the same rate is achieved for $\Lambda(f)= g(f^{(r)}(x_0))$, with $g$ twice differentiable and $g'\neq 0$ according to Proposition 3.9 in \cite{GoldsteinMesser92}.

\subsection{$\Lambda(f)=\int_0^1|f^{(m)}|^qd\lambda$}

According to Proposition 3.7 in \cite{GoldsteinMesser92}, the $q^{\text{th}}$ power of the $L_q$-norm is a differentiable functional of order $m$ for $q\geq 2$ on $\mathcal{W}_{p}$, $p\geq m$ (see also Example 2 from this reference). For instance, if $q$ is even, we find that
\[\Lambda(f+h) = \Lambda(f) +q\int_0^1 h^{(m)}(f^{(m)})^{q-1}d\lambda + \sum_{l=2}^{q} \binom{q}{l}\int_0^1(h^{(m)})^l(f^{(m)})^{q-l}d\lambda.\]
The last sum is $O(\norm{h}^2_{(2,m,\lambda)})$ for $\norm{h}_{(\infty,m)}$ small enough, where $\lambda$ is the Lebesgue measure. Therefore, for $m=0$, we see that this functional is smooth and the derivative satisfies $\omega_f = qf^{q-1}$. We apply Theorem \ref{th: rate smooth} and obtain an (undersmoothed) plug-in estimator which converges at rate $n^{-1/2}\vee(n\alpha^2)^{-1/2}$ whenever $p> 2$. Under this last condition, we indeed always have $\omega_f\in\mathcal{B}_{(\infty,1)}(M^\prime)$ for some $M^\prime$ depending on $M$ and $q$.

For $q=2$ and $0<m<p$, we integrate $T_f(h)$ by part to see that this is an atomic functional of index $s=m-1$ as 
\begin{align*}
T_f(h) &= 2\int_0^1 h^{(m)}f^{(m)}d\lambda \\
&=  2h^{(m-1)}(1)f^{(m)}(1) - 2 h^{(m-1)}(0)f^{(m)}(0) - \int_0^1 h^{(m-1)}f^{(m+1)}d\lambda.
\end{align*}
From Theorem \ref{th: rate atomic} and Theorem \ref{th: atomic wav}, plug-in estimators with $2^{j_n} = \lfloor(n\alpha^2)^{\frac{1}{2p+2}}\wedge n^{\frac{1}{2p+1}}\rfloor$ and $2^{j_n} = \lfloor(n\alpha^2\log^{-2a} n)^{\frac{1}{2p+2}}\wedge n^{\frac{1}{2p+1}}\rfloor$, respectively, converge at the rates $(n\alpha^2)^{-\frac{(p-m+1)}{2p+2}}\vee n^{-\frac{(p-m+1)}{2p+1}}$ and $(n\alpha^2 \log^{-2a} n)^{-\frac{(p-m+1)}{2p+2}}\vee n^{-\frac{(p-m+1)}{2p+1}}$ if $m+2\leq p\leq d+1$ and $f\in\mathcal{W}_{p}$.

\subsection{$\Lambda(f)=2\int_0^1fg/(f+g)d\lambda$}
We assume here that $g$ is a density satisfying $\inf_{0\leq x\leq 1} g(x)>0$ and $\norm{g}_{(\infty,1)}<1$. We can check that for $h$ such that $\norm{h}_{(\infty, 0)}$ is small enough
\begin{align*}
\Lambda(f+h) &=2\int_0^1\frac{fg}{f+h+g}d\lambda + 2\int_0^1\frac{hg}{f+h+g}d\lambda\\
&= \Lambda(f) - 2\int_0^1\frac{hfg}{(f+g)^2}d\lambda + 2\int_0^1\frac{h^2fg}{(f+g)^2(f+g+h)} d\lambda+ 2\int_0^1\frac{hg}{f+g+h}d\lambda\\
&= \Lambda(f) + 2\int_0^1\frac{hg^2}{(f+g)^2}d\lambda + 2\int_0^1\frac{h^2fg}{(f+g)^2(f+g+h)}d\lambda - 2\int_0^1\frac{h^2g}{(f+g)(f+g+h)}d\lambda\\
&= \Lambda(f) + 2\int_0^1\frac{hg^2}{(f+g)^2}d\lambda +O\left(\norm{h}^2_{(2,0,\lambda)}\right),
\end{align*}
for $\lambda$ the Lebesgue measures. Therefore, $\Lambda$ is smooth of order $m=0$ with $T_f(h)=\int h\omega_f$ and $\omega_f=2 \frac{g^2}{(f+g)^2}\in\mathcal{B}_{(\infty,1)}(M^\prime)$, for some $M^\prime$ depending on $M$, $\norm{g}_{(\infty,1)}$ and $\inf_{0\leq x\leq 1} g(x)$. The (undersmoothed) plug-in estimator from Theorem \ref{th: rate smooth} then achieves the rate $n^{-1/2}\vee(n\alpha^2)^{-1/2}$ as soon as $f\in\mathcal{W}_{p}$ with $p> 2$.

\subsection{$\Lambda(f)=\int_0^1f\log fd\lambda$}

We assume here that $\inf_{0\leq x\leq 1} f(x)>0$, so that for $\norm{h}_{(\infty, 0)}$ small enough,
\begin{align*}
\Lambda(f+h) &=\int_0^1f\log(f+h)d\lambda + \int_0^1h\log(f+h)d\lambda\\
&= \Lambda(f) + \int_0^1f\log(1+h/f)d\lambda + \int_0^1h\log(f)d\lambda + \int_0^1h\log(1+h/f)d\lambda\\
&= \Lambda(f) + \int_0^1h\left(1+\log(f)\right)d\lambda + \int_0^1f\left(\log\left(1+\frac{h}{f}\right)-\frac{h}{f}\right)d\lambda  + \int_0^1h\log\left(1+\frac{h}{f}\right)d\lambda\\
&=\Lambda(f) + T_f(h) + O\left(\norm{h}^2_{(2,0,\lambda)}\right),
\end{align*}
where $T_f(h) = \int_0^1h\left(1+\log(f)\right)$. Therefore, $\Lambda$ is smooth of order $m=0$. The (undersmoothed) plug-in estimator from Theorem \ref{th: rate smooth} then achieves the rate $n^{-1/2}\vee(n\alpha^2)^{-1/2}$ as soon as $f\in\left\{g\in\mathcal{W}_{p}\colon \inf_{0\leq x\leq 1} g(x)>0 \right\}$ with $p> 2$. Indeed, $1+\log(f) \in\mathcal{B}_{(\infty,1)}(M^\prime)$ for some $M^\prime$ depending on $M$ and $\inf_{0\leq x\leq 1} f(x)$.

\subsection{$\Lambda(f)=\int_0^1(f')^2/fd\lambda$}

We may rewrite, if $\inf_{0\leq x\leq 1} f(x)>0$,

\begin{align*}
    \Lambda(f+h) &= \int_0^1\frac{(f' + h')^2}{f+h}d\lambda\\
    &= \Lambda(f) -  \int_0^1\frac{h(f')^2}{f(f+h)}d\lambda + 2\int_0^1 \frac{f'h'}{f}d\lambda + \int_0^1 \frac{h'^2}{f}d\lambda\\
    &\quad -\int_0^1 \frac{2hh'f'}{f(f+h)}d\lambda-\int_0^1 \frac{hh'^2}{f(f+h)}d\lambda\\
    &= \Lambda(f) - \int_0^1\frac{h(f')^2}{f^2}d\lambda + 2\int_0^1 \frac{f'h'}{f}d\lambda + O\left(\norm{h}^2_{(2,1,\lambda)}\right)\\
    &= \Lambda(f) + T_f(h) + O\left(\norm{h}^2_{(2,1,\lambda)}\right), 
\end{align*}
where the remainder is valid for $\norm{h}_{(\infty, 1)}$ small enough, and where $T_f(h) = - \int_0^1\frac{h(f')^2}{f^2}d\lambda + 2\int_0^1 \frac{f'h'}{f}d\lambda$. In addition,
\[\int_0^1 \frac{f'h'}{f}d\lambda = h(1) \frac{f'(1)}{f(1)}  - h(0) \frac{f'(0)}{f(0)} - \int_0^1 h \frac{f''f-(f')^2}{f^2}d\lambda,\]
so that
\[T_f(h) = 2h(1) \frac{f'(1)}{f(1)}  - 2h(0) \frac{f'(0)}{f(0)} + \int_0^1\frac{h(f')^2}{f^2} - 2\int_0^1h\frac{f''}{f}.\]
We then see, that $\Lambda$ is an atomic functional of order $m=1$ and index $s=0$. From Theorem \ref{th: rate atomic} and Theorem \ref{th: atomic wav}, plug-in estimators with $2^{j_n} = \lfloor(n\alpha^2)^{\frac{1}{2p+2}}\wedge n^{\frac{1}{2p+1}}\rfloor$ and $2^{j_n} = \lfloor(n\alpha^2\log^{-2a} n)^{\frac{1}{2p+2}}\wedge n^{\frac{1}{2p+1}}\rfloor$, respectively, converge at the rates $(n\alpha^2)^{-\frac{p}{2p+2}}\vee n^{-\frac{p}{2p+1}}$ and $(n\alpha^2 \log^{-2a} n)^{-\frac{p}{2p+2}}\vee n^{-\frac{p}{2p+1}}$ if $3\leq p\leq d+1$ and $f\in\mathcal{W}_{p}$ such that $\inf_{0\leq x\leq 1} f(x)>0$.


%
\begin{appendix}
\section{Auxiliary results and proofs of Section~\ref{sec:LowerBounds} on lower bounds}
\label{app:LB}

\begin{lemma}\label{lemma:W0}
    Let $\mathcal W^\circ := \{f\in\mathcal W_p : f>0\}$ and $\Lambda:\mathcal W_p \to \mathbb R$ be differentiable as in \eqref{eq: form functional} and atomic of index $s$, $0\le s\le p$. Then $(\mathcal W^\circ,\Lambda)$ is non-degenerate as in Definition~\ref{def:ND}.
\end{lemma}
\begin{proof}
    Define $g(x):= \exp\left(-\frac{1}{1-4x^2} \right)$,  if $x\in(-\frac12,\frac12)$ and $g(x) := 0$, else, which is infinitely many times differentiable on $\R$. Notice that there must be an $x^*\in(-\frac12,\frac12)$, such that $g^{(s)}(x^*)\neq 0$, or otherwise $g$ is a polynomial on that interval. By symmetry we may assume that $x^*\in(-\frac12,0]$. For $b>0$, set $\varphi(x) := -g(x+\frac12) + bg(x-\frac12)$, which has compact support $[-1,1]$. If $x_0\in(0,1)$, we take $b=1$. If $x_0=0$, we take $b=(\int_{x^*}^{1/2}g(u)du )/(\int_{-1/2}^{1/2} g(u)du )$. Finally, if $x_0=1$, we set $b=(\int_{-1/2}^{1/2} g(u)du )/(\int_{-1/2}^{-x^*}g(u)du )$. Furthermore, if $x_0\in[0,1)$, define $\psi(x) := \varphi(x+x^*-\frac12)$ and if $x_0=1$ set it equal to $\psi(x) := \varphi(x-x^*+\frac12)$. 
    In any case, this entails that $\psi$ is compactly supported with $\psi\in C^p(\R)$ and $\psi^{(s)}(0) \neq 0$. Take $\delta_0=(x_0\land(1-x_0))/2$ if $x_0\in(0,1)$ and set $\delta_0=\frac12$ if $x_0\in\{0,1\}$. Now fix $f_0\in \mathcal W^\circ$, which entails that $\inf_{x\in[0,1]}f_0(x)>0$, and $\|f_0\|_{(\infty,p)}<M$, and define $\eta_0 = \left(\frac{\inf_{x\in[0,1]}f_0(x)}{\|\psi\|_\infty} \right)\land\left( \frac{M-\|f_0\|_{(\infty,p)}}{\|\psi\|_{(\infty,p)}}\right)>0$. For $0<\eta<\eta_0$, $0<\delta<\delta_0$ and $x\in[0,1]$ consider 
    $$
    f_1(x) := f_0(x) + \eta\delta^p\psi((x-x_0)/\delta).
    $$
    By construction, we have $\inf_{x\in[0,1]}f_1(x)\ge \inf_{x\in[0,1]}f_0(x) - \eta\|\psi\|_{L^\infty}>0$ and $\|f_1\|_{(\infty,p)} \le \|f_0\|_{(\infty,p)} + \eta\|\psi\|_{(\infty,p)} < M$. Clearly $f_1\in C^p[0,1]$. In the case $x_0\in[0,1)$, notice that
    \begin{align*}
        \int_0^1 \psi\left( \frac{x-x_0}{\delta}\right)dx
        &=\delta \int_{-\frac{x_0}{\delta}}^{\frac{1-x_0}{\delta}} \psi(u)du
        =\delta
        \int_{-\frac{x_0}{\delta}+x^*-\frac12}^{\frac{1-x_0}{\delta}+x^*-\frac12} \varphi(u)du\\
        &=
        -\delta \int_{-\frac{x_0}{\delta}+x^*}^{\frac{1-x_0}{\delta}+x^*}g(u)du + \delta b \int_{-\frac{x_0}{\delta}+x^*-1}^{\frac{1-x_0}{\delta}+x^*-1}g(u)du.
    \end{align*}
    If we specialize to the case $x_0\in(0,1)$, we observe that the integral limits contain the support of $g$ whenever $\delta<(x_0\land(1-x_0))/2$, and since $b=1$ in that case, we see that the expression in the previous display is equal to zero. In the case $x_0=0$ we see that for $\delta<\frac12$ the expression in the previous display reads
    $$
    -\delta \int_{x^*}^{\frac{1}{\delta}+x^*}g(u)du + \delta b \int_{x^*-1}^{\frac{1}{\delta}+x^*-1}g(u)du
    =
    -\delta \int_{x^*}^{\frac12}g(u)du + \delta b \int_{-\frac12}^{\frac12}g(u)du = 0,
    $$
    in view of our definition of $b$.
    A similar consideration for the case $x_0=1$ reveals that
    \begin{align*}
        \int_0^1 \psi\left( \frac{x-x_0}{\delta}\right)dx
        &=\delta \int_{-\frac{1}{\delta}}^{0} \psi(u)du
        =\delta
        \int_{-\frac{1}{\delta}-x^*+\frac12}^{-x^*+\frac12} \varphi(u)du\\
        &=
        -\delta \int_{-\frac{1}{\delta}-x^*+1}^{-x^*+1}g(u)du + \delta b \int_{-\frac{1}{\delta}-x^*}^{-x^*}g(u)du.
    \end{align*}
    Thus, if $\delta<\frac12$, this equals 
    $$
    -\delta \int_{-\frac12}^{\frac12}g(u)du + \delta b \int_{-\frac12}^{-x^*}g(u)du = 0
    $$
    in view of our choice of $b$ in that case.
    Hence, in all the cases, we have shown that $f_1\in\mathcal W^\circ$.
\end{proof}

\subsection{Proof of Theorem~\ref{thm:LBparam}}

We make use of Corollary~3.1 in \cite{Rohde20}. Therefore, we recall the total variation modulus of continuity of $\Lambda$ on $\mathcal W$, defined by $\omega_{TV}(\eps) := \sup\{|\Lambda(f_0) - \Lambda(f_1)| : f_0,f_1\in\mathcal W, \frac12\int_0^1 |f_0(x)-f_1(x)|dx\le \eps\}$. Now, from the cited corollary, we get
\begin{equation}\label{eq:RohdeSteinb}
\inf_{Q\in\mathcal Q_\alpha} \inf_{\hat{\Lambda}_n} \sup_{f\in\mathcal W} \E_{QP_f}\left[|\hat{\Lambda}_n - \Lambda(f)|\right]
\ge \frac18 \omega_{TV}\left(\left[\frac{1}{\sqrt{8n(e^\alpha-1)^2}}\right]^-\right),
\end{equation}
for all $n>\log 2$, where $\omega_{TV}(\eps^-)$ denotes the limit from the left.

In order to lower bound the total variation modulus, we first show that there exist two distinct densities $f,g\in\mathcal W$, such that $T_{f}(f) \neq T_{f}(g)$, where $T_{f}$ is the linear functional derivative of $\Lambda$ at $f$ from \eqref{eq: form functional}. Suppose that $T_{f}(f) = T_{f}(g)$ for all $f,g\in\mathcal W$. Now pick arbitrary but distinct densities $f$ and $g$ in $\mathcal W$, which is possible because $\Lambda$ is not constant on $\mathcal W$. For $t\in[0,1]$, set $g_t := (1-t)f + tg$, which belongs to $\mathcal W$ by convexity, and notice that for $h\in\R$, such that $t+h\in[0,1]$, we have $g_{t+h} = (1-t-h)f + (t+h)g = g_t + h(g-f)$. Define $F(t) := \Lambda(g_t)$. By \eqref{eq: form functional} and our assumption, we have
\begin{align*}
F(t+h) - F(t) &= \Lambda(g_{t+h}) - \Lambda(g_t) 
= T_{g_t}(g_{t+h}-g_t) + O(\|g_{t+h}-g_t\|_{(2,m,\lambda)}^2)\\
&=
T_{g_t}(g_{t+h}) - T_{g_t}(g_t) + 
O(h^2\|g-f\|_{(2,m,\lambda)}^2) = O(h^2\|g-f\|_{(2,m,\lambda)}^2), 
\end{align*}
whenever $t\in[0,1]$, $t+h\in[0,1]$ and $|h|$ is sufficiently small such that $\|h(g-f)\|_{(\infty,m)}$ is small enough. Thus, we conclude that $F$ is continuous on $[0,1]$ and $F'(t) = 0$ for all $t\in(0,1)$, which forces $F$ to be constant on $[0,1]$. In particular, $\Lambda(f) = F(0) = F(1) = \Lambda(g)$. Since $f$ and $g$ were arbitrary, $\Lambda$ is constant on $\mathcal W$, a contradiction. Thus, we may pick $f,g\in\mathcal W$ such that $c:=T_f(g-f)\neq0$.

Now fix arbitrary $\eps>0$ and set $f_0:=f\in\mathcal W$ and $f_1 := g_\eps = (1-\eps)f + \eps g \in\mathcal W$, so that $f_1-f_0 = \eps(g-f)$, $\frac12\|f_1-f_0\|_1 = \frac{\eps}{2}\|g-f\|_1 \le \frac{\eps}{2}(\|g\|_1+\|f\|_1) = \eps$, $T_{f_0}(f_1-f_0) = \eps T_{f}(g-f) = c\eps$ and $\|f_1-f_0\|^2_{(2,m,\lambda)} = \eps^2 \|g-f\|^2_{(2,m,\lambda)}$. Thus, we get from \eqref{eq: form functional},
\begin{align*}
    |\Lambda(f_0)-\Lambda(f_1)| 
    &= |T_{f_0}(f_1-f_0) + O(\|f_1-f_0\|^2_{(2,m,\lambda)})|\\
    &\ge |T_{f_0}(f_1-f_0)| - C_\Lambda\|f_1-f_0\|^2_{(2,m,\lambda)} \\
    &\ge |c|\eps - \eps^2 C_\Lambda\|g-f\|^2_{(2,m,\lambda)},
\end{align*}
provided that $\|f_1-f_0\|_{(\infty,m)} = \eps \|g-f\|_{(\infty,m)}$ is small enough, say less than some constant $c_\Lambda$, and where the positive constants $c_\Lambda$ and $C_\Lambda$ depend only on the functional $\Lambda$. The lower bound in the previous display therefore also lower bounds the modulus of continuity $\omega_{TV}(\eps)$ and together with \eqref{eq:RohdeSteinb} the proof is finished.
\hfill\qed

\subsection{Proof of Theorem~\ref{thm:LBatomic}}

We proceed as in the proof of Theorem~\ref{thm:LBparam} to get
\begin{equation}\label{eq:RohdeSteinb2}
\inf_{Q\in\mathcal Q_\alpha} \inf_{\hat{\Lambda}_n} \sup_{f\in\mathcal W} \E_{QP_f}\left[|\hat{\Lambda}_n - \Lambda(f)|\right]
\ge \frac18 \omega_{TV}\left(\left[\frac{1}{\sqrt{8n(e^\alpha-1)^2}}\right]^-\right),
\end{equation}
for all $n>\log 2$, where $\omega_{TV}(\eps^-)$ denotes the limit from the left.

Now fix $f_0\in\mathcal W$ such that $T_{f_0}$ is of index $s\le p$, that is, for all $h\in C^p[0,1]$ we have $T_{f_0}(h) = \sum_{j=0}^s \int_0^1 h^{(j)} d\mu_{j,f_0}$ where $\mu_{s,f_0}$ has a point mass at $x_{f_0}\in[0,1]$, say, with $\mu_{s,f_0}(\{x_{f_0}\}) \neq 0$. In what follows we suppress the dependence on $f_0$ for convenience and write $x_0=x_{f_0}$ and $\mu_j = \mu_{j,f_0}$.
By non-degeneracy, pick $\psi:\R\to\R$ compactly supported with $\psi^{(s)}(0)\neq0$, $\delta_0 = \delta_0(x_0)>0$ and $\eta_0 = \eta_0(f_0)>0$ as in Definition~\ref{def:ND}, such that 
$$
    f_1(x) := f_0(x) + \eta\delta^p\psi((x-x_{0})/\delta), \quad x\in[0,1]
$$
belongs to $\mathcal W$ for all $\eta\in(0,\eta_0)$ and all $\delta\in(0,\delta_0)$.
In particular, we have $\frac12\|f_1-f_0\|_{L^1} = \frac{\eta\delta^p}{2}\int_0^1 |\psi((x-x_{0})/\delta)|dx \le \eta\frac{\delta^{p+1}}{2}\int_\R |\psi(x)|dx = \eta\frac{\delta^{p+1}}{2} \|\psi\|_{L^1} \le \eps$, provided that we have $\delta \le (2\eps/(\eta\|\psi\|_{L^1}))^{\frac{1}{p+1}}$. Furthermore, $\|f_1-f_0\|^2_{(2,m,\lambda)} = \eta^2\delta^{2p} \|\psi((\cdot-x_{0})/\delta)\|^2_{(2,m,\lambda)}\le \eta^2\delta^{2p}\|\psi\|^2_{(\infty,m)}$ and 
\begin{align}\label{eq:sterm}
    \int_0^1 (f_1-f_0)^{(s)}d\mu_{s} = \eta\delta^{p-s} \int_0^1 \bar{\psi}_\delta(x)\mu_{s}(dx),
\end{align} 
where $\bar{\psi}_\delta(x):= \psi^{(s)}((x-x_{0})/\delta)$ is bounded by $\|\psi^{(s)}\|_{L^\infty}<\infty$. Since the support of $\psi$ is compact we have $\bar{\psi}_\delta(x)\to \psi^{(s)}(0)\mathds 1_{\{x_0\}}(x)$ for all $x\in[0,1]$ as $\delta\to0$. Therefore, by the dominated convergence theorem the integral on the right hand side of \eqref{eq:sterm} converges to $\psi^{(s)}(0)\mu_s(\{x_0\})\neq0$ as $\delta\to0$. Similarly, for $j<s$, we have
\begin{align*}
    \left|\int_0^1 (f_1-f_0)^{(j)}d\mu_{j}\right| \le \eta\delta^{p-j} \|\psi^{(j)}\|_{L^\infty}|\mu_{j}|([0,1]),
\end{align*} 
which shows that 
$$
|T_{f_0}(f_1-f_0)| \ge \left|\int_0^1 (f_1-f_0)^{(s)}d\mu_{s}\right| - \left| \sum_{j=0}^{s-1}\int_0^1 (f_1-f_0)^{(j)}d\mu_{j} \right|
\ge
\delta^{p-s}c_0 - O(\delta^{p-s+1})
$$ 
for a $c_0>0$ and for all sufficiently small $\delta>0$.
Thus, we get from \eqref{eq: form functional},
\begin{align*}
    |\Lambda(f_0)-\Lambda(f_1)| 
    &= |T_{f_0}(f_1-f_0) + O(\|f_1-f_0\|^2_{(2,m,\lambda)})|\\
    &\ge |T_{f_0}(f_1-f_0)| - C_0\delta^{2p}\|\psi\|^2_{(\infty,m)} \\
    &\ge \delta^{p-s}c_0 + O(\delta^{p-s+1}),
\end{align*}
provided that $\|f_1-f_0\|_{(\infty,m)} = O(\delta^{p-m})$ is small enough. The lower bound in the previous display therefore also lower bounds the modulus of continuity $\omega_{TV}(\eps)$ and together with \eqref{eq:RohdeSteinb2} and taking $\eps= (8n(e^\alpha-1)^2)^{-1/2}$ and $\delta=(2\eps/(\eta\|\psi\|_{L^1}))^{\frac{1}{p+1}}\land \frac{\delta_0}{2}$ the proof is finished.
\hfill\qed


\section{Proofs of Sections~\ref{sec:SplineEst} and \ref{sec:WaveletEst} on plug-in estimators}
\label{app:Main}
First, we note that for $f,g$ in $\mathcal{W}_p$, we have the following upper bound
\begin{equation}\label{eq: Taylor ineq Frechet deriv}
|\Lambda(f)-\Lambda(g)|\leq \underset{u\in \mathcal{W}_p}{\sup}\norm{T_u}_p\, \norm{f-g}_{(\infty,p)}.
\end{equation}
Indeed, for $t\in[0,1]$ and $\gamma(t)=f+t(g-f)$, $\Lambda(\gamma(t))$ defines a differentiable function since
\[\frac{\Lambda(\gamma(t+h))-\Lambda(\gamma(t))}{h}=T_{\gamma(t)}(g-f)+O\left(h\norm{f-g}^2_{(2,m,\lambda)}\right)\underset{h\to0}{\to} T_{\gamma(t)}(g-f),\]
and \eqref{eq: bound deriv func} implies that
\begin{align*}
  |\Lambda(f)-\Lambda(g)|&=|\Lambda(\gamma(0))-\Lambda(\gamma(1))|\\
  &=\left|\int_0^1 T_{\gamma(t)}(g-f) dt\right|\\
  &\leq \underset{u\in \mathcal{W}_p}{\sup}\norm{T_u}_p\, \norm{f-g}_{(\infty,p)}.
\end{align*}
Then, \eqref{eq: form functional} implies that for some $C>0$, there exists $\delta$ such that for any density estimator $\hat{f}$
\begin{align}
    \mathbb{E}\left[|\Lambda(\hat{f})-\Lambda(f)|\right]&=\mathbb{E}\left[|\Lambda(\hat{f})-\Lambda(f)|\mathds{1}_{\norm{f-\hat{f}}_{(\infty,m)}\leq \delta}\right]+\mathbb{E}\left[|\Lambda(\hat{f})-\Lambda(f)|\mathds{1}_{\norm{f-\hat{f}}_{(\infty,m)}> \delta}\right]\nonumber\\
    &\leq \mathbb{E}\left[|T_f(f-\hat{f})| + C\norm{\hat{f}_n-f}_{(2,m,\lambda)}^2\right]+\underset{u\in \mathcal{W}_p}{\sup}\norm{T_u}_p\ \mathbb{E}\left[\norm{f-\hat{f}}_{(\infty,p)}\mathds{1}_{\norm{f-\hat{f}}_{(\infty,m)}> \delta}\right]\nonumber\\
    &\leq \mathbb{E}\left[|T_f(f-\hat{f})|\right]+ C\mathbb{E}\left[\norm{\hat{f}_n-f}_{(2,m,\lambda)}^2\right]\nonumber\\
    &\qquad +\underset{u\in \mathcal{W}_p}{\sup}\norm{T_u}_p\,  \sqrt{\mathbb{E}\left[\norm{f-\hat{f}}_{(\infty,p)}^2\right]}\, \mathbb{P}\left[\norm{f-\hat{f}}_{(\infty,m)}> \delta\right]^{1/2},\label{eq: Taylor expansion Expectation}
\end{align}
where we used \eqref{eq: Taylor ineq Frechet deriv} and the last line follows from Cauchy-Schwarz inequality.

\subsection{Proof of Theorem~\ref{th: rate atomic}}

We first control the nonparametric convergence rates (pointwise and in $L^2$ distance) of the estimator $\hat{f}_n$ before translating these rates into semiparametric ones for $\Lambda(\hat{f}_n)$, using arguments from \cite{GoldsteinMesser92}.\\

\textit{Bias of $\hat{f}_n$}: It is clear that $\mathbb{E}[\hat{f}_n]=\mathcal{S}_{d,\mathbf{\xi}^{(j_n)}}f$, so that the bias of this estimator is given by \eqref{eq:bias} with $q=0$ and $j=j_n$. A bound on the norm of $f^{(q)} - \mathbb{E}_f\left[\hat{f}_n^{(q)}\right]$ can also be obtained  via the inequality \eqref{eq:bias} for $q \leq p-1$, noting that $\mathbb{E}\left[\hat{f}_n^{(q)}\right]= \mathbb{E}\left[\hat{f}_n\right]^{(q)}$.

\textit{Variance of $\hat{f}_n$}: Turning to the variance of this term, by independence, it is given by the sum of the variance coming from the Laplace terms, and the variance coming from the empirical polynomial coefficents. For any finite absolutely continuous measure $G$ on $[0,1]$ with bounded density $g$, we have
\begin{align} \label{eq: var decomp spline estim}
&\mathbb{E}\norm{\hat{f}_n-\mathcal{S}_{d,\mathbf{\xi}^{(j)}}f}_{L^2(G)}^2= \nonumber \\&\mathbb{E}\int_0^1 \left[\sum_{k=1}^{2^{j_n}+d} \left(n^{-1}\sum_{i=1}^n e_{k,j_n}(X_i)-\mathcal{L}_{k,d,\xi^{(j_n)}} \norm{B_{k,d,\xi^{(j_n)}}}_{L^2}\right)\mathbf{B}_{k,d,\xi^{(j_n)}}(x)\right]^2 g(x)dx\nonumber\\
&\qquad+\sigma_{\alpha,j_n}^2\mathbb{E}\int_0^1 \left[\sum_{k=1}^{2^{j_n}+d} \left(n^{-1}\sum_{i=1}^n Y_{k,i}\right) \mathbf{B}_{k,d,\xi^{(j_n)}}(x) \right]^2g(x) dx.
\end{align}
For any $0\leq x\leq1$, we have by independence and the fact that $Y_{k,i}$ has mean 0, that 
\begin{align*}
    &\mathbb{E}\left[\sum_{k=1}^{2^{j_n}+d} \left(n^{-1}\sum_{i=1}^n Y_{k,i}\right) \mathbf{B}_{k,d,\xi^{(j_n)}}(x) \right]^2 \\& \qquad= \sum_{k=1}^{2^{j_n}+d} \sum_{k'=1}^{2^{j_n}+d} \mathbb{E}\left[n^{-1}\sum_{i=1}^n Y_{k,i}\right]\left[n^{-1}\sum_{i'=1}^n Y_{k',i'}\right]  \mathbf{B}_{k,d,\xi^{(j_n)}}(x)\mathbf{B}_{k',d,\xi^{(j_n)}}(x)\\
    &\qquad= \sum_{k=1}^{2^{j_n}+d} \mathbb{E}\left[n^{-1}\sum_{i=1}^n Y_{k,i}\right]^2 \mathbf{B}_{k,d,\xi^{(j_n)}}(x)^2.
\end{align*}
Therefore, the second term from \eqref{eq: var decomp spline estim} is bounded by $2n^{-1}\sigma_{\alpha,j_n}^2\sum_{k=1}^{2^{j_n}+d} \int_0^1\mathbf{B}_{k,d,\xi^{(j_n)}}(x)^2 g(x)dx$. Since $\sigma_{\alpha,j_n}=C_d\alpha^{-1}2^{j_n/2}$ as in Proposition \ref{th: privacy}, this variance term is then upper bounded by $C_d \norm{g}_{L^\infty} n^{-1}\alpha^{-2} 2^{2j_n}$ (we use that $\mathbf{B}_{k,d,\xi^{(j_n)}}$ are normalized in $L^2$-norm). As for the first variance term from \eqref{eq: var decomp spline estim}, we now use that B-splines have almost disjoint supports. More precisely, (1.6) in \cite{LycheSpline} implies that 
$\mathbf{B}_{k,d,\xi^{(j_n)}}(x) \mathbf{B}_{k',d,\xi^{(j_n)}}(x) =0$
for any $0<x<1$ and $|k-k'|\geq d+1$. Writing $v_{n,k}=n^{-1}\sum_{i=1}^n e_{k,j_n}(X_i)-\mathcal{L}_{k,d,\xi^{(j_n)}}\norm{B_{k,d,\xi^{(j_n)}}}_{L^2}$, the disjoint support property gives that 
\begin{align*}
\mathbb{E}\int_0^1 \left(\sum_{k=1}^{2^{j_n}+d} v_{n,k}\mathbf{B}_{k,d,\xi^{(j_n)}}(x)\right)^2 &g(x) dx =  \sum_{k=1}^{2^{j_n}+d} \mathbb{E}\left[v_{n,k}^2\right]  \int_0^1\mathbf{B}_{k,d,\xi^{(j_n)}}(x)^2 g(x)dx \\
&+\sum_{k\neq k',\ |k-k'|\leq d} \mathbb{E}\left[v_{k,n}v_{k',n}\right]\int_0^1\mathbf{B}_{k,d,\xi^{(j_n)}}(x)\mathbf{B}_{k',d,\xi^{(j_n)}}(x) g(x)dx.
\end{align*}
Let's bound these two terms. By independence of observations and since the centered second-order moment is smaller than the non-centered one, we obtain 
\begin{align}
\mathbb{E}_f\left[v_{n,k}^2\right]&\leq n^{-1}\mathbb{E}_f\left[e_{k,j_n}(X_1)^2\right]\nonumber\\
&\leq h_{k,d,\xi^{(j_n)}}^{-2} \norm{B_{k,d,\xi^{(j_n)}}}_{L^2}^2\ n^{-1}\int_{\xi^{(j_n)}_{m_{k,d}}}^{\xi^{(j_n)}_{m_{k,d}+1}} \left(\sum_{i=0}^d a_{k,i}\left(\frac{x-\xi^{(j_n)}_{m_{k,d}}}{h_{k,d,\xi^{(j_n)}}}\right)^i\right)^2f(x)dx.\label{eq: bound empirical coefficient}
\end{align}
The square in the above integral is uniformly bounded by a constant $C_d$ depending on $d$ only (the coefficients $a_{k,i}$ are bounded by such a quantity, see Section \ref{sec:SplineEst}), so that the above expectation is upper bounded by $C_d \norm{f}_{L^\infty} n^{-1}$, given \eqref{eq: norm splines} and the fact that $h_{k,d,\xi^{j_n}}=2^{-j_n}$. We can then bound the first term as before by $C_d\norm{f}_{L^\infty} \norm{g}_{L^\infty} n^{-1}2^{j_n}$ since $\int_0^1\mathbf{B}_{k,d,\xi^{(j_n)}}(x)^2 g(x)dx\leq \norm{g}_{L^\infty}$. For the second term, we use Cauchy-Schwarz inequality to bound each element of the sum as
\[ \mathbb{E}\left[v_{n,k}v_{n,k'}\right]\int_0^1\mathbf{B}_{k,d,\xi^{(j_n)}}(x)\mathbf{B}_{k',d,\xi^{(j_n)}}(x) g(x)dx\leq C_d \norm{f}_{L^\infty} \norm{g}_{L^\infty} n^{-1}\]and then use the fact that there are at most $C_d 2^{j_n}$ elements in this sum to obtain the bound $C_d \norm{f}_{L^\infty} \norm{g}_{L^\infty} n^{-1} 2^{j_n}$ on this term.\\

For the purpose of characterizing pointwise variance $\mathbb{E}_f|\hat{f}_n(x_0) - (\mathcal{S}_{d,\mathbf{\xi}^{(j_n)}}f)(x_0)|^2$, which can be decomposed similarly, we use \eqref{eq: norm splines} and \eqref{eq: partition of unity}, the positivity of B-splines, and the fact that they have bounded support, included in $[\xi_k, \xi_{k+d+1}]$, to obtain for any $\xi_l \leq x_0\leq \xi_{l+1}$,
\begin{align*}
    \sigma_{\alpha,j_n}^2\mathbb{E}\left(\sum_{k=1}^{2^{j_n}+d} \left(n^{-1}\sum_{i=1}^n Y_{k,i}\right) \mathbf{B}_{k,d,\xi^{(j_n)}}(x_0) \right)^2
    &\leq\sigma_{\alpha,j_n}^2 \mathbb{E} \sup_{l-d\leq k\leq l} \left(n^{-1}\sum_{i=1}^n Y_{k,i}\right)^2\norm{B_{k,d,\xi^{(j_n)}}}_{L^2}^{-2}\\
    &\leq \sigma_{\alpha,j_n}^2C_d n^{-1} 2^{j_n} \leq C_d \alpha^{-2} n^{-1} 2^{2j_n}. 
\end{align*}
Also, we deduce from \eqref{eq: partition of unity}, the bounded support, \eqref{eq: bound empirical coefficient} and \eqref{eq: norm splines} that
\begin{align*}
    \mathbb{E}_f\left(\sum_{k=1}^{2^{j_n}+d} v_{n,k}\mathbf{B}_{k,d,\xi^{(j_n)}}(x_0)\right)^2
    &\leq \mathbb{E}_f \sup_{l-d\leq k\leq l} v_{n,k}^2\norm{B_{k,d,\xi^{(j_n)}}}_{L^2}^{-2}\\
    &\leq C_d \norm{f}_{L^\infty} n^{-1} 2^{j_n}. 
\end{align*}

Therefore, the variance of our estimator is bounded by $Cn^{-1}\left(\alpha^{-2}2^{2j_n}+2^{j_n}\right)$ (whether for $L^2$ or pointwise convergence), for $C$ a constant depending on $d$, $\norm{f}_{L^\infty}$ and $\norm{g}_{L^\infty}$ if $L^2(G)$--convergence is considered. The same arguments on the variance carry over to the derivatives of order $q\leq p-1$, thanks to \eqref{eq: derivatives}, with a multiplier of order $C_d 2^{qj_n}$, so that 
\begin{align*}
\underset{0\leq x_0\leq 1}{\sup}\left(\mathbb{E}_f|\hat{f}_n^{(q)}(x_0) - (\mathcal{S}_{d,\mathbf{\xi}^{(j_n)}}f)^{(q)}(x_0)|^2\right)\,&\vee\, \left( \mathbb{E}_f\norm{\hat{f}_n^{(q)} - (\mathcal{S}_{d,\mathbf{\xi}^{(j_n)}}f)^{(q)}}_{L^2(G)}^2\right)\\&\leq C_{d,M} n^{-1}\left[\alpha^{-2} 2^{(2+2q)j_n} + 2^{(1+2q)j_n}\right],\end{align*}
where we used that $\norm{f}_{L^\infty}$ and $\norm{g}_{L^\infty}$ are bounded by $M$ under Assumption \ref{assum: no singular continuous part} and $f\in\mathcal{W}_p$. Combined with \eqref{eq:bias} and  $\norm{f^{(q)}}_{\infty}\leq M$ since $f\in\mathcal{W}_p$,
\begin{align}\label{eq: derivative loss}
   \underset{0\leq x_0\leq 1}{\sup}&\left(\mathbb{E}_f|\hat{f}_n^{(q)}(x_0) - f^{(q)}(x_0)|^2\right)\,\vee\, \left( \mathbb{E}_f\norm{\hat{f}_n^{(q)} - f^{(q)}}_{L^2(G)}^2\right)\nonumber\\ &\leq C_{d,M}\left[ 2^{(2q-2p)j_n}+ n^{-1}(\alpha^{-2} 2^{(2+2q)j_n} + 2^{(1+2q)j_n})\right]. 
\end{align}
With the bound on $j_n$ from the statement of Theorem~\ref{th: rate atomic}, we can bound for $q\leq p-1$,
\begin{align*}
    &\underset{0\leq x_0\leq 1}{\sup}\left(\mathbb{E}_f|\hat{f}_n^{(q)}(x_0) - f^{(q)}(x_0)|^2\right)\,\vee\, \left(\mathbb{E}_f\norm{\hat{f}_n^{(q)} - f^{(q)}}_{L^2(G)}^2\right)\\
    &\leq C_{d,M} \left[(n\alpha^2)^{\frac{-2(p-q)}{2p+2}}\vee n^{\frac{-2(p-q)}{2p+1}} \right]\coloneqq \tau_{q,n}.
\end{align*}

\textit{Semiparametric rates for $\Lambda(\hat{f}_n)$}: For atomic functionals, we can bound the first two terms from \eqref{eq: Taylor expansion Expectation} by
\[C\left[\sum_{j=0}^{s_f} \mathbb{E}_f  \int_0^1 \left|(\hat{f}_n-f)^{(j)}\right|d\mu_{j,f} + \mathbb{E}_f \norm{\hat{f}_n-f}_{(2,m,\lambda)}^2 \right].\]
For any finite signed measure $\mu$ as in Assumption \ref{assum: no singular continuous part}, we use Jordan decomposition $\mu= \mu^+ - \mu^-$, where $\mu^+,\ \mu^-$ are positive finite measures, and introduce $|\mu|= \mu^+ + \mu^-$ and $\Tilde{\mu}= \frac{|\mu|}{|\mu|([0,1])}$ which is a probability measure. Assumption \ref{assum: no singular continuous part} implies that $\Tilde{\mu}$ also has no singular continuous part. From Fubini's and Jensen's Theorems, we bound for $r=1$ or $r=2$ and $j\leq s$,
\begin{align*}
\mathbb{E}_f\int_0^1 |(\hat{f}_n-f)^{(j)}|^rd\mu&\leq \mathbb{E}_f\int_0^1 |(\hat{f}_n-f)^{(j)}|^rd|\mu|\\
&\leq |\mu|([0,1]) \left(\int_0^1 \mathbb{E}_f|(\hat{f}_n-f)^{(j)}|^2 d\Tilde{\mu} \right)^{1/(3-r)}\\
&\leq |\mu|([0,1]) \left(\mathbb{E}_f\int_0^1 |(\hat{f}_n-f)^{(j)}|^2 d\Tilde{\mu}_{\text{cont}} + \mathbb{E}_f \int_0^1 |(\hat{f}_n-f)^{(j)}|^2 d\Tilde{\mu}_{\text{pp}} \right)^{1/(3-r)}\\
&\leq C_{d,M}\ |\mu|([0,1])\ \tau_{j,n}^{1/(3-r)},
\end{align*}
where $|\mu|([0,1])$ is upper bounded by a quantity depending on $M$ only by assumption. As a consequence, for $n$ large enough, the first two terms from \eqref{eq: Taylor expansion Expectation} are bounded by
\[C\left[\sum_{j=0}^{s_f} \tau_{j,n}^{1/2} + \sum_{j=0}^{m}\tau_{j,n} \right]\leq C\tau_{s,n}^{1/2},\]
as we assumed $p\geq 2m-s$ and $s_f\leq s$. The last term from \eqref{eq: Taylor expansion Expectation} according to Lemma \ref{lem: exponential decay outer taylor exp} is of order $o(e^{-n^c})$ for some $c>0$, so that for $n$ large enough
\[\underset{f\in\mathcal{W}_p}{\sup}\mathbb{E}_f|\Lambda(f)-\Lambda(\hat{f}_n)|\leq C\tau^{1/2}_{s,n}\leq C_{d,M}\left[(n\alpha^2)^{\frac{-2(p-s)}{2p+2}}\vee n^{\frac{-2(p-s)}{2p+1}} \right].\]
\hfill\qed

\subsection{Proof of Theorem~\ref{th: atomic wav}}

As in the previous section, we first control the nonparametric convergence rates (pointwise and in $L^2$ distance) of the estimator $\hat{f}_n$ before translating these rates into semiparametric ones for $\Lambda(\hat{f}_n)$, using arguments from \cite{GoldsteinMesser92}.

As a preliminary consideration, we introduce $P_{\mathcal{V}_{j_n+1}}f$ the orthogonal projection in $L^2[0,1]$ of $f$ on the space of splines with equispaced knots $\mathcal{V}_{j_n+1}$ as in \eqref{eq: equi spline space}, and we remark that $\E\hat{f}_n$ satisfies
\begin{align*} 
\E\hat{f}_n&= \sum_{k\in\mathcal{M}_{j_0-1}} \langle f , \mathbf{B}_{k,d,\xi^{(j_0)}}\rangle \Tilde{\psi}_{j_0-1,k}  + \sum_{j=j_0}^{j_n}\sum_{k\in\mathcal{M}_{j}} \langle f , \psi_{j,k}\rangle \Tilde{\psi}_{j,k}\\
&= \sum_{k\in\mathcal{M}_{j_0-1}} \langle P_{\mathcal{V}_{j_n+1}}f , \mathbf{B}_{k,d,\xi^{(j_0)}}\rangle \Tilde{\psi}_{j_0-1,k} + \sum_{j=j_0}^{j_n}\sum_{k\in\mathcal{M}_{j}} \langle P_{\mathcal{V}_{j_n+1}}f , \psi_{j,k}\rangle \Tilde{\psi}_{j,k}\\
&= P_{\mathcal{V}_{j_n+1}}f,
\end{align*}
where we used the definition of the dual frames in the last line. Since $\mathcal{S}_{d,\xi^{(j_n+1)}}f\in \mathcal{V}_{j_n+1}$ for $\xi^{(j_n+1)}$ as in \eqref{eq:knot seq multires}, the definition of orthogonal projection and equation \eqref{eq:bias} with $q=0$ implies that
\begin{equation}\label{eq: L2 bias proj spline}
    \norm{f-P_{\mathcal{V}_{j_n+1}}f}_{L^2} \leq \norm{f-\mathcal{S}_{d,\xi^{(j_n+1)}}f}_{L^2} \leq \norm{f-\mathcal{S}_{d,\xi^{(j_n+1)}}f}_{L^\infty} \leq C_d 2^{-p(j_n+1)} \norm{f^{(p)}}_{L^\infty},
\end{equation}
whenever $f\in \mathcal{W}_{p}$ and $p\leq d+1$. For any $g\in \mathcal{V}_{j_n+1}$, we obviously have $P_{\mathcal{V}_{j_n+1}}g=g$ so that 
\[\norm{P_{\mathcal{V}_{j_n+1}}f - \mathcal{S}_{d,\xi^{(j_n+1)}}f}_{L^\infty} = \norm{P_{\mathcal{V}_{j_n+1}}(f - \mathcal{S}_{d,\xi^{(j_n+1)}}f)}_{L^\infty}.\]
This last quantity can be bounded by $c_d \norm{f - \mathcal{S}_{d,\xi^{(j_n+1)}}f}_{L^\infty}$, for some quantity $c_d$ depending on $d$ only and not on the knot sequence, according to De Boor's conjecture (proved in \cite{Shadrin2001}, see Theorem 1 therein) which states that
\[\sup_{f\in L^\infty([0,1]),\ f\neq 0}\frac{\norm{P_{\mathcal{V}_j}f}_{L^\infty}}{\norm{f}_{L^\infty}} \leq c_d,\]
independently of $j$. As a consequence, we deduce from \eqref{eq:bias} with $q=0$ that
\begin{equation}\label{eq: supnorm bias proj spline}
    \norm{f-P_{\mathcal{V}_{j_n+1}}f}_{\infty} \leq (1+c_d)\norm{f-\mathcal{S}_{d,\xi^{(j_n+1)}}f}_{L^\infty} \leq C_d 2^{-p(j_n+1)} \norm{f^{(p)}}_{L^\infty}.
\end{equation}

\textit{Bias of $\hat{f}_n$}: The bias of the estimator $\hat{f}_n$ in $L^2$--norm is bounded by the right-hand side of \eqref{eq: L2 bias proj spline}, while the bias in supnorm is bounded in \eqref{eq: supnorm bias proj spline}.
From the first inequality of Theorem~\ref{th: from L2 to l2}, we deduce that the $\ell^2$ distance between the coefficients $\mathbf{c}_1$ and $\mathbf{c}_2$ in the (normalized) B-splines basis of $\E\hat{f}_n = P_{V_{j_n+1}}f$ and $\mathcal{S}_{d,\mathbf{\xi}^{(j_n+1)}}f$ respectively is bounded from above by a quantity proportional to the $L^2$--distance between these splines. With \eqref{eq:bias} and \eqref{eq: L2 bias proj spline}, the triangle inequality ensures that this distance of order $C_{d,M}2^{-pj_n}$. From \eqref{eq: derivatives}, the derivatives of these two splines lie in the space $\mathbb{S}_{d-1,\xi^{(j_n+1)}}$ and we see that their respective coefficients $\mathbf{c}_1'$ and $\mathbf{c}_2'$ in the (normalized) B-splines basis of this new space are separated in $\ell^2$ distance by at most \begin{equation}\label{eq: distance coef l2}\norm{\mathbf{c}_1'-\mathbf{c}_2'}_{\ell^2}\leq C_d2^{j_n}\norm{\mathbf{c}_1-\mathbf{c}_2}_{\ell^2}\leq C_d 2^{-j_n(p-1)}.\end{equation}
From the first inequality in Theorem~\ref{th: from L2 to l2} and \eqref{eq: distance coef l2}, we again deduce that the $L^2$--distance between $(P_{\mathcal{V}_{j_n+1}}f)^{(1)}$ and $(\mathcal{S}_{d,\mathbf{\xi}^{(j_n+1)}}f)^{(1)}$ is bounded by $C_{d,M}2^{-j_n(p-1)}$, with the constant depending on $M$ and $d$ only. The same arguments applied repeatedly to $\left(P_{\mathcal{V}_{j_n+1}}f\right)^{(q)}$ and $\left(\mathcal{S}_{d,\mathbf{\xi}^{(j_n+1)}}f\right)^{(q)}$, for $q=1,\dots,p-1$ give that these are separated in $L^2$--distance by a quantity bounded by $C_{d,M} 2^{-j_n(p-q)}$. Combining this with \eqref{eq:bias}, we obtain that this is also the case for $\left(P_{\mathcal{V}_{j_n+1}}f\right)^{(q)}$ and $f^{(q)}$,
\[\norm{\left(P_{\mathcal{V}_{j_n+1}}f\right)^{(q)}-f^{(q)}}_{L^2} \leq C_{d,M} 2^{-j_n(p-q)}.\]
Using the same arguments with the second inequality of Theorem \ref{th: from L2 to l2}, we also prove that
\begin{equation}\label{eq: supnorm bias proj}
\norm{\left(P_{\mathcal{V}_{j_n+1}}f\right)^{(q)}-f^{(q)}}_{L^\infty} \leq C_{d,M} 2^{-j_n(p-q)}.
\end{equation}

\textit{Variance of $\hat{f}_n$}: We now bound the variance of the estimator for $L^2$ and pointwise distances. Let's write 
\begin{align*}
K_{j_n}(x,y)&\coloneqq \sum_{k\in\mathcal{M}_{j_0-1}} \mathbf{B}_{k,d,\xi^{(j_0)}}(x) \Tilde{\psi}_{j_0-1,k}(y) + \sum_{j=j_0}^{j_n}\sum_{k\in\mathcal{M}_{j}} \psi_{j,k}(x)\Tilde{\psi}_{j,k}(y)\\
&= \sum_{j=j_0-1}^{j_n}\sum_{k\in\mathcal{M}_{j}} \psi_{j,k}(x)\Tilde{\psi}_{j,k}(y).
\end{align*}
The variance of the estimator, by independence, is given by the sum of the variance coming from the Laplace terms and the variance coming from the empirical spline wavelet coefficients. For any absolutely continuous measure $G$ with bounded Lebesgue density $g$ and $P_n$ the empirical measure of the $X_1\dots,X_n$, we have by
\begin{align} 
\mathbb{E}_f\norm{\hat{f}_n-\mathbb{E}_f\hat{f}_n}_{L^2(G)}^2 &= \mathbb{E}_f\int_0^1 \left(\int_0^1 K_{j_n}(x,y) dP_n(x) - \int_0^1 K_{j_n}(x,y) f(x)dx \right)^2 g(y)dy\nonumber\\
&+\sum_{k\in\mathcal{M}_{j_0-1}} 2\frac{\sigma_{\alpha,j_0-1}^2}{n} \norm{\Tilde{\psi}_{j_0-1,k}}_{L^2(G)}^2 + \sum_{j=j_0}^{j_n}\sum_{k\in\mathcal{M}_{j}} 2\frac{\sigma_{\alpha,j}^2}{n}\norm{\Tilde{\psi}_{j,k}}_{L^2(G)}^2\nonumber\\
&\leq \norm{g}_{L^\infty}n^{-1} \mathbb{E}_f\int_0^1 \left(  \sum_{j=j_0-1}^{j_n}\sum_{k\in\mathcal{M}_{j}} \psi_{j,k}(X_1)\Tilde{\psi}_{j,k}(y) \right)^2dy\nonumber\\
&+2n^{-1}\norm{g}_{L^\infty}\left[\sum_{k\in\mathcal{M}_{j_0-1}} \sigma_{\alpha,j_0-1}^2 \norm{\Tilde{\psi}_{j_0-1,k}}_{L^2}^2 + \sum_{j=j_0}^{j_n}\sum_{k\in\mathcal{M}_{j}} \sigma_{\alpha,j}^2\norm{\Tilde{\psi}_{j,k}}_{L^2}^2\right],\label{eq: bound splwav var cont G}
\end{align}
where we bounded the variance by the non-centered second-order moment in the last inequality. Next, we deal with the two terms of \eqref{eq: bound splwav var cont G} and prove that the first one is bounded by $C_d\norm{g}_{\infty}\norm{f}_{\infty}n^{-1} 2^{j_n}$, while the second one is bounded is bounded by $C_d \norm{g}_{L^\infty}\alpha^{-2} n^{-1} j_n^{2a}2^{2j_n}$

Note that all the functions $\psi_{j,k},\psi_{j',k'},\Tilde{\psi}_{j,k}$ and $\Tilde{\psi}_{j',k'}$ are orthogonal for different resolution levels $j\neq j'$, since $S_j^{-1}$ maps into $\mathcal{U}_j$ and the $\mathcal{U}_j$'s are mutually orthogonal. Then, the integral in the first term of \eqref{eq: bound splwav var cont G} is equal to
\[
    \sum_{j=j_0-1}^{j_n} \int_0^1 \left( \sum_{k\in\mathcal{M}_{j}} \psi_{j,k}(X_1)\Tilde{\psi}_{j,k}(y) \right)^2dy.
\]
Let's fix $j$ between $j_0-1$ and $j_n$. We recall that $\psi_{j,k} = S_{j}\Tilde{\psi}_{j,k}$ for some operator $S_{j}$ with $\norm{S_{j}}\leq B_d$ (for the induced $L^2$ operator norm), where $B_d$ is a quantity that depends on $d$ only. Let's write $\mathbf{T}_{j,X}=(\dots,\psi_{j,k}(X),\dots)^T$ where $k$ runs over $\mathcal{M}_{j}$ and $G_{\tilde{\psi},j}=\left(\int_0^1  \Tilde{\psi}_{j,l}(y)\Tilde{\psi}_{j,k}(y) dy\right)_{l,k\in \mathcal{M}_j}$ the Gram matrix of those dual wavelets in $\mathcal{U}_j$. For an arbitrary orthonormal basis $\mathcal{O}_j=(O_q)_{q\in \mathcal{M}_j}$ of $\mathcal{U}_j$ (or $\mathcal{V}_{j_0}$ if $j=j_0-1$), we can also write $C_j$ the matrix whose columns are formed by the coordinates of the $\psi_{j,k}$ basis functions in $\mathcal{O}_j$ (i.e., $(C_j)_{q'k}= \langle O_{q'}, \psi_{j,k}\rangle$), and $\bar{S}_{j}^{-1}$ the matrix of the endomorphism $S_{j}^{-1}$ in  $\mathcal{O}_j$ (i.e., $(\bar{S}_{j}^{-1})_{qq'}= \langle O_{q}, S_j^{-1}O_{q'}\rangle$) which has spectral norm bounded by $A_d^{-1}$ by \eqref{eq: bound operator norm frame inv}. We then have 
\begin{align*}
\langle \Tilde{\psi}_{j,l},  \Tilde{\psi}_{j,k}\rangle &= \sum_{q\in\mathcal{M}_j}  \langle O_{q}, \Tilde{\psi}_{j,l}\rangle \langle O_{q}, \Tilde{\psi}_{j,k}\rangle\\
&=\sum_q \sum_{q'} \langle O_{q'}, \psi_{j,l}\rangle \langle S_j^{-1}O_{q'}, O_q\rangle  \sum_{q''} \langle O_{q''}, \psi_{j,k}\rangle \langle S_j^{-1}O_{q''}, O_q\rangle,
\end{align*}
since $\Tilde{\psi}_{j,l} = \sum_{q'\in\mathcal{M}_j} \langle O_{q'}, \psi_{j,l}\rangle S_j^{-1}O_{q'}$.
Using the fact that $S_j^{-1}$ is self-adjoint, we can then check that
\begin{align}\label{eq: get rid of tilde}
    \int_0^1 \left( \sum_{k\in\mathcal{M}_{j}} \psi_{j,k}(X_1)\Tilde{\psi}_{j,k}(y) \right)^2dy &= \mathbf{T}_{j,X_1}^T G_{\tilde{\psi},j}\mathbf{T}_{j,X_1}\nonumber\\
    &= \mathbf{T}_{j,X_1}^T C_j^T (\bar{S}_{j}^{-1})^{T} \bar{S}_{j}^{-1} C_j \mathbf{T}_{j,X_1}\nonumber\\
    &\leq A_d^{-2} \mathbf{T}_{X_1}^T C_j^T C_j \mathbf{T}_{X_1}\nonumber\\
    &= A_d^{-2} \int_0^1 \left(  \sum_{k\in\mathcal{M}_{j}} \psi_{j,k}(X_1)\psi_{j,k}(y) \right)^2dy.
\end{align}

We also recall that for any $x\in[0,1]$ and fixed $j\geq j_0-1$, there are at most $C_d$ non-zero wavelet evaluations $\psi_{j,k}(x)$. We denote by $I_j(x)$ the set of indices of these non-zero values. Using Cauchy-Schwarz inequality twice on the second line below (once for the integral, once for the expectation)
\begin{align*}
\mathbb{E}_f\int_0^1 \left(  \sum_{k\in\mathcal{M}_{j}} \psi_{j,k}(X_1)\psi_{j,k}(y) \right)^2dy&=\mathbb{E}_f \sum_{k,k'\in I_j(X_1)}\psi_{j,k}(X_1)\psi_{j,k'}(X_1)\int_0^1\psi_{j,k}(y)\psi_{j,k'}(y) dy \\
&\leq \norm{f}_{L^\infty} \sum_{\substack{k,k'\in\mathcal{M}_{j},\\ |k-k'|\leq C_d}} \int_0^1\psi_{j,k}(y)^2 dy\int_0^1\psi_{j,k'}(y)^2 dy \\
&= \sum_{\substack{k,k'\in\mathcal{M}_{j},\\ |k-k'|\leq C_d}} \norm{f}_{L^\infty} \leq C_d \norm{f}_{L^\infty} 2^j.
\end{align*}
Summing these terms, we obtain
\[\norm{g}_{L^\infty}n^{-1} \mathbb{E}_f\int_0^1 \left(  \sum_{j=j_0-1}^{j_n}\sum_{k\in\mathcal{M}_{j}} \psi_{j,k}(X_1)\Tilde{\psi}_{j,k}(y) \right)^2dy \leq C_d\norm{g}_{L^\infty}\norm{f}_{L^\infty}n^{-1} 2^{j_n}. \]

Next, for the second term in \eqref{eq: bound splwav var cont G}, noting that the elements of the wavelet basis are normalised and in view of \eqref{eq: bound operator norm frame inv}, 
\begin{align*}
    \sum_{k\in\mathcal{M}_{j_0-1}} \sigma_{\alpha,j_0-1}^2 \norm{\Tilde{\psi}_{j_0-1,k}}_{L^2}^2 + \sum_{j=j_0}^{j_n}\sum_{k\in\mathcal{M}_{j}} \sigma_{\alpha,j}^2\norm{\Tilde{\psi}_{j,k}}_{L^2}^2 &\leq \sum_{k\in\mathcal{M}_{j_0-1}} \sigma_{\alpha,j_0-1}^2 A_d^{-2} + \sum_{j=j_0}^{j_n}\sum_{k\in\mathcal{M}_{j}} \sigma_{\alpha,j}^2 A_d^{-2} \\
    &= C_d \alpha^{-2} \left[ \sum_{k\in\mathcal{M}_{j_0-1}} 2^{j_0} + \sum_{j=j_0}^{j_n}\sum_{k\in\mathcal{M}_{j}} j^{2a}2^{j}\right]\\
    &\leq C_d \alpha^{-2} j_n^{2a}2^{2j_n}.
\end{align*}

In conclusion, 
\begin{equation}\label{variance stim sp}
    \mathbb{E}_f\norm{\hat{f}_n-\mathbb{E}_f\hat{f}_n}_{L^2(G)}^2 \leq C_d\norm{g}_{L^\infty}n^{-1}(\norm{f}_{L^\infty}2^{j_n} + 2^{2j_n}j_n^{2a}\alpha^{-2}).
\end{equation}
From \eqref{eq: derivatives}, the $q$th derivative of $\hat{f}_n-\mathbb{E}_f\hat{f}_n$ for $0\leq q\leq d-1$ is also a spline, and denoting by $\mathbf{b}(X_1)^{(q)}$ the coefficients in the relevant normalized B-spline basis, we have $\mathbb{E}_f\norm{\mathbf{b}(X_1)^{(q)}}_{\ell^2}^2\leq C_d 2^{2qj_n}\mathbb{E}_f\norm{\mathbf{b}(X_1)}_{\ell^2}^2$. Using Theorem \ref{th: from L2 to l2} and the fact that $\hat{f}_n-\mathbb{E}_f\hat{f}_n \in \mathcal{V}_{j_n+1}$, we obtain a bound on $\mathbb{E}_f\norm{\mathbf{b}(X_1)}_{\ell^2}^2$ similar to the one in \eqref{variance stim sp} (more precisely, we reproduce the above proof with $G$ the Lebesgue measure, $g=1$ and then use Theorem \ref{th: from L2 to l2}), up to a constant depending on $d$. Using Theorem \ref{th: from L2 to l2} again, we deduce that
\begin{equation}\label{variance stim sp derivatives}
    \mathbb{E}_f\norm{\hat{f}_n^{(q)}-\mathbb{E}_f\hat{f}_n^{(q)}}_{L^2(G)}^2 \leq C_d\norm{g}_{L^\infty}n^{-1}(\norm{f}_{L^\infty}2^{j_n(1+2q)} + 2^{2j_n(1+q)}j_n^{2a}\alpha^{-2}).
\end{equation}

We next control the pointwise variance. For any $x_0\in [0,1]$, we have
\begin{align}
\mathbb{E}_f(\hat{f}_n(x_0)-\mathbb{E}_f\hat{f}_n(x_0))^2 &\leq n^{-1} \mathbb{E}_f \left(  \sum_{j=j_0-1}^{j_n}\sum_{k\in\mathcal{M}_{j}} \psi_{j,k}(X_1)\Tilde{\psi}_{j,k}(x_0) \right)^2\nonumber\\
&+2n^{-1}\left[\sum_{k\in\mathcal{M}_{j_0-1}} \sigma_{\alpha,j_0-1}^2 \Tilde{\psi}_{j_0-1,k}(x_0)^2 + \sum_{j=j_0}^{j_n}\sum_{k\in\mathcal{M}_{j}} \sigma_{\alpha,j}^2\Tilde{\psi}_{j,k}(x_0)^2\right].\label{eq:pointwise:var}
\end{align}
We deal with the two terms on the right-hand side separately.
The expectation in the first term of \eqref{eq:pointwise:var} can be bounded as follows:
\begin{align*}
    \mathbb{E}_f \left(  \sum_{j=j_0-1}^{j_n}\sum_{k\in\mathcal{M}_{j}} \psi_{j,k}(X_1)\Tilde{\psi}_{j,k}(x_0) \right)^2 &\leq \norm{f}_{L^\infty} \int_0^1 \left(  \sum_{j=j_0-1}^{j_n}\sum_{k\in\mathcal{M}_{j}} \psi_{j,k}(y)\Tilde{\psi}_{j,k}(x_0) \right)^2 dy\\
    &= \norm{f}_{L^\infty} \sum_{j=j_0-1}^{j_n}\int_0^1 \left(\sum_{k\in\mathcal{M}_{j}} \psi_{j,k}(y)\Tilde{\psi}_{j,k}(x_0) \right)^2 dy.
\end{align*}
For a fixed $x_0\in[0,1]$, the function $w\colon\ y\mapsto \sum_{k\in\mathcal{M}_{j}}\Tilde{\psi}_{j,k}(x_0) \psi_{j,k}(y)$ is the element $w$ of $\mathcal{U}_j$ (or $\mathcal{V}_{j_0}$ if $j=j_0-1$) satisfying $\langle w, \Tilde{\psi}_{j,k}\rangle = \Tilde{\psi}_{j,k}(x_0)$. Indeed,  we know that $w = \sum_{k\in\mathcal{M}_{j}}\langle w, \Tilde{\psi}_{j,k}\rangle \psi_{j,k}$ and the $\psi_{j,k}$ form a basis (see below Theorem \ref{th: frame operator}). This is the reproducing kernel property of a reproducing kernel Hilbert space. According to Theorem 5 of \cite{Rakotomamonjy2005} (or Remark 6 therein), $\mathcal{U}_j$ is indeed an RKHS as, for any $x\in [0,1]$, 
\[\norm{\sum_{k\in\mathcal{M}_{j}}\Tilde{\psi}_{j,k}(x) \psi_{j,k}}_{L^2} \leq \sum_{k\in\mathcal{M}_{j}}|\Tilde{\psi}_{j,k}(x)| \norm{\psi_{j,k}}_{L^2}= \sum_{k\in\mathcal{M}_{j}}|\Tilde{\psi}_{j,k}(x)| < \infty.\]
Theorem 7 from \cite{Rakotomamonjy2005} explicits the form of the kernel which is equal to
\begin{align*}
  K_j\colon [0,1]^2 & \longrightarrow \mathbb{R} \\[-1ex]
  (s,t) & \longmapsto K_j(s,t) = \sum_{k\in\mathcal{M}_{j}}\Tilde{\psi}_{j,k}(s) \psi_{j,k}(t),
\end{align*}
and the symmetry of the kernel ensures that
\begin{equation}\label{eq: sym kernel splwav}
\sum_{k\in\mathcal{M}_{j}}\Tilde{\psi}_{j,k}(s) \psi_{j,k}(t) = \sum_{k\in\mathcal{M}_{j}}\Tilde{\psi}_{j,k}(t) \psi_{j,k}(s).
\end{equation}
We can then write
\begin{align*}
     \int_0^1 \left(\sum_{k\in\mathcal{M}_{j}} \psi_{j,k}(y)\Tilde{\psi}_{j,k}(x_0) \right)^2 dy &= \int_0^1 \left(\sum_{k\in\mathcal{M}_{j}} \psi_{j,k}(x_0)\Tilde{\psi}_{j,k}(y) \right)^2 dy\\
    &\leq A_d^{-2} \int_0^1 \left(\sum_{k\in\mathcal{M}_{j}} \psi_{j,k}(x_0)\psi_{j,k}(y) \right)^2 dy,
\end{align*}
where the inequality was proved in \eqref{eq: get rid of tilde} already. Using that $\norm{\psi_{j,k}}_{L^\infty}\leq C_d 2^{j/2}$, $\norm{\psi_{j,k}}_{L^2}= 1$ and that, for any $x_0\in[0,1]$, at most $C_d$ wavelets $\psi_{j,k}$ are nonzero, we deduce from Cauchy-Schwarz inequality that
\[\int_0^1 \left(\sum_{k\in\mathcal{M}_{j}} \psi_{j,k}(x_0)\psi_{j,k}(y) \right)^2 dy \leq C_d2^j.\]
Next we deal with the second term in \eqref{eq:pointwise:var}. Using that $\Tilde{\psi}_{j,k} = \sum_l \langle \Tilde{\psi}_{j,k}, \Tilde{\psi}_{j,l} \rangle  \psi_{j,l}$ (see below Theorem \ref{th: frame operator}) and $\mathbf{T}_{j,x_0}=(\dots,\psi_{j,k}(x_0),\dots)^T$ (with index $k$ running over $\mathcal{M}_j$),
\begin{align*}
\sigma_{\alpha,j}^2 \sum_{k\in\mathcal{M}_{j}}\Tilde{\psi}_{j,k}(x_0)^2 = \sigma_{\alpha,j}^2(G_{\tilde{\psi},j}\mathbf{T}_{j,x_0})^TG_{\tilde{\psi},j}\mathbf{T}_{j,x_0}= \sigma_{\alpha,j}^2\mathbf{T}_{j,x_0}^TG_{\tilde{\psi},j}^2\mathbf{T}_{j,x_0},
\end{align*}
where $G_{\tilde{\psi},j}$ is the Gram matrix already used above. As above in \eqref{eq: get rid of tilde}, we can bound $\mathbf{T}_{j,x_0}^TG_{\tilde{\psi},j}^2\mathbf{T}_{j,x_0}$ by $A_d^{-2}\mathbf{T}_{j,x_0}^TG_{\psi,j}^2\mathbf{T}_{j,x_0}$ for $G_{\psi,j}$ the Gram matrix of the wavelets in $\mathcal{U}_j$ (instead of the dual wavelets). As the wavelets have compact support, this second Gram matrix is band-limited, with at most $C_d$ nonzero elements around the diagonal, and its non-zero elements are bounded in absolute value by $1$ by Cauchy-Schwarz inequality. Therefore, for any nonzero eigenvector $\mathbf{z}$ of $G_{\psi,j}$ associated to eigenvalue $\lambda$ and taking the index $k$ of its largest element in absolute value, the following inequality stands
\[ |\lambda| |\mathbf{z}_k| \leq \sum_l \left|(G_{\psi,j})_{k,l}\right| |\mathbf{z}_l| \leq C_d |\mathbf{z}_k|. \]
As a consequence, this matrix has its eigenvalues bounded in absolute value by $C_d$ and we can conclude that
\[\norm{G_{\tilde{\psi},j}\mathbf{T}_{j,x_0}}^2_{\ell^2} \leq A_d^{-2} C_d^2 \norm{\mathbf{T}_{j,x_0}}^2_{\ell^2}.\]
Using again that for any $x_0\in[0,1]$ at most $C_d$ wavelets $\psi_{j,k}$ are nonzero and these are bounded by $C_d 2^{j/2}$, this implies 
\[\sigma_{\alpha,j}^2 \sum_{k\in\mathcal{M}_{j}}\Tilde{\psi}_{j,k}(x_0)^2 \leq A_d^{-2} C_d^2 \sigma_{\alpha,j}^2 \sum_{k\in\mathcal{M}_{j}}\psi_{j,k}(x_0)^2 \leq C_d \sigma_{\alpha,j}^2 2^j.\]

Next we deal with the pointwise variance of the $q$--th derivative of the estimator with $q\leq d-1$, which follows a similar decomposition as in \eqref{eq:pointwise:var}. Let's denote $I_{j}(x_0)\subset \mathcal{M}_j$ the set of indices such that $\psi_{j,k}^{(q)}(x_0)\neq 0$ for $k\in I_{j}(x_0)$, which is such that $\left|I_{j}(x_0)\right|\leq C_d$ since at most $C_d$ wavelets have their support intersecting with $\{x_0\}$. From \eqref{eq: sym kernel splwav}, $\sum_{k\in\mathcal{M}_{j}} \psi_{j,k}(X_1)\Tilde{\psi}_{j,k}^{(q)}(x_0)=\sum_{k\in\mathcal{M}_{j}} \Tilde{\psi}_{j,k}(X_1)  \psi_{j,k}^{(q)}(x_0)$ and, following the same arguments as above,
\begin{align*}
    \mathbb{E}_f \left(  \sum_{j=j_0-1}^{j_n}\sum_{k\in\mathcal{M}_{j}} \psi_{j,k}(X_1)\Tilde{\psi}_{j,k}^{(q)}(x_0) \right)^2& \leq  \norm{f}_{L^\infty} \sum_{j=j_0-1}^{j_n}\int_0^1 \left(\sum_{k\in\mathcal{M}_{j}} \Tilde{\psi}_{j,k}(y)\psi_{j,k}^{(q)}(x_0) \right)^2 dy\\
    &=\norm{f}_{L^\infty} \sum_{j=j_0-1}^{j_n}\int_0^1 \left(\sum_{k\in I_j(x_0)} \Tilde{\psi}_{j,k}(y)\psi_{j,k}^{(q)}(x_0) \right)^2 dy\\
    &\leq  \norm{f}_{L^\infty} \sum_{j=j_0-1}^{j_n} \left[\sum_{k\in I_j(x_0)}\psi_{j,k}^{(q)}(x_0)^2\right]\int_0^1 \left(\sum_{k\in I_j(x_0)} \Tilde{\psi}_{j,k}(y)^2 \right) dy.
\end{align*}
Since $\norm{\psi_{j,k}^{(q)}}_{L^\infty}\leq C_{d} 2^{(1+2q)j/2}$ and $\norm{\Tilde{\psi}_{j,k}}_{L^2}=\norm{S_j^{-1}\psi_{j,k}}_{L^2}\leq A_d^{-1}\norm{\psi_{j,k}}_{L^2}=A_d^{-1}$, the above is bounded by
\[\norm{f}_{L^\infty}\sum_{j=j_0-1}^{j_n}  C_{d} 2^{(1+2q)j}= C_d \norm{f}_{L^\infty}2^{(1+2q)j_n}.\]
Finally, from \eqref{eq: derivatives}, we know that any element of $\mathcal{V}_{j_n+1}$ is $(d-1)$--differentiable. Therefore, for $q\leq d-1$, we can introduce $\mathbf{T}_{j,x_0}^{(q)}=(\dots,\psi_{j,k}^{(q)}(x_0),\dots)^T$ (with the index $k$ of the vector running over $\mathcal{M}_j$) and generalize our previous notations. By linearity, we also have
\[\sum_{k\in\mathcal{M}_{j}}\Tilde{\psi}_{j,k}^{(q)}(x_0)^{2}=\norm{G_{\tilde{\psi},j}\mathbf{T}_{x_0}^{(q)}}^2_{\ell^2}\leq A_d^{-2} C_d^2 \norm{\mathbf{T}_{x_0}^{(q)}}^2_{\ell^2}.\]
Using the fact that the wavelets have compact support and hence at most $C_d$ wavelet derivatives $\psi_{j,k}^{(q)}$ are nonzero at any point $x_0$ (since it is the case for $\psi_{j,k}$), and since these are bounded by $C_{d} 2^{(1+2q)j/2}$ according to \eqref{eq: spline wavelets}, there exists a constant $C_{d}$ such that 
\begin{equation}\label{eq: bound l2 sum adj wav}
    \sum_{k\in\mathcal{M}_{j}}\Tilde{\psi}_{j,k}^{(q)}(x_0)^{2}\leq A_d^{-2} C_d \norm{\mathbf{T}_{j,x_0}^{(q)}}^2_{l^2} \leq C_{d} 2^{(1+2q)j}.
\end{equation}
We can conclude that
\begin{equation}\label{variance stim sp pointwise}
    \mathbb{E}_f\left[\hat{f}_n^{(q)}(x_0)-\mathbb{E}_f\hat{f}_n^{(q)}(x_0)\right]^2 \leq C_{d,M}n^{-1}(2^{j_n(1+2q)} + 2^{2j_n(1+q)}j_n^{2a}\alpha^{-2}).
\end{equation}

In the end, we obtain by combining \eqref{eq: supnorm bias proj}, \eqref{variance stim sp}, \eqref{variance stim sp derivatives} and \eqref{variance stim sp pointwise} that
\begin{align}\label{eq: L2 risk deriv}
&\left(\underset{0\leq x_0\leq 1}{\sup}\mathbb{E}_f|\hat{f}_n^{(q)}(x_0) - f^{(q)}(x_0)|^2\right)\,\vee\, \left( \mathbb{E}_f\norm{\hat{f}_n^{(q)} - f^{(q)}}_{L^2(G)}^2\right)  \nonumber\\ &\qquad\qquad\qquad\qquad\qquad \leq C_{d,M}\left[ 2^{-2j_n(p-q)} + n^{-1}(\alpha^{-2}j_n^{2a} 2^{j_n(2+2q)} + 2^{j_n(1+2q)})\right].
\end{align}
For $j_n$ given in the theorem, we get
\begin{align*}
&\left(\underset{0\leq x_0\leq 1}{\sup}\mathbb{E}_f|\hat{f}_n^{(q)}(x_0) - f^{(q)}(x_0)|^2\right)\,\vee\, \left( \mathbb{E}_f\norm{\hat{f}_n^{(q)} - f^{(q)}}_{L^2(G)}^2\right)\nonumber\\&\qquad\qquad\qquad\qquad\qquad\qquad\qquad   \leq C_{d,M} \left[(n\alpha^2 \log^{-2a} n)^{\frac{-2(p-q)}{2p+2}}\vee n^{\frac{-2(p-q)}{2p+1}}\right].
\end{align*} 

\textit{Semiparametric rates for $\Lambda(\hat{f}_n)$}: We conclude as at the end of the proof of Theorem \ref{th: rate atomic}, and obtain 
\[
\underset{f\in\mathcal{W}_p}{\sup}\E_f|\Lambda(f)-\Lambda(\hat{f}_n)|\leq C_{d,M} \left[(n\alpha^2 \log^{-2a} n)^{-\frac{(p-s)}{2p+2}}\vee n^{-\frac{(p-s)}{2p+1}}\right].
\]
\hfill\qed

\subsection{Proof of Theorem~\ref{th: rate smooth}}

Let's first focus on $T_f(\hat{f}_n-f)=\int_{[0,1]} (\hat{f}_n-f)(t)\ \omega_f(t)\ dt$. Recall that $P_{\mathcal{V}_{j_n+1}}f$ denotes the orthogonal projection in $L^2([0,1])$ of $f$ onto $\mathcal{V}_{j_n+1}$ and write 
\[f-\hat{f}_n = f - P_{\mathcal{V}_{j_n+1}}f + P_{\mathcal{V}_{j_n+1}}f - \hat{f}_n.\]
Under the assumptions on $\omega_f$ and the lower bound on $2^{j_n}$, inequality \eqref{eq: supnorm bias proj spline} and Cauchy-Schwarz inequality imply that 
\begin{align*}
|T_f(f-P_{\mathcal{V}_{j_n+1}}f)|&\leq \norm{f-P_{\mathcal{V}_{j_n+1}}f}_{L^\infty}\int_0^1 |\omega_f(t)|dt \\
&\leq C_d M' 2^{-p(j_n+1)} \norm{f^{(p)}}_{L^\infty}\\
&\leq C_{d,p,M,M'}\left[n^{-1/2}\vee(n\alpha^2)^{-1/2}\right].
\end{align*}
We now rewrite $T_f\left(P_{\mathcal{V}_{j_n+1}}f - \hat{f}_n\right) = T_1+T_2$ as the sum of two terms:
    \begin{align*}T_1&=\sum_{k\in\mathcal{M}_{j_0-1}} \left(\langle P_{\mathcal{V}_{j_n+1}}f , \mathbf{B}_{k,d,\xi^{(j_0)}}\rangle-n^{-1}\sum_{i=1}^n \mathbf{B}_{k,d,\xi^{(j_0)}}(X_i)\right)T_f(\Tilde{\psi}_{j_0-1,k}) \\
    &\quad\quad+ \sum_{j=j_0}^{j_n}\sum_{k\in\mathcal{M}_{j}} \left( \langle P_{\mathcal{V}_{j_n+1}}f , \psi_{jk}\rangle - n^{-1}\sum_{i=1}^n\psi_{j,k}(X_i)\right) T_f(\Tilde{\psi}_{j,k}),
    \end{align*}
    and 
        \begin{align*}T_2&=\sum_{k\in\mathcal{M}_{j_0-1}} \left[\frac{\sigma_{\alpha,j_0-1}}{n} \sum_{i=1}^n Y_{i(j_0-1)k}\right]T_f(\Tilde{\psi}_{j_0-1,k}) + \sum_{j=j_0}^{j_n}\sum_{k\in\mathcal{M}_{j}} \left[\frac{\sigma_{\alpha,j}}{n} \sum_{i=1}^n Y_{ijk}\right] T_f(\Tilde{\psi}_{j,k}).
    \end{align*}
The term $T_1$ is the average of $n$ i.i.d. centered random variables whose variance $V_{1,j_n}$ can be bounded as follows.
     \begin{align*}
         V_{1,j_n} &\leq \int_0^1 f(x)\left[\sum_{j=j_0-1}^{j_n}\sum_{k\in\mathcal{M}_{j}} \psi_{j,k}(x)T_f(\Tilde{\psi}_{j,k})\right]^2dx \\
        &= \int_0^1 f(x)\left[\sum_{j=j_0-1}^{j_n}\sum_{k\in\mathcal{M}_{j}} \langle \omega_f, \Tilde{\psi}_{j,k} \rangle \psi_{j,k}(x) \right]^2dx\\
        &= \int_0^1 f(x)\ (P_{\mathcal{V}_{j_n+1}}\omega_f)^2(x)\ dx.
     \end{align*}

Using \eqref{eq: L2 bias proj spline} where we replace $f$ with $\omega_f$ since $\omega_f\in\mathcal{B}_{(\infty,1)}(M^\prime)$, we can ensure that for $j_n$ large enough $\norm{P_{\mathcal{V}_{j_n+1}}\omega_f}_{L^2}\leq 2\norm{\omega_f}_{L^2}$. Combined with the previous bound on $V_{1,j_n}$, we obtain $V_{1,j_n}\leq 4M\norm{\omega_f}^2_{L^2}$ since $f$ is bounded by $M$. This implies 
\[E|T_1|\leq \sqrt{Var(T_1)} \leq n^{-1/2} V_{1,j_n}^{1/2}\leq C_M n^{-1/2} .\]

Turning to $T_2$, it is also the average of $n$ i.i.d. centered random variables with variance $V_{2,j_n}$. For $P_{\mathcal{U}_j}$ the orthgonal projector on $\mathcal{U}_j$ and with the slight abuse of notation $P_{\mathcal{U}_{j_0-1}} = P_{\mathcal{V}_{j_0}}$, we have for $j\geq j_0-1$, using that the $\Tilde{\psi}_{j,k}$ form a frame with upper frame bound $A_d^{-1}$ and the $\psi_{j,k}$ a frame with lower frame bound $A_d$, 
\[\sum_{k\in\mathcal{M}_{j}} \left|\langle \omega_f, \Tilde{\psi}_{j,k}\rangle\right|^2 \leq A_d^{-1}\norm{P_{\mathcal{U}_j}\omega_f}_{L^2}^2\leq A_d^{-2}\sum_{k\in\mathcal{M}_{j}} \left|\langle \omega_f, \psi_{j,k}\rangle\right|^2. \]Also, $\left|\langle \omega_f, \psi_{j,k}\rangle\right|\leq C_d M 2^{-3j/2}$ by \eqref{eq: decay spline wavelet coeff}.
By independence of the Laplace noise, we obtain

     \begin{align*}
         V_{2,j_n} &=  2\sum_{j=j_0-1}^{j_n} \sigma_{\alpha,j}^2\sum_{k\in\mathcal{M}_{j}} T_f(\Tilde{\psi}_{jk})^2 \\
        &= 2\sum_{j=j_0-1}^{j_n} \sigma_{\alpha,j}^2\sum_{k\in\mathcal{M}_{j}} \left|\langle \omega_f, \Tilde{\psi}_{j,k}\rangle\right|^2\\
        &\leq 2A_d^{-2}\sum_{j=j_0-1}^{j_n} \sigma_{\alpha,j}^2\sum_{k\in\mathcal{M}_{j}} \left|\langle \omega_f, \psi_{j,k}\rangle\right|^2 \\
        &\leq \alpha^{-2}C_d + \alpha^{-2}C_dM^2 \left(\frac{a}{a-1}\right)^2\sum_{j=j_0}^{j_n} j^{2a} 2^{-2j}\\
        &= C_{d,a,M}\alpha^{-2},
     \end{align*}
where we used that $j_0$ depends on $p$ only. As a consequence, we now have
\[\E|T_2|\leq \sqrt{Var(T_2)} \leq n^{-1/2} V_{2,j_n}^{1/2}\leq C_{d,a,M}(n\alpha^2)^{-1/2}.\]
In conclusion,
\[\E\left|T_f(\hat{f}_n-f)|\right|\leq C_{d,a,M}\left[n^{-1/2}\vee(n\alpha^2)^{-1/2}\right].\]

Then, combining \eqref{eq: L2 risk deriv} applied to $q=m$ with the lower and upper bounds on $2^{j_n}$ and the condition on $p$ shows that the remainder term in \eqref{eq: form functional} is of order $O(n^{-1/2}\vee(n\alpha^2)^{-1/2})$. Indeed, the lower bound and the fact that $p\geq 2m$ gives
\[2^{-2j_n(p-m)} \lesssim n^{-(p-m)/p}\vee(n\alpha^2)^{-(p-m)/p} \leq n^{-1/2}\vee(n\alpha^2)^{-1/2}.\]
In addition, using $2^{j_n}\leq \log^{-a/(m+1)}(n\alpha^2)(n\alpha^2)^{1/(4m+4)}$, we obtain
\[(n\alpha^{2})^{-1} j_n^{2a}2^{j_n(2+2m)}\leq C_p (n\alpha^2)^{-1/2},\]
and using $2^{j_n}\leq n^{1/(4m+3)}$ if $m>0$ or $2^{j_n}\leq n^{1/4}$ if $m=0$, we get in both cases
\[n^{-1} 2^{j_n(1+2m)}= o\left(n^{-1/2}\right).\]


We can now conclude, using \eqref{eq: Taylor expansion Expectation} and Lemma \ref{lem: exponential decay outer taylor exp}, that for $n$ lage enough,
\[
\E|\Lambda(f)-\Lambda(\hat{f}_n)|\leq C_{d,a,M}\left[n^{-1/2}\vee(n\alpha^2)^{-1/2}\right].
\]
\hfill\qed

\subsection{Additional lemmas}

\begin{lemma}\label{lem: bound sup by l2}
For any integers $m<p$ and function $f\in C^p[0,1)$, it holds that \[\norm{f}_{(\infty,m)}\leq 2\norm{f}_{(2,m+1)}.\]
\end{lemma}
\begin{proof}
    Fix $j< p$. By the mean value theorem, there exists $c\in[0,1]$ such that $f^{(j)}(c) = \int_0^1 f^{(j)}(t) dt\leq \sqrt{\int_0^1 f^{(j)}(t)^2 dt}$ (by Cauchy-Schwarz).
    
    Then, by the fundamental theorem of calculus, for any $x\in[0,1]$, we have that
\[f^{(j)}(x)=f^{(j)}(c) + \int_c^x f^{(j+1)}(t) dt \leq \norm{f^{(j)}}_{L^2} + \norm{f^{(j+1)}}_{L^2},\]
the last inequality coming from a second application of Cauchy-Schwarz inequality.
Finally, this proves that
\begin{align*}
\norm{f}_{(\infty,m)}&=\sum_{i=0}^m \norm{f^{(i)}}_{L^\infty}\leq \sum_{i=0}^m \left[\norm{f^{(i)}}_{L^2} + \norm{f^{(i+1)}}_{L^2}\right] \\
&\leq 2\sum_{i=0}^{m+1} \norm{f^{(i)}}_{L^2}=2\norm{f}_{(2,m+1)}.
\end{align*}
\end{proof}

\begin{lemma}\label{lemma: bound norm deriv splines}
    For $j\geq 0$ and any $g\in \mathcal{V}_j$ for $\mathcal{V}_j$ as in \eqref{eq: equi spline space} with $d\geq p+1$, there exists a constants $C_{d,p},\bar{C}_{d,p}$ depending on $p$ and $d$ only such that
    \[\norm{g}_{(\infty,p)}\leq C_{d,p} 2^{jp} \norm{g}_{L^\infty}\quad\text{and}\quad \norm{g}_{(2,p)}\leq \bar{C}_{d,p} 2^{jp} \norm{g}_{L^2}.\]
\end{lemma}
\begin{proof}
    This follows from Theorem \ref{th: from L2 to l2}, bounding the supnorm of spline functions by their coefficient in the B-spline basis, and \eqref{eq: derivatives} relating the B-spline coefficients of the different derivatives.
\end{proof}

\begin{lemma}\label{lem: exponential decay outer taylor exp}
For $m\geq 0$ and $1\leq p\leq d-1$.  Let $\hat{f}_n$ be either of
\begin{enumerate}
    \item the estimator \eqref{eq: estimator}, satisfying the assumptions of Theorem \ref{th: rate atomic};
    \item the estimator \eqref{eq: plug_in wavelet estimator}, satisfying the assumptions of Theorem \ref{th: atomic wav}, or of Theorem \ref{th: rate smooth} with $m>0$.
\end{enumerate}
Then, for any $f\in\mathcal{W}_p$, there exists $c>0$ such that
\[\sqrt{\mathbb{E}\left[\norm{f-\hat{f}_n}_{(\infty,p)}^2\right]}\, \mathbb{P}\left[\norm{f-\hat{f}_n}_{(\infty,m)}> \delta\right]^{1/2}= o(e^{-n^{c}}).\]
Also, let $\hat{f}_n$ be either of
\begin{enumerate}
    \item the estimator \eqref{eq: plug_in wavelet estimator}, satisfying the assumptions of Theorem \ref{th: rate smooth} with $m=0$;
    \item the estimator \eqref{def:Lepski}, satisfying the assumptions of Theorem \ref{th: adaptive rate semiparam splwav}.
\end{enumerate}
Then, for any $f\in\mathcal{W}_p$ and $c>0$,
\[\sqrt{\mathbb{E}\left[\norm{f-\hat{f}_n}_{(\infty,p)}^2\right]}\, \mathbb{P}\left[\norm{f-\hat{f}_n}_{(\infty,m)}> \delta\right]^{1/2}= o(n^{-c}).\]
\end{lemma}
\begin{proof}
We first focus on the non-adapative estimators \eqref{eq: estimator} and \eqref{eq: plug_in wavelet estimator}, before considering the adaptive one.

We note that, under the assumptions of either of Theorem \eqref{th: rate atomic}, \eqref{eq: plug_in wavelet estimator} or \eqref{th: atomic wav}, both estimators satisfy $2^{j_n}=o\left(n\right)$, since we assume $m\geq0$ and $p\geq 0$.
Since $f\in\mathcal{W}_p$, 
\[\sqrt{\mathbb{E}\left[\norm{f-\hat{f}_n}_{(\infty,p)}^2\right]}\leq \sqrt{2}\sqrt{\mathbb{E}\left[\norm{\hat{f}_n}_{(\infty,p)}^2\right]+\mathbb{E}\left[\norm{f}_{(\infty,p)}^2\right]}\leq \sqrt{2}\sqrt{\mathbb{E}\left[\norm{\hat{f}_n}_{(\infty,p)}^2\right]+M}.\]
In addition, as $\hat{f}_n\in \mathcal{V}_{j_n}$ for the choice \eqref{eq: estimator} and $\hat{f}_n\in \mathcal{V}_{j_n+1}$ for the choice \eqref{eq: plug_in wavelet estimator}, we can apply Lemma \ref{lemma: bound norm deriv splines} as $d\geq p+1$ to bound
\begin{equation}\label{eq: crude bound on strong distance}
\sqrt{\mathbb{E}\left[\norm{f-\hat{f}_n}_{(\infty,p)}^2\right]}\leq \sqrt{2}\sqrt{C_{d,p}2^{2j_np}\left(2M^2+2\mathbb{E}\left[\norm{\hat{f}_n-f}_{L^\infty}^2\right]\right)+M}.
\end{equation}
From Lemma \ref{lem: bound sup by l2}, $\mathbb{E}\left[\norm{\hat{f}_n-f}_{L^\infty}^2\right]=\mathbb{E}\left[\norm{\hat{f}_n-f}_{(\infty,0)}^2\right]\leq 4 \mathbb{E}\left[\norm{\hat{f}_n-f}_{(2,1)}^2\right]$. Then, as we assume $1\leq d-1$, from \eqref{eq: derivative loss} and \eqref{eq: L2 risk deriv} with $q=1$, we obtain for both choices that $\mathbb{E}\left[\norm{\hat{f}_n-f}_{(2,1)}^2\right]\leq o\left(n^4\right)$, so that \eqref{eq: crude bound on strong distance} now implies $\sqrt{\mathbb{E}\left[\norm{f-\hat{f}_n}_{(\infty,p)}^2\right]}=o(n^{2+p})$.

We next bound the probability factor. We first note from Lemma \ref{lem: bound sup by l2} that
\[\mathbb{P}\left[\norm{f-\hat{f}_n}_{(\infty,m)}> \delta\right]\leq \mathbb{P}\left[\norm{f-\hat{f}_n}_{(2,m+1)}> \delta/2\right]\]
As seen in \eqref{eq: estimator} and \eqref{eq: plug_in wavelet estimator}, the estimators we consider are of the form
\[\hat{f}_n(\cdot) = \left(\frac{1}{n}\sum_{i=1}^n [\mathbf{e}(X_i)+\mathbf{Z}_i]\right)^T \mathbf{f}(\cdot)\]
with the elements of $\mathbf{f}(\cdot)=(f_k(\cdot))_{k=1,\dots,2^{j_n}+d}$ being a frame of $\mathcal{V}_j$, $(e_{k})_{k=1,\dots,2^{j_n}+d}$ being the dual frame with $\mathbf{e}(X_i)=(e_k(X_i))_{k=1,\dots,2^{j_n}+d}$ and the entries of the vectors $\mathbf{Z}_i$ being proportional to independent Laplace random variables with parameter $1$. Indeed, recall from Section \ref{sec: Riesz bases of splines} that the B-spline basis normalized B-splines $\mathbf{B}_{k,d,\xi^{(j)}}$ is a frame of $\mathcal{V}_j$ and spline wavelets $\psi_{j,i}$ are built to form a frame.
Therefore, all the estimators we consider can be expressed as a sum \[\hat{f}_{1,X_1,\dots,X_n}+\hat{f}_2\coloneqq \left(\frac{1}{n}\sum_{i=1}^n \mathbf{e}(X_i)\right)^T \mathbf{f} + \left(\frac{1}{n}\sum_{i=1}^n \mathbf{Z}_i\right)^T \mathbf{f}\] with the first term depending on observations $X_1,\dots,X_n$ and the second term depending on independent Laplace random variables only. We separate the two terms
\[\mathbb{P}\left[\norm{f-\hat{f}_n}_{(2,m+1)}> \delta/2\right] \leq \mathbb{P}\left[\norm{f-\hat{f}_{1,X_1,\dots,X_n}}_{(2,m+1)}> \delta/4\right]+\mathbb{P}\left[\norm{\hat{f}_2}_{(2,m+1)}> \delta/4\right].\]
An inspection of the proofs of Theorems \ref{th: rate atomic}, \ref{th: atomic wav} and \ref{th: rate smooth}, remarking that we bounded the variance terms from the Laplace noise separately, shows that
\begin{align*}
    \mathbb{E}\left[\norm{\hat{f}_{1,X_1,\dots,X_n}-f}_{(2,m+1)}^2\right]
    &\leq  C_{d,M} \sum_{q=0}^{m+1} \left[ 2^{(2q-2p)j_n}+ n^{-1}2^{(1+2q)j_n}\right],
\end{align*}
where we discard the Laplace variance term from \eqref{eq: derivative loss} and \eqref{eq: L2 risk deriv} in particular. Since $m+1<p$, we obtain from Jensen's inequality that
\[\mathbb{E}\left[\norm{\hat{f}_{1,X_1,\dots,X_n}-f}_{(2,m+1)}\right]^2\leq\mathbb{E}\left[\norm{\hat{f}_{1,X_1,\dots,X_n}-f}_{(2,m+1)}^2\right]\leq O\left(2^{-j_n}+ n^{-1}2^{(2m+3)j_n}\right)\to 0\]
as soon as $j_n\to \infty$ and $2^{j_n}=o\left( n^{1/(2m+3)}\right)$ which is satisfied by assumption. Therefore, for $n$ large enough,
\begin{align*}
&\mathbb{P}\left[\norm{f-\hat{f}_{1,X_1,\dots,X_n}}_{(2,m+1)}> \delta/4\right] \\
&\qquad\qquad\qquad\leq \mathbb{P}\left[\left|\norm{f-\hat{f}_{1,X_1,\dots,X_n}}_{(2,m+1)}-\mathbb{E}\norm{f-\hat{f}_{1,X_1,\dots,X_n}}_{(2,m+1)}\right|> \delta/8\right].
\end{align*}
We now reproduce the argument from Remark 5.1.14 of \cite{Gine_Nickl_2015}, and prove that $\norm{f-\hat{f}_{1,X_1,\dots,X_n}}_{(2,m+1)}$ is a random variable depending on the observations $X_1,\dots,X_n$ with bounded differences (as defined in Definition 3.3.12 of \cite{Gine_Nickl_2015}). 
For any $1\leq i\leq n$ and $X_i'\in[0,1]$, let 
\[f_1=f_{1,X_1,\dots,X_i,\dots,X_n},\qquad \hat{f}_1'=f_{1,X_1,\dots,X_i',\dots,X_n},\] and $\hat{g}_1=f_{1,X_i}$ and  $\hat{g}_1'=f_{1,X_i'}$ be the estimators based on $X_i,X_i'$ when $n=1$, so that $\hat{f}_1-\hat{f}_1'=n^{-1}\left(\hat{g}_1-\hat{g}_1'\right)$. Then, by definition of our projection estimators, using Lemmas \ref{lemma: bound norm deriv splines} as $d\geq p+1\geq m+1$,
\begin{align*}
\left|\norm{f-\hat{f}_1}_{(2,m+1)}-\norm{f-\hat{f}_1'}_{(2,m+1)}\right|&\leq \norm{\hat{f}_1-\hat{f}_1'}_{(2,m+1)}\\
&= \frac{1}{n}\norm{\hat{g}_1-\hat{g}_1'}_{(2,m+1)}\\
&\leq \frac{2\bar{C}_{d,m+1}}{n} 2^{j_n(m+1)}\norm{\hat{g}_1-\hat{g}_1'}_{L^2}.
\end{align*}
Then, since $(\hat{g}_1-\hat{g}_1')(\cdot)=\left(\mathbf{e}(X_i')-\mathbf{e}(X_i)\right)^T \mathbf{f}(\cdot)$ and the entries of $\mathbf{f}$ form a frame, we use \eqref{eq: frame bounds def} to obtain $\norm{\hat{g}_1-\hat{g}_1'}_{L^2}\leq C_d \norm{\mathbf{e}(X_i')-\mathbf{e}(X_i)}_{\ell^2}$. Then, bounding the entries of the vectors as in the proofs of Propositions \ref{th: privacy} and \ref{prop:PrivacyWavelets}, so that $\norm{\mathbf{e}(X_i')-\mathbf{e}(X_i)}_{\ell^\infty}\leq C2^{j_n/2}$, it follows that $\norm{\hat{g}_1-\hat{g}_1'}_{L^2}\leq C_d 2^{j_n}$ and
\[\left|\norm{f-\hat{f}_1}_{(2,m+1)}-\norm{f-\hat{f}_1'}_{(2,m+1)}\right|\leq C 2^{j_n(m+2)}/n.\]
Therefore, in any case except if $m=0$ in Theorem \ref{th: rate smooth}, Theorem 3.3.14 in \cite{Gine_Nickl_2015} gives
\begin{align*}
\mathbb{P}\left[\left|\norm{f-\hat{f}_{1,X_1,\dots,X_n}}_{(2,m+1)}-\mathbb{E}\norm{f-\hat{f}_{1,X_1,\dots,X_n}}_{(2,m+1)}\right|> \delta/8\right] &\leq 2e^{-C\delta^2\,n\, 2^{-j_n(2m+4)}}\nonumber\\
&\leq 2e^{-C\delta^2\,\min(n^{\frac{2m-1}{4m+3}},n^{\frac{1}{2m+5}})}.
\end{align*}
where we used that $2^{j_n}\leq Cn^{\frac{1}{4m+3}}$ in Theorem \ref{th: rate smooth} and $2^{j_n}\leq Cn^{\frac{1}{2p+1}}\leq Cn^{\frac{1}{2m+5}}$ in Theorem \ref{th: rate atomic} and \ref{th: atomic wav}. So there exists $c>0$ such that
\begin{equation}\label{eq: exp decay prob large empir dist}
    \mathbb{P}\left[\norm{f-\hat{f}_{1,X_1,\dots,X_n}}_{(2,m+1)}> \delta/4\right]=O(e^{-n^c}).
\end{equation}
If $m=0$ in Theorem \ref{th: rate smooth}, the bound we obtain using $2^{j_n}\leq \log^{-a'}(n)n^{1/4}$ is
\begin{align*}
\mathbb{P}\left[\left|\norm{f-\hat{f}_{1,X_1,\dots,X_n}}_{(2,m+1)}-\mathbb{E}\norm{f-\hat{f}_{1,X_1,\dots,X_n}}_{(2,m+1)}\right|> \delta/8\right] &\leq 2e^{-C\delta^2\log^{4a'}n}\\
&\leq 2n^{-C\delta^2\log^{4a'-1}n}.
\end{align*}
As a consequence, for any $c>0$,
\begin{equation}\label{eq: more than polynom decay prob large empir dist}
    \mathbb{P}\left[\norm{f-\hat{f}_{1,X_1,\dots,X_n}}_{(2,m+1)}> \delta/4\right]=O(n^{-c}).
\end{equation}

Again, using Lemmas \ref{lemma: bound norm deriv splines} and \ref{lem: bound sup by l2}, we have that 
\[\mathbb{P}\left[\norm{\hat{f}_2}_{(2,m+1)}> \delta/4\right]\leq \mathbb{P}\left[\norm{\hat{f}_2}_{L^2}^2> C2^{-2j_n(m+1)}\delta^2\right].\]
For the estimator \eqref{eq: estimator}, we obtain the following bound from Theorem \ref{th: from L2 to l2}
\begin{align*}
\mathbb{P}\left[\norm{\hat{f}_2}_{L^2}^2> C2^{-2j_n(m+1)}\delta^2\right]&\leq \mathbb{P}\left[\sum_{k=1}^{2^{j_n}+d}\left(\sum_{i=1}^n Y_{k,i}\right)^2> C_dn^2 \sigma_{\alpha,j_n}^{-2}2^{-2j_n(m+1)}\delta\right]\\
&\leq \sum_{k=1}^{2^{j_n}+d}\mathbb{P}\left[\left|\sum_{i=1}^n Y_{k,i}\right|>C_dn \sigma_{\alpha,j_n}^{-1}2^{-j_n(m+3/2)}\delta\right],
\end{align*}
where $Y_{k,i}$ are independent Laplace random variables with parameter 1 and $\sigma_{\alpha,j_n}=C_d\alpha^{-1}2^{j_n/2}$. From Bernstein's inequality,
\begin{align*}
\mathbb{P}\left[\left|\sum_{i=1}^n Y_{k,i}\right|>C_dn \sigma_{\alpha,j_n}^{-1}2^{-j_n(m+3/2)}\delta\right]&=\mathbb{P}\left[\left|\sum_{i=1}^n Y_{k,i}\right|>C_d\alpha n 2^{-j_n(m+2)}\delta\right]\\
&\leq e^{-C_d\alpha^2 \delta^2 n\, \min\left(2^{-j_n(m+2)},\, 2^{-2j_n(m+2)}\right)}\\
&\leq e^{-C_d\alpha^2 \delta^2 n^{1-\frac{2m+4}{2m+5}}},
\end{align*}
where we used $2^{j_n}\leq C n^{1/{2p+1}}\leq C n^{1/{2m+5}}$ in Theorem \ref{th: rate atomic}, so there exists $c>0$ such that
\begin{align}
    \mathbb{P}\left[\norm{\hat{f}_2}_{(2,m+1)}> \delta/4\right]&\leq 2^{j_n+d}e^{-C_d\alpha^2 \delta^2 n^{-\frac{1}{2m+5}}}\nonumber\\
    &\leq Cn^{\frac{1}{2p+1}}e^{-C\alpha^2 \delta^2 n^{-\frac{1}{2m+5}}}\nonumber\\
    &\leq e^{-n^c}.\label{eq: bound deviation lap Bsp}
\end{align}
Similarly, for the estimator as in \eqref{eq: plug_in wavelet estimator}, using \eqref{eq: frame bounds def} and the frame property of $\psi_{j,i}$,
\begin{align*}
\mathbb{P}\left[\norm{\hat{f}_2}_{L^2}^2> C2^{-2j_n(m+1)}\delta^2\right]&\leq \sum_{j=j_0-1}^{j_n} \sum_{k\in\mathcal{M}_j} \mathbb{P}\left[   \left|\sum_{i=1}^nY_{ijk}\right|>C_dn\sigma_{\alpha,j}^{-1} 2^{-j(m+3/2)}\delta\right]\nonumber\\
&\leq \sum_{j=j_0-1}^{j_n} \sum_{k\in\mathcal{M}_j} \mathbb{P}\left[  \left|\sum_{i=1}^nY_{ijk}\right|>C_d\alpha n j^{-a}2^{-j(m+2)}\delta\right]\nonumber\\
&\leq \sum_{j=j_0-1}^{j_n} \sum_{k\in\mathcal{M}_j} \mathbb{P}\left[  \left|\sum_{i=1}^nY_{ijk}\right|>C_d\alpha n j_n^{-a}2^{-j_n(m+2)}\delta\right]\nonumber\\
&\leq Cn^{\frac{1}{2p+1}}e^{-C_d\alpha^2 \delta^2 \log(n)^{-a} n^{-\frac{1}{2m+5}}}
\end{align*}
where we used $\sigma_{\alpha,j}\leq C\alpha^{-1} j^a 2^{j/2}$. As we assume $p\geq m+2$ in both Theorem \ref{th: atomic wav} and Theorem \ref{th: rate smooth}, there exists $c>0$ such that a bound as in \eqref{eq: bound deviation lap Bsp} is satisfied.

Therefore, combining $\sqrt{\mathbb{E}\left[\norm{f-\hat{f}_n}_{(\infty,p)}^2\right]}=O(n^{2+p})$, \eqref{eq: exp decay prob large empir dist}, \eqref{eq: more than polynom decay prob large empir dist}  and \eqref{eq: bound deviation lap Bsp}, we can conclude that, if $m=0$ in Theorem \ref{th: rate smooth}, for any $c>0$,
\[\sqrt{\mathbb{E}\left[\norm{f-\hat{f}_n}_{(\infty,p)}^2\right]}\mathbb{P}\left[\norm{f-\hat{f}_n}_{(\infty,m)}> \delta\right]^{1/2}= o(n^{-c}).\]
and otherwise, for some $c>0$,
\[\sqrt{\mathbb{E}\left[\norm{f-\hat{f}_n}_{(\infty,p)}^2\right]}\mathbb{P}\left[\norm{f-\hat{f}_n}_{(\infty,m)}> \delta\right]^{1/2}= o(e^{-n^{c}}).\]

Finally, for the estimator \eqref{def:Lepski}, we note that since $2^{\hat{j}_n}\leq 2^{j_{\max}}\leq n^{1/3}=o(n)$ and from Theorem \ref{thm:lepski}, $p\geq m+1$, $\mathbb{E}\left[\norm{\hat{f}_n-f}_{(2,1)}^2\right]=o(1)$. So, from \eqref{eq: crude bound on strong distance}, 
\[\sqrt{\mathbb{E}\left[\norm{f-\hat{f}_n}_{(\infty,p)}^2\right]}=o(n^{p}).\]
Then, from Lemma \eqref{lem:overshoot} with $C_0$ large enough and \eqref{eq:asymp:oracle}, for any $c>0$,
\begin{align*}
\mathbb{P}\left[\norm{f-\hat{f}_n}_{(\infty,m)}> \delta\right]&\leq \mathbb{P}\left[\norm{f-\hat{f}_n}_{(\infty,m)}> \delta,\quad \hat{j}_n\leq j_n^*\right]+\mathbb{P}\left[\hat{j}_n>j_n^*\right]\\
&\leq \mathbb{P}\left[\norm{f-\hat{f}_n}_{(\infty,m)}> \delta,\quad \hat{j}_n\leq j_n^*\right]+o(n^{-c}).
\end{align*}
Then, the first term can be bounded by
\[\mathbb{P}\left[\norm{f-\hat{f}_n}_{(\infty,m)}> \delta,\quad \hat{j}_n\leq j_n^*\right]\leq \mathbb{P}\left[\norm{\hat{f}_n^{j_n^*}-\hat{f}_n^{\hat{j}_n}}_{(\infty,m)}> \delta,\quad \hat{j}_n\leq j_n^*\right]+\mathbb{P}\left[\norm{f-\hat{f}_n^{j_n^*}}_{(\infty,m)}> \delta/2\right],\]
and the second term can be bounded as in \eqref{eq: exp decay prob large empir dist} and \eqref{eq: bound deviation lap Bsp} by $O(e^{-n^c})$ for some $c>0$. Finally, from Lemma \ref{lem: bound sup by l2} and by definition in \eqref{def:hat:jn}, on the event $\left\{\hat{j}_n\leq j_n^*\right\}$, 
\begin{align*}
    \norm{\hat{f}_n^{j_n^*}-\hat{f}_n^{\hat{j}_n}}_{(\infty,m)}^2&\leq 2\norm{\hat{f}_n^{j_n^*}-\hat{f}_n^{\hat{j}_n}}_{(2,m+1)}^2\\
    &\leq \tau n^{-1} 2^{2j_n^*(m+1)}j_n^*\left(2^{j_n^*}+2^{2j_n^*}(j_n^*)^{2a}\alpha^{-2}\right)\\
    &=o(1),
\end{align*}
where the last line follows from \eqref{eq:asymp:oracle} and the assumption $p\geq m+2$. Finally, for $n$ large enough, $\mathbb{P}\left[\norm{\hat{f}_n^{j_n^*}-\hat{f}_n^{\hat{j}_n}}_{(\infty,m)}> \delta,\quad \hat{j}_n\leq j_n^*\right]=0$, and, for $c>0$,
\[\mathbb{P}\left[\norm{f-\hat{f}_n}_{(\infty,m)}> \delta\right]=o(n^{-c}),\]
which concludes the proof.
\end{proof}

\section{Proofs of Section~\ref{sec:Adaptation} on adaptation}
\label{sec:app:adaptation}

\subsection{Proof of Lemma~\ref{lem:overshoot}}

Let us introduce the notation $j^-=j-1\geq j_n^*$ for $j>j_n^*$. Then, by the union bound we have
\begin{align}
\P_f(\hat{j}_n=j)&\leq 
\sum_{s=0}^{\lfloor p \rfloor}\sum_{l= j}^{j_{\max}} \P_f \Big( \|(\hat{f}_n^{j^-} -\hat{f}_n^l)^{(s)}\|_{L^2}^2> \tau  n^{-1}2^{2ls}l (2^{l}+2^{2l}l^{2a}\alpha^{-2}) \Big)\nonumber\\
&\qquad  + \sum_{t=0}^{M_n}\sum_{s=0}^{\lfloor p \rfloor}\sum_{l= j}^{j_{\max}} \P_f \Big( |(\hat{f}_n^{j^-} -\hat{f}_n^l)^{(s)}(x_t)|^2> \tau  n^{-1}2^{2ls}l (2^{l}+2^{2l}\alpha^{-2}l^{2a}) \Big).
\label{eq:lem1:1st}
\end{align}
We show below that for all $s\in\{0,...,\lfloor p \rfloor\}$, $i\in \{0,...,n\}$ and ${l\in\mathcal{J},\, l\geq j}$ we have that
\begin{align}
 \P_f \Big( |(\hat{f}_n^{j^-} -\hat{f}_n^l)^{(s)}(x_t)|^2> \tau  n^{-1}2^{2ls}l (2^{l}+2^{2l}\alpha^{-2}l^{2a}) \Big)\leq 4 e^{-(C_0/2)j},\label{eq:UB:prob:pointwise}\\
 \P_f \Big( \|(\hat{f}_n^{j^-} -\hat{f}_n^l)^{(s)}\|_{L^2}^2> \tau  n^{-1}2^{2ls}l (2^{l}+2^{2l}\alpha^{-2}l^{2a}) \Big)\leq 9(M_n+1) e^{-(C_0/2)j},\label{eq:UB:prob:L2}
\end{align}
where $C_0$ is given in the definition of $\tau$, see \eqref{def:hat:jn}. Then by plugging in the above upper bounds into \eqref{eq:lem1:1st} we get that
\begin{align*}
\P_f(\hat{j}_n=j)&\leq  \sum_{s=0}^{\lfloor p \rfloor} \sum_{l= j}^{j_{\max}} 9(M_n+1) e^{-(C_0/2)j}+\sum_{t=0}^{M_n} \sum_{s=0}^{\lfloor p \rfloor} \sum_{l= j}^{j_{\max}}8e^{-(C_0/2)j}\\
&\leq 2^6 j_{\max}(M_n+1)(\lfloor p \rfloor +1)e^{-(C_0/2)j},
\end{align*}
providing the statement. We are left to prove assertions  \eqref{eq:UB:prob:pointwise} and \eqref{eq:UB:prob:L2}.

We start with the former one. Let us introduce the notation $\beta_{jk}=\int f(x) \psi_{jk}(x)dx$, then for any $x_t$, $t=0,...,M_n$
\begin{align}
\big( (\hat{f}_n^{j^-} -\hat{f}_n^l)^{(s)}(x_t)\big)^2&\leq
3\Big\{ \sum_{\ell=j}^{l}  n^{-1}\sum_{i=1}^n   \Big(\sum_{k\in\mathcal{M}_\ell} \big(\psi_{\ell k}(X_i)-\beta_{\ell k} \big) \Big) \tilde\psi_{\ell k}^{(s)}(x_t)\Big\}^2 \nonumber\\
&\qquad +3\Big\{\sum_{\ell=j}^{l} \sum_{k\in\mathcal{M}_\ell} \beta_{\ell k}  \tilde\psi_{\ell k}^{(s)}(x_t)\Big\}^2\nonumber\\
&\qquad+ 3 \Big\{\sum_{\ell=j}^{l}  n^{-1}\sum_{i=1}^n \sum_{k\in\mathcal{M}_\ell}  \sigma_{\alpha,\ell}Y_{\ell k i} \tilde\psi_{\ell k}^{(s)}(x_t)\Big\}^2.\label{eq:BiasEst_pointwise}
\end{align}
For convenience, let us introduce the notations
\begin{align*}
W_{s,\ell}(x)&:=n^{-1}\sum_{i=1}^n\sum_{k\in\mathcal{M}_\ell} \sigma_{\alpha,\ell}Y_{\ell k i} \tilde\psi_{\ell k}^{(s)}(x),\\
V_{s,\ell}(x)&:=n^{-1}\sum_{i=1}^n V_{s,\ell,i}(x),\quad\text{where}\quad 
V_{s,\ell,i}(x)= \sum_{k\in\mathcal{M}_\ell} \big(\psi_{\ell k}(X_i)-\beta_{\ell k} \big)  \tilde\psi_{\ell k}^{(s)}(x).
\end{align*}
Then, by union bound, one can observe that \eqref{eq:UB:prob:pointwise} is implied by the following three inequalities, for arbitrary $x\in[0,1]$
\begin{align}
&\P_f\Big[ \Big(\sum_{\ell=j}^{l} W_{s,\ell}(x)\Big)^2\geq 2C_0\frac{c_{a,p,s}^2 2^{2l(s+1)} l^{2a+1}}{n\alpha^2}  \Big]\leq 2 e^{-(C_0/2)j},\label{eq:help1:Laplace_pointwise_adapt}\\
&\Big| \sum_{\ell=j}^{l} \sum_{k\in\mathcal{M}_\ell} \beta_{\ell k}  \tilde\psi_{\ell k}^{(s)}(x)\Big|^2\leq C_0 C_{s,p} n^{-1}2^{2ls}l(2^{l}+2^{2l}\alpha^{-2}l^{2a}),\label{eq:help2:Laplace_pointwise}\\
&\P_f\Big[ \Big(\sum_{\ell=j}^{l} V_{s,\ell}(x)\Big)^2\geq 2C_0\frac{c_{p,\psi,L}^2 2^{l(2s+1)}l}{n}  \Big]\leq  2 e^{-(C_0/2)j},\label{eq:help3:Laplace_pointwise}
\end{align}
where the constants $c_{a,p,s}$, $C_{s,p}$, and $c_{p,\psi,L}$ are specified below, in the respective proofs. Next we prove the three assertions above.\\

\textit{Proof of \eqref{eq:help1:Laplace_pointwise_adapt}.} Note, that the Laplace transform of $n^{-1}\sigma_{\alpha,\ell} Y_{\ell k i} \tilde\psi_{\ell k}^{(s)}(x)$ satisfies
\begin{align*}
\E \Big[ \exp\Big(  \frac{z}{n}  \sigma_{\alpha,\ell}Y_{\ell k i} \tilde\psi_{\ell k}^{(s)}(x)\Big) \Big]
&=\frac{1}{1-\big( \frac{z}{n}\sigma_{\alpha,\ell}   \tilde\psi_{\ell k}^{(s)}(x)\big)^2}\\
&\leq \exp\Big(  2z^2 \big(\frac{\sigma_{\alpha,\ell}}{n}  \tilde\psi_{\ell k}^{(s)}(x)\big)^2 \Big)
\end{align*}
for $|z|<\frac{n}{2\sigma_{\alpha,\ell} |\tilde\psi_{\ell k}^{(s)}(x)|}$. Then, by the independence of the contaminations $Y_{ijk}$, 
\begin{align*}
\E\Big[ \exp\Big\{z  \sum_{\ell=j}^{l}  W_{s,\ell}(x)  \Big\}\Big]
&\leq \Big[ \prod_{\ell=j}^{l}  \prod_{k\in\mathcal{M}_\ell}\exp\Big(2z^2 \big(\frac{\sigma_{\alpha,\ell}}{n}  \tilde\psi_{\ell k}^{(s)}(x)\big)^2 \Big)  \Big]^n\\
&= \exp\Big( \frac{1}{2}\Big[2 c_{a,p} \frac{ \big[ \sum_{\ell=j}^{l} 2^{\ell} \ell^{2a}  \sum_{k\in\mathcal{M}_\ell}\big( \tilde\psi_{\ell k}^{(s)}(x)\big)^2\big]^{1/2}}{\sqrt{n}\alpha}\Big]^2 z^2 \Big),
\end{align*}
for any $z$ satisfying
$$|z|<\min_{\ell\in\{j,...,l\}}\min_{k\in\mathcal{M}_\ell}\frac{n}{2\sigma_{\alpha,\ell} |\tilde\psi_{\ell k}^{(s)}(x)|}.$$
This together with \eqref{eq: bound l2 sum adj wav} implies that $ \sum_{\ell=j}^{l}W_{s,\ell}(x)$ is $SubExp\big( \nu_{s}^2(x),b_{s}(x)\big)$ with
\begin{align}
\nu_{s}(x)&=2 c_{a,d} \sum_{\ell=j}^{l} \frac{ 2^{\ell/2} \ell^a \big[\sum_{k\in\mathcal{M}_\ell}\big( \tilde\psi_{\ell k}^{(s)}(x)\big)^2\big]^{1/2}}{\sqrt{n}\alpha}
\leq  \frac{ c_{a,d,s}2^{l(s+1)} l^a}{\sqrt{n}\alpha},\nonumber\\
b_{s}(x)&=\max_{\ell\in\{j,...,l\}}\max_{k\in\mathcal{M}_\ell}\frac{\sqrt{C_{a,d}} \ell^a 2^{\ell/2}|\tilde\psi_{\ell k}^{(s)}(x)|}{n\alpha}\leq \frac{ c_{a,d,s} 2^{l(s+1)}l^a}{n\alpha},\label{def:nu:b}
\end{align}
for some large enough constant $c_{a,d,s}$ depending on $C_{d,s}$ and $\Tilde{C}_{d,s}$ from \eqref{eq: bound l2 sum adj wav} and \eqref{UB:basis:deriv} and $c_{a,d}$ from \eqref{eq: noise spline private}.

 Hence by Bernstein's inequality
\begin{align}
\P\Big(|\sum_{\ell=j}^{l}  W_{s,\ell}(x)|\geq  \big(\nu_{s}(x)\sqrt{u} \big)  \vee  \big(b_{s}(x)u\big)\Big)\leq 2e^{-u/2}.\label{eq:bernstein}
\end{align}
Then, by the above display with $u=C_0l\leq\sqrt{n}$, for some $C_0\geq 2\log 2$, we get that
\begin{align}
\P\Big[ \Big|\sum_{\ell=j}^{l} W_{s,\ell}(x)\Big|\geq \frac{c_{a,p,s} 2^{l(s+1)} l^{a}}{\sqrt{n}\alpha}\sqrt{C_0l}  \Big]
&\le 2 e^{-(C_0/2)l},\label{eq:UB:W:pointwise}
\end{align}
concluding the proof of  \eqref{eq:help1:Laplace_pointwise_adapt}.\\

\textit{Proof of \eqref{eq:help2:Laplace_pointwise}.} Note that for $j>j_n^*$, in view of \eqref{eq:def:B},
\begin{align*}
\big( (P_{\mathcal{V}_{l}}f- P_{\mathcal{V}_{j^-}}f)^{(s)}(x)\big)^2 &\lesssim \|(f-   P_{\mathcal{V}_{j^{-}}}f)^{(s)}\|_{L^\infty}^2+ \|(f-  P_{\mathcal{V}_l}f)^{(s)}\|_{L^\infty}^2\\
&\lesssim B(j^{-},f,s,p')^2+B(l,f,s,p')^2\\
& \lesssim n^{-1}2^{2ls}l(2^{l}+2^{2l}\alpha^{-2}l^{2a}).
\end{align*}
finishing the proof of  \eqref{eq:help2:Laplace_pointwise}.\\

\textit{Proof of \eqref{eq:help3:Laplace_pointwise}.} Note that by the (nearly) disjoint support of the basis $\psi_{\ell k}$ (only $c_d$ have shared support at any given point) and $\|\psi_{\ell k}\|_{L^\infty}\leq 2^{\ell/2}c_{\psi}$
\begin{align}
|V_{s,\ell,i}(x)|\leq c_dc_{\psi} 2^{\ell/2+1} \max_{k\in\mathcal{M}_{\ell}} | \tilde\psi_{\ell k}^{(s)}(x)|.\label{ub:V_i:sup}
\end{align}
Hence, in view of $e^u\leq 1+u+eu^2$ for $u\leq 1$,
\begin{align*}
\E\Big[\exp\Big( z \sum_{\ell=j}^l V_{s,\ell}(x) \Big) \Big]
&=\Big[ \E \exp\Big(  \frac{z}{n} \sum_{\ell=j}^l V_{s,\ell,1}(x) \Big)\Big]^n\\
&\leq \Big[1+e\frac{z^2}{n^2} \E\big(\sum_{\ell=j}^l  V_{s,\ell,1}(x)\big)^2\Big]^n\\
&\leq \exp\Big\{e\frac{z^2}{n} \E\big(\sum_{\ell=j}^l  V_{s,\ell,1}(x)\big)^2\Big\},
\end{align*}
for $|z|\leq  n/(c_dc_{\psi} \sum_{\ell=j}^l  2^{\ell/2+1} \max_{k\in\mathcal{M}_{\ell}} |  \tilde\psi_{\ell k}^{(s)}(x)|)$. The above inequality together with the disjoint support of $\psi_{\ell k}$ and $\psi_{\ell k'}$ for $|k-k'|>c_d/2$ and the orthogonality of the spaces $W_{\ell}$ and $W_{\ell'}$ for $\ell\neq\ell'$ imply that
\begin{align*}
 \E\big(\sum_{\ell=j}^l  V_{s,\ell,1}(x)\big)^2
&= \int \Big( \sum_{\ell=j}^l   \sum_{k\in\mathcal{M}_\ell} \big(\psi_{\ell k}(y)-\beta_{\ell k} \big)  \tilde\psi_{\ell k}^{(s)}(x)\Big)^2f(y)dy\\
&\leq \|f\|_{L^\infty} \int \Big( \sum_{\ell=j}^l   \sum_{k\in\mathcal{M}_\ell} \psi_{\ell k}(y) \tilde\psi_{\ell k}^{(s)}(x)\Big)^2dy\\
&\leq  c_d\|f\|_{L^\infty}    \sum_{\ell=j}^l   \sum_{k\in\mathcal{M}_\ell}\big(\tilde\psi_{\ell k}^{(s)}(x)\big)^2\int  \psi_{\ell k}^2(y)dy \\
&=  c_d\|f\|_{L^\infty}    \sum_{\ell=j}^l   \sum_{k\in\mathcal{M}_\ell} \big(\tilde\psi_{\ell k}^{(s)}(x)\big)^2
\end{align*}
which in turn implies that
\begin{align*}
\E\Big[\exp\Big( z \sum_{\ell=j}^l V_{s,\ell}(x) \Big) \Big]
\leq \exp\Big\{\frac{z^2}{2} \frac{2ec_d \|f\|_{L^\infty} }{n}\sum_{\ell=j}^l \sum_{k\in\mathcal{M}_\ell} (\tilde\psi_{\ell k}^{(s)}(x))^2 \Big\},
\end{align*}
for $|z|\leq  n/(c_dc_{\psi} \sum_{\ell=j}^l  2^{\ell/2+1} \max_{k\in\mathcal{M}_{\ell}} |  \tilde\psi_{\ell k}^{(s)}(x)|)$. 
Therefore, in view of \eqref{eq: bound l2 sum adj wav}, $\sum_{\ell=j}^l V_{s,\ell}(x)$ is $SubExp\big(\nu_{s}^2(x),b_{s}(x)\big)$ with
\begin{align}
\nu_{s}(x)&= \frac{\sqrt{2ec_d\|f\|_{L^\infty}}  }{\sqrt{n}} \sqrt{\sum_{\ell=j}^l\sum_{k\in\mathcal{M}_\ell} (\tilde\psi_{\ell k}^{(s)}(x))^2}\leq \frac{c_{d,\psi,L}} {\sqrt{n}} 2^{l(s+1/2)} ,\nonumber\\
b_{s}(x)&=c_dc_{\psi} \sum_{\ell=j}^l  2^{\ell/2+1} \max_{k\in\mathcal{M}_{\ell}} |  \tilde\psi_{\ell k}^{(s)}(x)|\leq  \frac{c_{d,\psi,L}} {n}2^{l(s+1)}, \label{def:nu:b2}
\end{align}
for some large enough constant $c_{d,\psi,L}$. 

Therefore, by applying Bernstein's inequality, similarly to \eqref{eq:bernstein}, we get that
\begin{align}
\P\Big(|\sum_{\ell=j}^l V_{s,\ell}(x)| \geq  |\nu_{s}(x)|\sqrt{u}\vee |b_{s}(x)| u\Big)\leq 2e^{-u/2}.\label{eq:bernstein2}
\end{align}
By combining the above displays we have for $u=C_0l$, for $C_0>2\log 2$, in view of $2^{l/2}\sqrt{C_0l}\leq \sqrt{n}$ that
\begin{align}
\P\Big(|\sum_{\ell=j}^l V_{s,\ell}(x)| \geq   \frac{c_{d,\psi,L}}{\sqrt{n}}2^{l(s+1/2)} \sqrt{C_0l} \Big)\leq 2e^{-(C_0/2)l},\label{ub:V_sl}
\end{align}
 finishing the proof of assertion \eqref{eq:help3:Laplace_pointwise}.\\

It remained to deal with the $L^2$ terms, i.e. assertion \eqref{eq:UB:prob:L2}.  We follow similar reasoning to the pointwise case \eqref{eq:UB:prob:pointwise}. The main challenge is the deduction of $L^2$-concentration inequalities from the pointwise one.

First note that
\begin{align}
\|(\hat{f}_n^{j^-} -\hat{f}_n^l)^{(s)}\|_{L^2}^2&\leq
3\Big\|\sum_{\ell=j}^{l} n^{-1}\sum_{i=1}^n   \Big(\sum_{k\in\mathcal{M}_\ell} \psi_{\ell k}(X_i)-\beta_{\ell k}  \Big) \tilde\psi_{\ell k}^{(s)}\Big\|_{L^2}^2 \nonumber\\
&\qquad +3\Big\| \sum_{\ell=j}^{l} \sum_{k\in\mathcal{M}_\ell} \beta_{\ell k} \tilde\psi_{\ell k}^{(s)}\Big\|_{L^2}^2
+ 3 \Big\|\sum_{\ell=j}^{l}  n^{-1}\sum_{i=1}^n \sum_{k\in\mathcal{M}_\ell} \frac{\sigma_\ell}{\alpha}Y_{\ell k i} \tilde\psi_{\ell k}^{(s)}\Big\|_{L^2}^2.\label{eq:BiasEst}
\end{align}
Hence it is sufficient to prove that
\begin{align}
&\P\Big[\| \sum_{\ell=j}^{l} W_{s,\ell}\|_{L^2}^2\geq   \frac{C_0^2 \tilde{c}_{a,d,s} 2^{2l(s+1)} l^{2a+1}}{n\alpha^2}  \Big]\leq 4(M_n+1)  e^{-(C_0/2)j},\label{eq:help1:Laplace_L2}\\
&\| \sum_{\ell=j}^{l} \sum_{k\in\mathcal{M}_\ell} \beta_{\ell k}  \tilde\psi_{\ell k}^{(s)}\|_{L^2}^2\leq c_{s,p} n^{-1}2^{2ls}l(2^{l}+2^{2l}l^{2a}\alpha^{-2}),\label{eq:help2:Laplace_L2}\\
&\P\Big[\| \sum_{\ell=j}^{l} V_{s,\ell}\|_{L^2}^2\geq C_0\frac{\tilde{c}_{d,\psi,L} 2^{l (2s+1)}l}{n} \Big]\leq 2 M_n e^{-(C_0/2)l},\label{eq:help3:Laplace_L2}
\end{align}
where the constants $\tilde{c}_{a,d,s}$, $c_{s,p}$, and $\tilde{c}_{d,\psi,L}$ are defined below, in the respective proofs.

First note that in view of the definition of $j_n^*$, for all $l>j>j_n^*$,
\begin{align*}
\Big\| \sum_{\ell=j}^{l} \sum_{k\in\mathcal{M}_\ell} \beta_{\ell k}  \tilde\psi_{\ell k}^{(s)}\Big\|_{L^2}^2
&\lesssim \|(P_{\mathcal{V}_l}f- f)^{(s)}\|_{L^2}^2+ \|(f-P_{\mathcal{V}_{j^{-}}}f)^{(s)}\|_{L^2}^2\nonumber\\
&\lesssim B(l,f,s,p)^2+B(j^-,f,s,p)^2\\
&\lesssim n^{-1}2^{2ls}l(2^{l}+2^{2l}l^{2a}\alpha^{-2}),
\end{align*}
providing  \eqref{eq:help2:Laplace_L2}.

Next we consider \eqref{eq:help1:Laplace_L2}. First note that with probability one it holds that
\begin{align}
\|  \sum_{\ell=j}^{l} W_{s,\ell}\|_{L^2}^2 
&=
\sum_{t=1}^{M_n}\int_{x_{t-1}}^{x_t}\Big(\sum_{\ell=j}^l W_{s,\ell}(x)\Big)^2dx\nonumber\\
&\leq M_n^{-1}\sum_{t=1}^{M_n}  \Big(\sum_{\ell=j}^{l} W_{s,\ell}(x_t)\Big)^2+ M_n^{-2} \sup_{x\in[0,1]}|  \sum_{\ell=j}^{l} W_{s,\ell}^{(1)}(x)|^2. \label{eq:ub:L2:W}
\end{align}
The first term on the right hand side is controlled by the pointwise upper bounds above, hence it is left to deal with the second term. By noting that the absolute value of a $Lap(1)$ distribution is an exponential with parameter 1, and the sum of exponentials are gamma distributed, we get that
\begin{align*}
\sup_x| \sum_{\ell=j}^{l}W_{s,\ell}^{(1)}(x)|&\leq n^{-1}\sum_{i=1}^n\sum_{\ell=j}^{l}\sum_{k\in\mathcal{M}_\ell} \sigma_{\alpha,\ell}|Y_{\ell k i}| \sup_x| \tilde\psi_{\ell k}^{(s+1)}(x)|\\
&\leq \frac{c_{a,d} }{n \alpha}\sum_{\ell=j}^{l} 2^{\ell(s+2)}\ell^a\sum_{i=1}^n\sum_{k\in\mathcal{M}_\ell}|Y_{\ell k i}|=: \sum_{\ell=j}^{l}\tilde{W}_{sup,\ell},
\end{align*}
where $\tilde{W}_{sup,\ell}\sim Gamma\big(n2^\ell, \alpha n c_{a,d}^{-1}2^{-\ell(s+2)}\ell^{-a}\big)$. Furthermore, note that from the tail bound of the Erlang distribution we have for $Z\sim Gamma(N,\lambda)$, $N\in\mathbb{N}$ and $\ell \geq 6$, that
\begin{align*}
P(Z\geq \ell N/\lambda )\leq \sum_{k=0}^{N}\frac{1}{k!}e^{-\ell N}(\ell N)^k\leq e^ {-\ell N} \ell^{N}\sum_{k=0}^{N} \frac{N^k}{k!}\leq e^{-\ell N +N+\ln(\ell)N}\leq e^{-(\ell/2)N}.
\end{align*}
 Therefore, in view of the previous displays,
\begin{align*}
\P&\Big(  \sup_{x\in[0,1]}| \sum_{\ell=j}^{l} W_{s,\ell}^{(1)}(x)|^2\geq   \sum_{\ell=j}^{l}    C_0^2c_{a,d}^2\alpha^{-2} 2^{2\ell(s+3)}\ell^{2a+2}  \Big)\\
&\qquad\qquad \leq \sum_{\ell=j}^{l}  \P \big( \tilde{W}_{sup,\ell}\geq   C_0c_{a,d} \alpha^{-1}2^{\ell(s+3)}\ell^{a+1}   \big)\leq \sum_{\ell=j}^{l}e^{-(C_0/2)\ell}\leq 2 e^{-(C_0/2)j},
\end{align*}
for $C_0\geq 2$.

Furthermore, note that there exists a constant $\tilde{c}_{a,d}$ such that $ \sum_{\ell=j}^{l}    C_0^2c_{a,d}^2\alpha^{-2} 2^{2\ell(s+3)}\ell^{2a+2}\leq    C_0^2\tilde{c}_{a,d}^2 \alpha^{-2}  2^{2l(s+3)}l^{2a+2}$.
Therefore, by combining the above displays with \eqref{eq:UB:W:pointwise}, we get for $l<j_{\max}=(\log n)/3$ and $M_n\gtrsim n^{4/3}$ that for large enough constant $ \tilde{c}_{a,d,s}$ and $C_0\geq 1$
\begin{align*}
&\P\Big( \|  \sum_{\ell=j}^{l}  W_{s,\ell}\|_{L^2}^2\geq   \frac{C_0^2 \tilde{c}_{a,d,s} 2^{2l(s+1)} l^{2a+1}}{n\alpha^2} \Big)\\
&\quad\leq 
\P\Big( M_n^{-1}\sum_{t=1}^{M_n} | \sum_{\ell=j}^{l}  W_{s,\ell}(x_t)|^2\geq   \frac{C_0 c_{a,d,s}^2 2^{2l(s+1)}l^{2a+1}}{n\alpha^2} \Big)\\
&\qquad\qquad+\P\Big( M_n^{-2}\sup_{x\in[0,1]}| \sum_{\ell=j}^{l}  W_{s,\ell}^{(1)}(x)|^2  \geq   \frac{C_0^2\tilde{c}_{a,d}^2  2^{2l(s+3)}l^{2a+2} }{\alpha^2M_n^2} \Big)\\
&\quad\leq M_n 4e^{-(C_0/2)l}+ 2 e^{-(C_0/2)j}\leq 4 (M_n+1)  e^{-(C_0/2)j},
\end{align*}
where in the second line we used that $M_n^2\geq n^{8/3}\gg n 2^{4l}$
concluding the proof of  \eqref{eq:help1:Laplace_L2}.\\

Finally, we prove \eqref{eq:help3:Laplace_L2}. Note that in view \eqref{ub:V_i:sup} and \eqref{UB:basis:deriv},
\begin{align*}
\sup_x| \sum_{\ell=j}^{l} V_{s,\ell}^{(1)}(x)|&\leq \frac{1}{n} \sum_{i=1}^n\sup_x |\sum_{\ell=j}^{l} V_{s,\ell,i}^{(1)}(x)|\\
&\leq 2c_d c_{\psi} \sum_{\ell=j}^{l}  2^{\ell/2}\max_{k\in \mathcal{M}_\ell} \sup_{x} |\tilde\psi_{\ell k}^{(s+1)}(x)|
\leq 2c_{d,s,\psi}2^{l (s+2)}. 
\end{align*}
Furthermore, similarly to \eqref{eq:ub:L2:W},
\begin{align*}
\|\sum_{\ell=j}^{l} V_{s,\ell}\|_{L^2}^2 \leq M_n^{-1}\sum_{t=1}^{M_n} |\sum_{\ell=j}^{l}V_{s,\ell}(x_t)|^2+ M_n^{-2} \sup_{x\in[0,1)}|\sum_{\ell=j}^{l} V_{s,\ell}^{(1)}(x)|^2.
\end{align*}
By combining the above displays with \eqref{eq:help3:Laplace_pointwise} for the design points $x_1,...,x_{M_n}$, we get for $l\leq j_{\max}=(\log n)/3$ and large enough constant $\tilde{c}_{d,s,\psi}$ that
\begin{align*}
&\P\Big( \|\sum_{\ell=j}^{l} V_{s,\ell}\|_{L^2}^2\geq   C_0\frac{\tilde{c}_{d,s,\psi} 2^{l (2s+1)}l}{n} \Big)\\
&\qquad\leq 
 \sum_{t=1}^{M_n} \P\Big(  \Big|\sum_{\ell=j}^{l} V_{s,\ell}(x_t)\Big|\geq  \frac{c_{d,\psi,L}^2 2^{l (s+1/2)}\sqrt{C_0 l}}{\sqrt{n}}\Big)
\leq M_n 2 e^{-(C_0/2)l},
\end{align*}
where the first inequality holds for $M_n^2\geq n^{8/3} \gg  n2^{3l}$. \hfill\qed

\subsection{Proof of Theorem~\ref{thm:lepski}}

We distinguished the cases $\hat{j}_n\leq j_n^*$ and $\hat{j}_n>j_n^*$, and deal with them separately. We start with the former one and consider the $L^2$ norm first. Then by triangle inequality, in view of the definition of $\hat{j}_n$ in \eqref{def:hat:jn}, the upper bound for the variance in \eqref{eq:UB:var}, the projection error in \eqref{eq:def:B} and the upper bound for the mean squared error at the oracle $j_n^*$ given in \eqref{ub:var:jn*}, we get that

\begin{align*}
\E_f\|(\hat{f}_n^{\hat{j}_n}-f)^{(s)}\|_{L^2} 1_{\hat{j}_n\leq j_n^*}
&\leq   \E_f\|(\hat{f}_n^{\hat{j}_n}-\hat{f}_n^{j_n^*})^{(s)}\|_{L^2}1_{\hat{j}_n\leq j_n^*}+ \E_f\|(\hat{f}_n^{j_n^*}-f)^{(s)}\|_{L^2}   \\
&\lesssim \tau  n^{-1/2}2^{j_n^*s}(2^{j_n^*}+2^{2j_n^*}(j_n^*)^{2a}\alpha^{-2})^{1/2}+\E_f \|(\hat{f}_n^{j_n^*})^{(s)}-\E_f(\hat{f}_n^{j_n^*})^{(s)} \|_{L^2}\\
&\qquad\qquad+\| (f-P_{\mathcal{V}_{j_n^*}} f)^{(s)}\|_{L^2}\\
& \lesssim (n\alpha^2\log^{-(1+2a)} n)^{-\frac{p-s}{2p+2}}\vee (n/\log n)^{-\frac{p-s}{2p+1}}.
\end{align*}

Next we consider the point-wise error. Let us take an arbitrary $x\in[0,1]$ and let us denote the nearest grid point to $x$ by $x_t$. Then by triangle inequality
\begin{align}
\E_f|(\hat{f}_n^{\hat{j}_n}-f)^{(s)}(x)| 1_{\hat{j}_n\leq j_n^*}
&\leq \E_f|(\hat{f}_n^{\hat{j}_n}-f)^{(s)}(x_t)| 1_{\hat{j}_n\leq j_n^*}+|f^{(s)}(x)-f^{(s)}(x_t)|\nonumber\\
&\qquad\quad+\E_f|(\hat{f}_n^{\hat{j}_n})^{(s)}(x_t)-(\hat{f}_n^{\hat{j}_n})^{(s)}(x)| 1_{\hat{j}_n\leq j_n^*}.\label{eq:pointwise}
\end{align}
We deal with the three terms on the right hand side separately. The upper bound of the first term follows the $L^2$ case above
\begin{align*}
 \E_f|(\hat{f}_n^{\hat{j}_n}-f)^{(s)}(x_t)| 1_{\hat{j}_n\leq j_n^*}
&\leq    \E_f|(\hat{f}_n^{\hat{j}_n}-\hat{f}_n^{j_n^*})^{(s)}(x_t)| 1_{\hat{j}_n\leq j_n^*}+ \E_f|(\hat{f}_n^{j_n^*}-f)^{(s)}(x_t)|\\
&\lesssim \tau  n^{-1/2}2^{j_n^*s/2}(2^{j_n^*}+2^{2j_n^*}(j_n^*)^{2a}\alpha^{-2})^{1/2}+\E_f|(\hat{f}_n^{j_n^*})^{(s)}(x_t)-\E (\hat{f}_n^{j_n^*})^{(s)}(x_t)|\\
&\qquad\qquad + |(f-P_{\mathcal{V}_{j^*_n}}f)^{(s)}(x_t)|\\
&\lesssim   (n\alpha^2\log^{-(1+2a)} n)^{-\frac{p-s}{2p+2}}\vee (n/\log n)^{-\frac{p-s}{2p+1}}.
\end{align*}
Then, for the second term, in view of $f\in W^{p}$, for $p\geq s+1$, we get that
\begin{align*}
|f^{(s)}(x)-f^{(s)}(x_t)|\leq |x-x_t| \|f^{(s+1)}\|_{L^\infty}\lesssim M_n^{-1}.
\end{align*}
Finally, for the third term, note that
\begin{align}
\E_f|(\hat{f}_n^{\hat{j}_n})^{(s)}(x_t)-(\hat{f}_n^{\hat{j}_n})^{(s)}(x)| 1_{\hat{j}_n\leq j_n^*}
&\leq \sum_{j\leq j_n^*}\E_f|(\hat{f}_n^{j})^{(s)}(x_t)-(\hat{f}_n^{j})^{(s)}(x)| 1_{\hat{j}_n=j}\nonumber\\
&\leq  |x-x_t| \sum_{j\leq j_n^*}\E_f \sup_{z\in [x,x_t]}| (\hat{f}_n^{j})^{(s+1)}(z)|1_{\hat{j}_n=j}.\label{eq:help:third}
\end{align}
Noting that  $|x-x_t|\leq 1/M_n$ it remained to give an upper bound for the second term on the right-hand side of \eqref{eq:help:third}. Also, for any $j\leq j_n^*$, 
\begin{align*}\sup_{z\in [x,x_t]} | (\hat{f}_n^{j})^{(s+1)}(z)| 
&\leq \sum_{\ell =j_0-1}^{j_n^*}\sum_{k\in\mathcal{M}_\ell} \|\tilde\psi_{\ell k}^{(s+1)}\|_{L^\infty}  |\bar{Z}_{\ell k}|.
\end{align*}
In view of \eqref{UB:basis:deriv}, \eqref{eq:pointwise:var} and \eqref{eq: bound l2 sum adj wav}
\begin{align}
\sum_{j\leq j_n^*}\E_f \sup_{z\in [x,x_t]}| (\hat{f}_n^{j})^{(s+1)}(z)|1_{\hat{j}_n=j}
&\leq \sum_{\ell =j_0-1}^{j_n^*}\sum_{\ell =j_0-1}^{j_n^*}\sum_{k\in\mathcal{M}_\ell}  \|\tilde\psi_{\ell k}^{(s+1)}\|_{L^\infty} \E_f |\bar{Z}_{\ell k}|\nonumber\\
&\leq \sum_{\ell =j_0-1}^{j_n^*}\sum_{\ell =j_0-1}^{j_n^*}\sum_{k\in\mathcal{M}_\ell}  \|\tilde\psi_{\ell k}^{(s+1)}\|_{L^\infty} \sqrt{\mathbb{V}_f \bar{Z}_{\ell k}}\nonumber\\
&\lesssim \sum_{\ell =j_0-1}^{j_n^*}\sum_{\ell =j_0-1}^{j_n^*}\sum_{k\in\mathcal{M}_\ell} \|\tilde\psi_{\ell k}^{(s+1)}\|_{L^\infty}\Big(  \frac{\sigma_{\alpha,\ell}}{\sqrt{n}} +\frac{1}{\sqrt{n}} \Big)\nonumber\\
&\lesssim \sum_{\ell =j_0-1}^{j_n^*}\sum_{\ell =j_0-1}^{j_n^*}  \frac{\sigma_{\alpha,\ell}+1}{\sqrt{n}}2^{\ell/2}  \Big(\sum_{k\in\mathcal{M}_\ell}  \|\tilde\psi_{\ell k}^{(s+1)}\|_{L^\infty}^2\Big)^{1/2}\nonumber\\
&\lesssim \sum_{\ell =j_0-1}^{j_n^*}\sum_{\ell =j_0-1}^{j_n^*}  \frac{\sigma_{\alpha,\ell}+1}{\sqrt{n}} 2^{\ell(s+2)}\nonumber\lesssim j_n^{*1+a} 2^{j_n^*(s+5/2)}n^{-1/2}\\
& \lesssim (\log n)^{c_a} \Big((n\alpha^2)^{\frac{(s-p)+3/2}{2+2p}}+n^{\frac{(s-p)+2}{1+2p}} \Big), \label{eq:UB:sup:deriv}
\end{align}
for some constant $c_a>0$.
 Hence, the right hand side of \eqref{eq:help:third} is bounded from above by $ (\log n)^{c_a} \big((n\alpha^2)^{\frac{(s-p)+3/2}{2+2p}}+n^{\frac{(s-p)+2}{1+2p}}\big) M_n^{-1}$ for $M_n\gtrsim n$ and $p\geq 1$ (recall that we assumed that $p\geq s+1$). Therefore, by substituting the above three upper bounds into \eqref{eq:pointwise}, we arrive at
\begin{align*}
\E_f|(\hat{f}_n^{\hat{j}_n}-f)^{(s)}(x)| 1_{\hat{j}_n\leq j_n^*}\lesssim   (n\alpha^2\log^{-2a} n)^{-\frac{p-s}{2p+2}}\vee n^{-\frac{p-s}{2p+1}}.
\end{align*}

It remained to consider the event $\hat{j}_n>j_n^*$. Then in view of Cauchy-Schwarz inequality, assertions \eqref{eq:UB:var} and  \eqref{eq:asymp:oracle}, Lemma \ref{lem:overshoot} and the definition of $j_n^*$, for large enough choice of $C_0$ (in the definition of $\tau$), we get the following upper bound for the $L^2$-error
\begin{align*}
\E_f\|(\hat{f}_n-f)^{(s)}\|_2 1_{\hat{j}_n> j_n^*}&\leq \sum_{j\in\mathcal{J},\, j>j_n^*} \sqrt{\E_f\|(\hat{f}_n^j-f)^{(s)}\|_{L^2}^2}\, \mathbb{P}(\hat{j}_n=j)^{1/2}\\
&\lesssim  \sum_{j\in\mathcal{J},\, j>j_n^*} n^{-1/2}j^{1/2}2^{js}(2^{j/2}+2^j j^a\alpha^{-1})  \sqrt{M_n j_{\max}}e^{-(C_0/2) j} \\
&\lesssim n^{-1/2} e^{-(C_0/2- (2+1/2)\log 2)j_n^*} 
=o\Big( \frac{1}{n(\alpha\vee1)^2}\Big).
\end{align*}

Finally, we deal with the point-wise case. Note that similarly to \eqref{eq:UB:sup:deriv}
\begin{align*}
\E_f((\hat{f}_n^j- f)^{(s)} (x))^2
&\leq \E_f((\hat{f}_n^j)^{(s)} (x))^2+f^2(x)\\
&\leq \sum_{\ell=j_0-1}^{j}\sum_{k\in\mathcal{M}_\ell}\|\tilde\psi_{\ell k}^{(s)}\|_{L^\infty}\Big(\frac{\sigma_{\alpha,\ell}}{\sqrt{n}}+\frac{1}{\sqrt{n}}\Big)+f^2(x)\\
&\lesssim \sum_{\ell=j_0-1}^{j} 2^{\ell(s+1)}\lesssim 2^{j(s+1)}.
\end{align*}

Then, in view of the definition of $j_n^*$, for large enough choice of $C_0$,
\begin{align*}
\E_f|(\hat{f}_n-f)^{(s)}(x)| 1_{\hat{j}_n> j_n^*}
&\leq  \sum_{j\in\mathcal{J},\, j>j_n^*} \sqrt{\E_f((\hat{f}_n^j- f)^{(s)} (x))^2}\, \mathbb{P}(\hat{j}_n=j)^{1/2}\\
&\lesssim  \sum_{j\in\mathcal{J},\, j>j_n^*}  \sqrt{M_n j_{\max}}2^{j(s+1)/2}e^{-(C_0/4) j}\\
&=o\Big( \frac{1}{n(\alpha\vee1)^2}\Big).
\end{align*}

\hfill\qed


\end{appendix}

\bibliographystyle{imsart-nameyear}
\bibliography{Lit.bib}{}

\end{document}